\documentclass[oneside,11pt]{article}
\usepackage[utf8]{inputenc}
\usepackage{amsmath,amsfonts,amssymb,amsthm}
\usepackage{mathrsfs} 
\usepackage{dsfont}
\usepackage{mathtools}
\usepackage[margin=1in]{geometry} 
\usepackage{graphicx}
\usepackage{float}
\usepackage{tikz-cd}
\usepackage{graphicx}
\usepackage[font=small,labelfont=bf]{caption}
\usepackage[colorinlistoftodos]{todonotes}
\usepackage{titlesec}
\usepackage{wrapfig}

\usepackage{array}

\usepackage{hyperref}
\hypersetup{
    colorlinks,
    citecolor=black,
    filecolor=black,
    linkcolor=black,
    urlcolor=black
}

\usepackage{enumitem}

\newcommand{\R}{\mathbb R}
\newcommand{\C}{\mathbb C}
\renewcommand{\H}{\mathbb H}
\newcommand{\N}{\mathbb N}
\newcommand{\D}{\mathbb D}
\newcommand{\E}{\mathbb E}

\renewcommand{\P}{\mathbb P}

\newcommand{\hcap}{\text{hcap}}
\newcommand{\crad}{\text{crad}}

\renewcommand{\Re}{\text{Re}}
\renewcommand{\Im}{\text{Im}}

\newcommand{\vare}{\varepsilon}

\newtheorem{theorem}{Theorem}

\newtheorem{proposition}{Proposition}
\newtheorem{lemma}{Lemma}
\newtheorem{corollary}{Corollary}
\newtheorem{thmCite}{Theorem}
\newtheorem{lemmaCite}{Lemma}
\newtheorem{propositionCite}{Proposition}

\DeclareMathOperator{\Vol}{Vol}
\DeclareMathOperator{\Tr}{Tr}
\DeclareMathOperator{\zdet}{det_\zeta}

\theoremstyle{definition}

\newtheorem{definition}{Definition}
\theoremstyle{remark}
\newtheorem{remark}{Remark}

\title{The $\rho$-Loewner Energy: \\Large Deviations, Minimizers, and Alternative Descriptions}
\author{Ellen Krusell\\ \textit{KTH Royal Institute of Technology}\\ \texttt{ekrusell@kth.se} }
\date{}
\begin{document}
\sloppy

\maketitle
\abstract{We introduce and study the $\rho$-Loewner energy, a variant of the Loewner energy with a force point on the boundary of the domain. We prove a large deviation principle for SLE$_\kappa(\rho)$, as $\kappa \to 0+$ and $\rho>-2$ is fixed, with the $\rho$-Loewner energy as the rate function in both radial and chordal settings. The unique minimizer of the $\rho$-Loewner energy is the SLE$_0(\rho)$ curve. We show that it exhibits three phases as $\rho$ varies and give a flow-line representation. We also define a whole-plane variant for which we explicitly describe the trace.\par We further obtain alternative formulas for the $\rho$-Loewner energy in the reference point hitting phase, $\rho > -2$. In the radial setting we give an equivalent description in terms of the Dirichlet energy of $\log|h'|$, where $h$ is a conformal map onto the complement of the curve, plus a point contribution from the tip of the curve. In the chordal setting, we derive a similar formula under the assumption that the chord ends in the $\rho$-Loewner energy optimal way. Finally, we express the $\rho$-Loewner energy in terms of $\zeta$-regularized determinants of Laplacians.}

\pagenumbering{arabic}
\section{Introduction and main results}
\subsection{Introduction}\label{section:introduction}
A chord is a simple curve connecting two distinct boundary points of a domain in the complex plane. The chordal Loewner differential equation gives a way of encoding an appropriately parametrized chord in a simply connected domain by a continuous and real-valued function. Consider the following reference setting: Suppose $\gamma$ is a chord from $0$ to $\infty$ in $\H$, the upper half-plane, (parametrized appropriately on $[0,\infty)$). Then, the family $(g_t)_{t\geq 0}$ of conformal maps $g_t:\H\setminus\gamma_t\to\H$, normalized so that $g_t(z)=z+O(z^{-1})$ at $\infty$, where $\gamma_t=\gamma([0,t])$, satisfies the
chordal Loewner equation
\begin{equation}\partial_t g_t(z)=\frac{2}{g_t(z)-W_t},\quad g_0(z)=z,\label{eq:chordalloewner}\end{equation}
with driving function $W_t=g_t(\gamma(t))$. The function $W_t$ encodes $\gamma$ in the sense that one can, given $W_t$, retrieve $(g_t)_{t\geq 0}$ by solving (\ref{eq:chordalloewner}) and then obtain $\gamma$ by
$$\gamma(t)=\lim_{y\to 0+}g_t^{-1}(W_t+iy).$$
In general, solving the chordal Loewner equation for an arbitrary real-valued and continuous function $W_t$, yields a family $(g_t)_{t\geq 0}$ of conformal maps $g_t:\H\setminus K_t\to\H$, where $(K_t)_{\geq 0}$ is a family of continuously growing compact subsets of $\overline \H$.\par
The Loewner equation can be used to describe (candidates for) scaling limits of interfaces in planar critical lattice models. Schramm \cite{S00} observed that such random curves should satisfy two important properties, conformal invariance and a domain Markov property, and that a random curve has these properties if and only if its driving function is of the form $W_t=\sqrt{\kappa}B_t$, where $\kappa\geq 0$ and $B_t$ is a standard one-dimensional Brownian motion. These curves, called chordal Schramm-Loewner evolution, SLE$_\kappa$, have, indeed, been shown to be the scaling limit of interfaces in lattice models for several values of $\kappa$, and play a major role in random conformal geometry \cite{S11a,S11b,S22}.\par
In \cite{W19a}, Wang showed that chordal SLE$_\kappa$ satisfies a large deviation principle (LDP) as $\kappa\to 0+$. The rate function $I^{(\H;0,\infty)}$ is the chordal Loewner energy, which for chords $\gamma$ from $0$ to $\infty$ in $\H$ is defined by
\begin{equation*}I^{(\H;0,\infty)}(\gamma)=\frac{1}{2}\int_0^\infty \dot W_t^2 dt,\end{equation*}
when the driving function $W_t$ of $\gamma$ is absolutely continuous on $[0,T]$, for all $T>0$, and $I^{(\H;0,\infty)}(\gamma)=\infty$ otherwise. Heuristically, this means that
\begin{equation*}\text{``}\P[\text{SLE}_\kappa \text{ stays close to }\gamma]\sim \exp\bigg(-\frac{I^{(\H;0,\infty)}(\gamma)}{\kappa}\bigg),\ \text{ as $\kappa\to 0+$.''}\end{equation*} 
(See also \cite{FS17}.) The chordal Loewner energy has a natural extension, called the loop Loewner energy, to the space of Jordan curves on the Riemann sphere \cite{RW21}. In later work, Wang showed that the loop Loewner energy has unexpected ties to Teichmüller theory: the loop Loewner energy coincides (up to a multiplicative constant) with the universal Liouville action, a Kähler potential on the Weil-Petersson universal Teichmüller space \cite{W19b}. In particular, the class of finite Loewner energy loops coincides with the class of Weil-Petersson quasicircles, which has many different characterizations \cite{B19}. The links between random conformal geometry and Teichmüller theory revealed by Wang's result are not yet well-understood. It was also shown in \cite{W19b} that the Loewner energy of a smooth Jordan curve can be expressed in terms of $\zeta$-regularized determinants of Laplacians.\par
These discoveries sparked interest in large deviations for SLE type processes, as $\kappa\to 0+$, which was studied in the subsequent papers \cite{PW21,G23,AHP24,AP24}. Parallel to the LDP development, several articles have been devoted to investigating the Loewner energy and its properties, see, e.g., \cite{FS17,RW21,VW20,VW24,MW23,SW24,BBVW23,M23}. See also the survey \cite{W22}.\par
In the present paper, we study large deviations on SLE$_\kappa(\rho)$ curves and the corresponding large deviations rate function, which we call the $\rho$-Loewner energy. 
The SLE$_\kappa(\rho)$ curve is a natural generalization of SLE$_\kappa$ where one puts an attractive or repelling force ($\rho\in\R$) at one or several marked points, called force points. In this paper we deal with the case of one force point, which will, with a few exceptions, be located on the boundary of the domain. The chordal SLE$_\kappa(\rho)$ curve in $\H$ starting at $0$, with reference point $\infty$ and force point $z_0\in\overline\H\setminus\{0\}$, is the random curve whose Loewner driving function satisfies the SDE
\begin{equation}d W_t = \Re\frac{\rho}{W_t-z_t}dt+\sqrt{\kappa}dB_t,\quad W_0=0,\label{eq:chordaldrive}\end{equation}
where $z_t=g_t(z_0)$, and $B_t$ is a one-dimensional standard Brownian motion. Using the radial Loewner equation
\begin{equation}\partial_t g_t(z)=g_t(z)\frac{W_t+g_t(z)}{W_t-g_t(z)}\label{eq:radialLoewner} \end{equation} 
one can, in a similar way, define the radial SLE$_\kappa(\rho)$ curve in $\D$, the unit disk, starting at $1$, with reference point $0$ and force point $z_0=e^{iv_0}\in\partial\D\setminus\{1\}$. It is the random curve whose driving function $W_t=e^{iw_t}$ satisfies the SDE
\begin{equation}dw_t = \frac{\rho}{2}\cot\frac{w_t-v_t}{2}dt+\sqrt{\kappa}dB_t,\quad w_0=0,\label{eq:radialdrive}\end{equation}
where $e^{iv_t}=g_t(e^{iv_0})$, and $B_t$ is a one-dimensional standard Brownian motion.\par
The first appearance of SLE$_\kappa(\rho)$ was in \cite{LSW03}, where it was shown that SLE$_{8/3}(\rho)$ (in the chordal setting with a boundary force point) is the outer boundary of random sets that satisfy a conformal restriction property. In \cite{W04,D05}, this work was further developed. In the latter it was shown that a chordal SLE$_\kappa(\kappa-4)$, with force point at $x_0>0$ and $\kappa\geq 4$, is a chordal SLE$_\kappa$ conditioned not to hit $[x,\infty)$. Two further special cases of chordal SLE$_\kappa(\rho)$ processes, now with a force point $z_0\in\H$, are SLE$_\kappa(\kappa-6)$, $\kappa\leq 4$, which (after re-parametrization) is a radial SLE$_\kappa$ from $0$ to $z_0$, and SLE$_\kappa(\kappa-8)$, $\kappa\leq 4$, which is (the first part) of a chordal SLE$_\kappa$ ``conditioned to pass through $z_0$'' \cite{SW05}. The SLE$_\kappa(\rho)$ curves are also main players in Imaginary Geometry, (see, e.g., \cite{MS16,MS16c,MS16d,MS17}), where they appear as generalized flow-lines of the Gaussian free field (GFF) with suitable boundary conditions.
\subsection{Main results}\label{section:mainresults}
\paragraph{The $\rho$-Loewner energy}
The goal of the present paper is to study the $\rho$-Loewner energy, the large deviations rate function for SLE$_\kappa(\rho)$ as $\kappa\to 0+$, which is defined in the following way.
\begin{definition}[$\rho$-Loewner energy]\label{def:rhoenergy}Fix $z_0\in\overline\H\setminus\{0\}$ and $\rho\in\R$. Let $\gamma$ be a curve in $\H\setminus\{z_0\}$ parametrized by half-plane capacity, starting at $0$. The chordal $\rho$-Loewner energy of $\gamma$ with respect to the reference point $\infty$ and force point $z_0$, is defined by 
$$I^{(\H;0,\infty)}_{\rho,z_0}(\gamma)=\frac{1}{2}\int_{0}^T\bigg(\dot W_t-\rho\Re\frac{1}{W_t-z_t}\bigg)^2 dt,$$
where $z_t=g_t(z_0)$, if $W$ is absolutely continuous, and set to $\infty$ otherwise. Similarly, fix $e^{iv_0}\in\partial\D$ and $\rho\in\R$. Let $\gamma$ be a curve in $\D\setminus\{0\}$ parametrized by conformal radius, starting at $1$. The radial $\rho$-Loewner energy of $\gamma$ with respect to the reference point $0$ and force point $e^{iv_0}$, is defined by 
$$I^{(\D;1,0)}_{\rho,e^{iv_0}}(\gamma)=\frac{1}{2}\int_{0}^T\bigg(\dot w_t-\frac{\rho}{2}\cot\frac{w_t-v_t}{2}\bigg)^2 dt,$$
where $e^{iv_t}=g_t(e^{iv_0})$, if $w$ is absolutely continuous, and set to $\infty$ otherwise. The radial Loewner energy of $\gamma$ with respect to the reference point $0$ is $I^{\D;1,0}(\gamma)=I^{\D;1,0}_{0,e^{iv_0}}(\gamma).$
\end{definition}
\begin{remark}\label{remark:Freidlin}
Specialists familiar with Freidlin-Wentzell theory (see, e.g., Section \cite[Section 5.6]{DZ10}) might immediately recognize that the formulas for the $\rho$-Loewner energy are consistent with that theory. That is, if one, heuristically, applies the the Freidlin-Wentzell theorem to the driving processes of SLE$_\kappa(\rho)$, as $\kappa\to 0+$, then one obtains the $\rho$-Loewner energy as the rate function. Note however, that one can not obtain a large deviation principle on the driving process directly in this way since the drift terms in (\ref{eq:chordaldrive}) and (\ref{eq:radialdrive}) are not uniformly Lipschitz.\end{remark}
\begin{remark}
For curves $\gamma$ which are bounded away from both the reference point and the force point, one can derive an integrated formula for the $\rho$-Loewner energy, comparing it to the regular Loewner energy. E.g., in the chordal setting with $z_0\in\H$ we have 
\begin{equation}I^{(\H;0,\infty)}_{\rho,z_0}(\gamma)=I^{(\H;0,\infty)}(\gamma)+\rho\log\frac{\sin\theta_T}{\sin \theta_0}-\frac{\rho(8+\rho)}{8}\log\frac{|g_T'(z_0)|y_T}{y_0},\label{eq:integratedintro}\end{equation}
where $z_t=g_t(z_0)$, $y_t=\Im z_t$, and $\theta_t=\arg z_t\in(0,\pi)$. The integrated formulas in the other settings are presented in Proposition \ref{prop:integrated}.
\end{remark}
We define the $\rho$-Loewner energy of a curve $\gamma$ in a simply connected domain $D$, starting at a point $a\in\partial D$, with respect to a reference point $b\in\overline D\setminus\{a\}$ and force point $c\in\partial D\setminus\{a,b\}$ using a conformal map from $D$ to $\H$ or $\D$ mapping $a$ and $b$ appropriately. In \cite{D07,SW05}, it was shown that an SLE$_\kappa(\rho)$ starting at $a$, with reference point $b$ and force point $c$ is (after re-parametrization) an SLE$_\kappa(\kappa-\rho-6)$ starting at $a$, with reference point $c$ and force point $b$. In the deterministic setting we have the analog
\begin{equation}I^{(D;a,b)}_{\rho,c}(\gamma)=I^{(D;a,c)}_{-6-\rho,b}(\gamma).\label{eq:coorchange}\end{equation}
This provides a way to translate facts about the chordal Loewner energy to the radial setting. In particular, we obtain upper and lower bounds for the radial Loewner energy in terms of the chordal Loewner energy (see Corollary \ref{cor:chordalradial}). The integrated formulas and the coordinate change property (\ref{eq:coorchange}) are proved in Section \ref{section:definition}.
\paragraph{Large deviation principle.}
Our first main result, which is proved in Section \ref{section:ldp}, is a large deviation principle (LDP) of SLE$_\kappa(\rho)$, as $\kappa\to 0+$. 
\begin{theorem}[LDP]\label{thm:ldp}
Fix $a\in\partial \D$, $b\in\overline \D\setminus \{a\}$, $c\in\partial\D\setminus\{a,b\}$, and $\rho>-2$. Let $\mathcal X^{(D;a,b)}$ denote the space of simple curves from $a$ to $b$ in $\D$ equipped with the Hausdorff topology. In this topology, the SLE$_\kappa(\rho)$ processes starting at $a$, with reference point $b$ and force point $c$, satisfy the large deviation principle with good rate function $I^{(\D;a,b)}_{\rho,c}$ as $\kappa\to 0+$.
\end{theorem}
The LDP result also holds for the family SLE$_\kappa(\kappa+\rho)$, as $\kappa\to 0+$ and $\rho>-2$ is fixed. Hence, the SLE$_\kappa(\tilde\rho)$ starting at $a$, with reference point $c$ and force point $b$, satisfies the LDP, as $\kappa\to 0+$ and $\tilde\rho>-4$ is fixed, with rate function $I^{(\D;a,c)}_{\rho,b}$. One should note here that, when $\tilde\rho>-4$ and $\kappa$ is small, the SLE$_\kappa(\tilde\rho)$ is a.s. stopped within finite time at $t=\tau_{0+}$, where $\tau_{0+}=\lim_{\vare\to 0+}\inf\{t:|W_t-x_t|\leq \vare\}$.\par
\begin{remark}The forthcoming article \cite{AP24} proves an LDP for radial SLE$_{0+}$ in the topology of uniform convergence on the space of simple curves modulo re-parametrization. 
\end{remark}
\begin{remark}The assumption that $\rho>-2$ is a technical one. Our proof of Theorem \ref{thm:ldp} relies on an LDP for the driving processes on finite time intervals $[0,T]$, which we then carry to an LDP on the curves. Since, the driving process of SLE$_\kappa(\rho)$ is terminated within time $T$ with positive probability (for all small $\kappa$) when $\rho\leq -2$, our proof does not work beyond $\rho>-2$.
\label{remark:ldp}\end{remark}
\begin{figure}
\includegraphics[width=\textwidth]{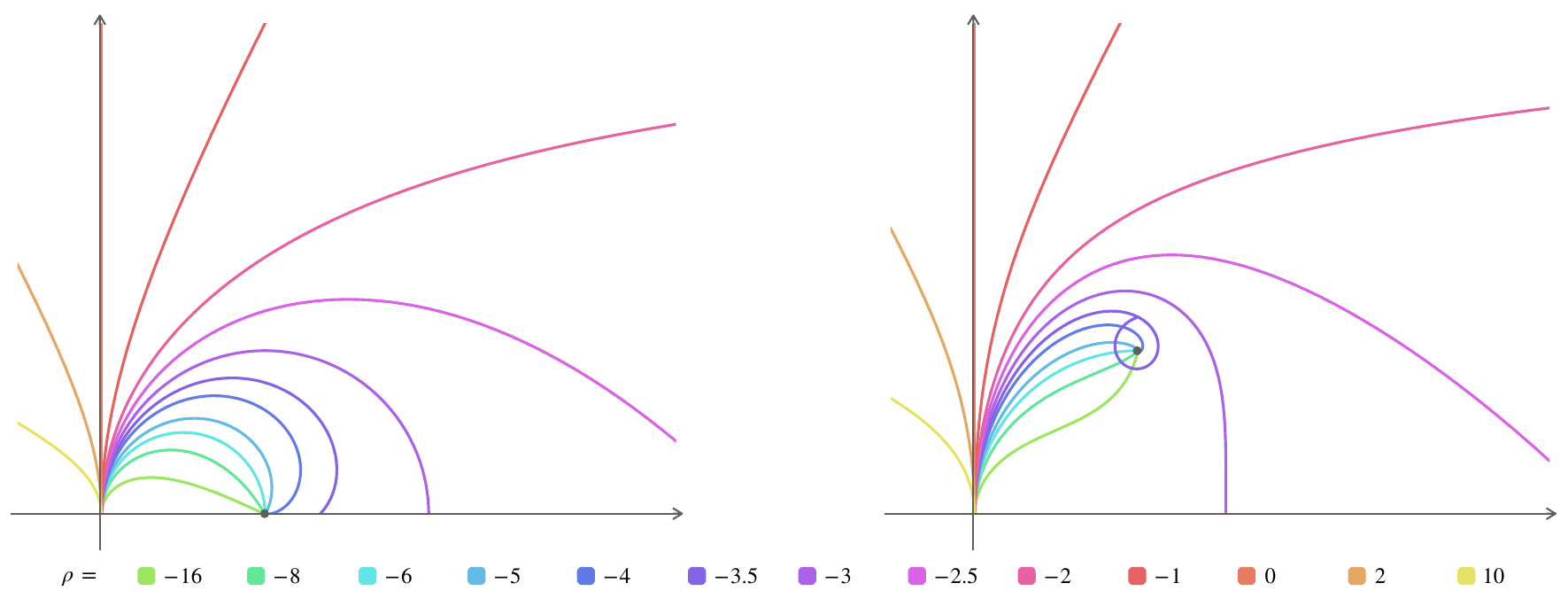}
\caption{Chordal SLE$_0(\rho)$ curves. On the left, the force point is on the boundary, and on the right, the force point is in the interior with $\arg z_0=\pi/2$.\label{fig:sle0}}
\end{figure}
\paragraph{Minimizers and flow-lines.}
It follows directly from the definition of the $\rho$-Loewner energy that its unique minimizer is the SLE$_0(\rho)$ curve. Indeed, there is a unique driving function for which the integrand in the $\rho$-Loewner energy is zero for all $t$, and this is, by definition, the driving function of SLE$_0(\rho)$. In Section \ref{section:minimizers}, we study the SLE$_0(\rho)$ curves. In particular, we show that the Imaginary Geometry interpretation of SLE$_\kappa(\rho)$ as generalized flow-lines of the Gaussian free field (see, e.g., \cite{MS16,MS17}) has the expected analog when $\kappa=0$. That is, the SLE$_0(\rho)$ curve is a (classical) flow-line of the appropriate harmonic field (see Proposition \ref{prop:flowline}). It can be seen in Figure \ref{fig:sle0}, which depicts SLE$_0(\rho)$ curves for some values of $\rho$, that the behavior of SLE$_0(\rho)$ changes drastically with $\rho$. If $\rho\in(-\infty,-4]$, then the SLE$_0(\rho)$ terminates at the force point. If instead $\rho\in(-4,-2)$, then the SLE$_0(\rho)$ terminates by hitting back on itself or the boundary. Lastly, if $\rho\in [-2,\infty)$, then the SLE$_0(\rho)$ approaches the reference point. This (and Remark \ref{remark:ldp}) exemplifies a change of behavior in the $\rho$-Loewner energy when varying $\rho$. After Section \ref{section:minimizers}, we therefore restrict our attention to the cases covered by Theorem \ref{thm:ldp}. In Section \ref{section:minimizers}, we also define a whole-plane variant of radial SLE$_0(\rho)$ for $\rho\geq -2$ and give a flow-line characterization of them. The whole-plane SLE$_0(\rho)$ curves are portrayed, for some values of $\rho>-2$, in Figure \ref{fig:wholeplane}. As expected, given the phase transition of radial SLE$_0(\rho)$ at $\rho=-2$, the whole-plane SLE$_0(-2)$ is different from the whole-plane SLE$_0(\rho)$, $\rho>-2$, in that it does not terminate within finite time. In fact, the whole-plane SLE$_0(-2)$ is a logarithmic spiral, where the spiraling rate depends on a parameter which can be interpreted as the infinitesimal starting position of the force point, see Remark \ref{rmk:wholeplane-2}.
\begin{figure}
\includegraphics[width=\textwidth]{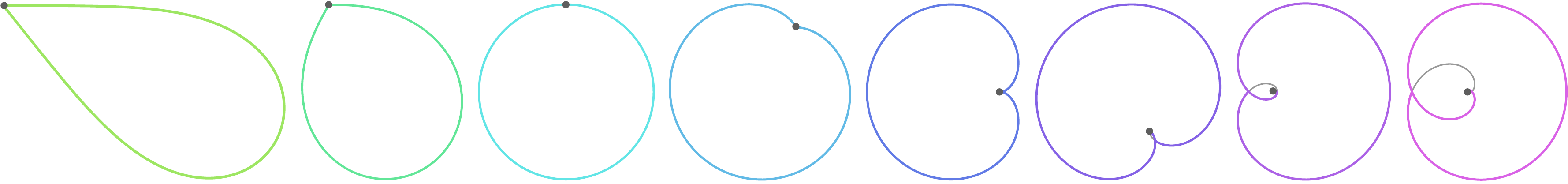}
\caption{Positively oriented whole-plane SLE$_0(\rho)$ curves started at $0$ in the direction $1$ with force point $0$ and reference point $\infty$. From left to right $\rho=-16,\ -8,\ -6,\ -5,\ -4,\ -3.5,\ -3,$ and $ -2.5$. The origin is marked by a gray dot, and in the three right-most figures, the continuation of the flow-line, after self-intersection, is shown in gray. In particular, the SLE$_0(-6)$ is a circle through $0$ and SLE$_0(-4)$ is a cardioid with cusp at $0$.\label{fig:wholeplane}}
\end{figure}
\paragraph{Dirichlet energy formulas.}In Section \ref{section:finiteenergy}, we prove Dirichlet energy formulas for the $\rho$-Loewner energy, in the chordal and radial setting respectively, with a boundary force point and $\rho>-2$.\par
Let $\Sigma=\C\setminus\R^+$. It was shown in \cite{W19b} that, for a chord $\gamma$ from $0$ to $\infty$ in $\Sigma$,
\begin{equation}I^{(\Sigma;0,\infty)}(\gamma)=\frac{1}{\pi}\int_{\Sigma\setminus\gamma}|\nabla\log|h'(z)\|^2dz^2,\label{eq:chorDirichlet}\end{equation}
where $h:\Sigma\setminus\gamma\to\Sigma\setminus\R^-$ is conformal, fixes $\infty$, and maps the upper and lower component of $\Sigma\setminus\gamma$ onto $\H$ and $\H^\ast$ respectively (the latter denoting the lower half-plane). As opposed to the original formula in terms of the driving function, this formula considers the curve as a whole, and in that sense offers a static, rather than dynamic, point of view. Denote by
\begin{equation}\mathcal D(h)=\frac{1}{\pi}\int_D|\nabla \log |h'(z)||^2dz^2\end{equation}
for a conformal map $h$ defined on $D\subset \C$.
\begin{theorem}[Dirichlet energy formula, radial setting]\label{thm:radial}Fix $\rho>-2$ and $z_0\in\Sigma$. Consider a simple curve $\gamma\subset\Sigma$ starting at $0$ and ending at $z_0$. Let $h:\Sigma\setminus\gamma\to\Sigma$ be the conformal map satisfying $h(z_0)=0$, $h(\infty)=\infty$ and $|h'(\infty)|=1$. Denote by $\gamma^0$ the SLE$_0(\rho)$ from $0$ to $z_0$ with force point at $\infty$ and let $h^0:\Sigma\setminus\gamma^0\to\Sigma$ be the conformal map with the same normalization as $h$. Then $\gamma$ has finite $\rho$-Loewner energy if and only if $\mathcal D(h)<\infty$, in which case
\begin{equation}I^{(\Sigma;0,z_0)}_{\rho,\infty}(\gamma)=\mathcal D(h)-\mathcal D(h^0)-\frac{(\rho+6)(\rho-2)}{8}\log|H'(z_0)|_\eta,\label{eq:globalradial}\end{equation}
where $H=(h^0)^{-1}\circ h$ and 
\begin{equation}|H'(z_0)|_{\eta}:=\lim_{s\to 0+}\frac{|H(\eta(s))-H(z_0)|}{|\eta(s)-z_0|},\label{eq:etaderivative}\end{equation}
where $\eta$ is the hyperbolic geodesic from $z_0$ to $\infty$ in $\Sigma\setminus\gamma$. 
\end{theorem}
\begin{remark}The proof uses the Dirichlet energy formula (\ref{eq:chorDirichlet}) and the integrated formula for the $\rho$-Loewner energy (\ref{eq:integratedintro}). The main part of the proof is to show that (\ref{eq:etaderivative}) exists whenever $D(h)<\infty$. Note that $D(h)<\infty$ if and only if $\gamma\cup \eta\cup \R^+$ is a Weil-Petersson quasicircle. Hence, $\gamma$, with $D(h)<\infty$, is not necessarily $C^1$ at $z_0$.\end{remark}
Let $x_0>0$ and denote the upper prime end at $x_0$ of $\Sigma$ by $x_0^+$. In Proposition \ref{proposition:angleapproach} we show that, for $\rho>-2$, chords $\gamma\subset\Sigma$ from $0$ to $\infty$ of finite $\rho$-Loewner energy, with respect to the force point $x_0^+$, approach $\infty$ with an angle $2\alpha\pi$, where $\alpha=\frac{\rho+2}{\rho+4}$, in a certain sense. The corner at $\infty$ causes the Dirichlet energy of $\log|h'|$ to diverge. To combat this, we introduce a re-normalized Dirichlet energy. Consider a chord $\gamma$ from $0$ to $\infty$ in $\Sigma$. Let $h:\Sigma\setminus\gamma\to\Sigma\setminus\R^-$ be a conformal map as in (\ref{eq:chorDirichlet}). For $\beta\in(0,1)$ we define, provided that the limit exists,
\begin{equation}\mathcal D_\beta(h):=\lim_{R\to\infty}\bigg(\frac{1}{\pi}\int_{B(0,R)\setminus(\gamma\cup\R^+)}|\nabla\log|h'\|^2dz^2-c_\beta\log R\bigg),\ c_\beta=\frac{(1-2\beta)^2}{2\beta(1-\beta)}.\label{eq:renormalizedDir}\end{equation}
In light of (\ref{eq:chorDirichlet}), $D_\beta(h)$ can be viewed as a re-normalization of the chordal Loewner energy of $\gamma$ when $\gamma\cup\R^+$ forms a $2\beta\pi$ corner at $\infty$. 
This type of re-normalized Loewner energy has been studied previously in the case of piece-wise linear Jordan curves in \cite{B23}. In a similar spirit, but using a different re-normalization procedure, the divergence of the Loewner energy in the presence of corners, has been studied in connection to Coulomb gas in a Jordan domain with piece-wise analytic boundary \cite{JV23}.\par
We write $\gamma_{[T,\infty)}=\gamma([T,\infty))$, for $T\geq 0$. 
\begin{theorem}[Dirichlet energy formula, chordal setting]\label{thm:chordal}Fix $\rho>-2$ and $x_0>0$, and let $\alpha=\frac{\rho+2}{\rho+4}$. Suppose $\gamma\subset\Sigma$ is a chord from $0$ and to $\infty$ and that there is a $T\geq 0$ such that $\gamma_{[T,\infty)}$ is the $I^{(\Sigma,0,\infty)}_{\rho,x_0^+}$-optimal extension of $\gamma_T$. Then,
\begin{equation}I^{(\Sigma;0,\infty)}_{\rho,x_0^+}(\gamma)=D_\alpha(h)-D_\alpha(h^0)-\frac{\rho(\rho+4)}{4}\log|H'(x_0)|,\label{eq:globalchordal}\end{equation}
where $\gamma^0$ denotes the SLE$_0(\rho)$ from $0$, with reference point $\infty$ and force point $x_0^+$, $h^0$ is its corresponding conformal map, and $H$ is the conformal map from the upper component of $\Sigma\setminus\gamma$ to the upper component of $\Sigma\setminus\gamma^0$ with $H(x_0)=x_0$, $H(\infty)=\infty$ and $H'(\infty)=1$.  
\end{theorem}
We expect (\ref{eq:globalchordal}) to hold for a wider class of curves, that do not necessarily approach $\infty$ in the optimal way, at least under some assumptions on the regularity close to $\infty$.
\paragraph{$\zeta$-regularized determinants.}
Finally, in Section \ref{section:laplacians}, we relate the $\rho$-Loewner energy to $\zeta$-regularized determinants of Laplacians by introducing a $\rho$-Loewner potential, in the same spirit as in \cite{W19b,PW21}. For a domain $D\subset \C$ endowed with a Riemannian metric $g$, we let $\Delta_{(D,g)}$ denote the Friedrichs extension of the Dirichlet Laplace-Beltrami operator on $(D,g)$.\par
Consider a conformal metric $g=e^{2\sigma}dz^2$, $\sigma\in C^\infty(\D)$, and fix $a\in \partial \D$, $b\in\overline \D\setminus\{a\}$, and $c\in\partial \D\setminus\{a,b\}$. For each curve $\gamma\in\D$ from $a$ to $b$ we let $\eta=\eta(\gamma)$ denote the hyperbolic geodesic from $b$ to $c$ in $\D\setminus\gamma$. We define, for $\rho>-2$, the $\rho$-Loewner potential of $\gamma$ with respect to $g$ by
\begin{equation}\mathcal H^{(\D;a,b)}_{\rho,c}(\gamma;g):=\beta \mathcal H^\D(\gamma;g)+(1-\beta) \mathcal H^\D(\gamma\cup\eta;g),\label{eq:loewnerpotential}\end{equation}
provided that the right-hand side is defined, where 
$$\mathcal H^D(K,g)=\log\zdet\Delta_{(D,g)}-\sum_{D_i}\log\zdet\Delta_{(D_i,g)},$$
and $D_i$ are the connected components of $D\setminus K$. Here the coefficients on the right are given by
\[
\beta:=\frac{(2-\rho)(\rho+6)}{12}, \quad 1-\beta = \frac{\rho(\rho+4)}{12},
\]
which are not positive for all $\rho>-2$, so this is not quite a convex combination.\par 
In order to guarantee that the $\rho$-Loewner potential is defined, and to relate it to the $\rho$-Loewner energy, we impose a few regularity assumptions on the metric $g$ and curve $\gamma$. To this end, let $\Gamma\subset\D$ be a curve from $1$ to $0$, smooth up to its endpoints, and fix a conformal map $\varphi:U\cap D\to \D\setminus\Gamma$, for $a\in U\subset\overline \D$ relatively open, with $\varphi$ extending continuously to $U$ and $\varphi(a)=0$. We obtain a smooth slit structure, denoted by $(\overline\D,\varphi)$, on $\overline\D$ by declaring that the $z$-coordinate is smooth on $\overline\D\setminus\{a\}$ and that the $w=\varphi(z)$-coordinate is smooth on $U$.
We say that a $(\overline\D,\varphi)$-smooth curve $\gamma\in\overline\D$, starting at $a$, is smoothly attached at $a$ if $\varphi(\gamma)\cup\Gamma$ is smooth at $a$. For a fixed choice of $a$, $b$, $c$, $\Gamma$ and $\varphi$ as above, we let $\hat{\mathcal X}$ denote the class of $(\overline \D,\varphi)$-smooth curves $\gamma\subset\D\cup\{a,b\}$ from $a$ to $b$, smoothly attached at $a$ such that 
\begin{itemize}
\item if $b\in \D$, then $\gamma\cup\eta$ is smooth at $b$,
\item if $b\in \partial \D$, then there exists a $T\geq 0$ such that $\gamma_{[T,\infty)}$ is the $I^{(\D,a,b)}_{\rho,c}$-optimal extension of $\gamma_T$.
\end{itemize}  
\begin{proposition}[Determinants of Laplacians]\label{prop:laplacians}Fix $\rho>-2$, and $a$, $b$, $c$, $\Gamma$ and $\varphi$ as above. For all $(\overline\D,\varphi)$-smooth conformal metrics $g$, and $\gamma_1,\gamma_2\in\hat{\mathcal X}$ we have
$$I^{(\D,a,b)}_{\rho,c}(\gamma_1)-I^{(\D,a,b)}_{\rho,c}(\gamma_2)=12\Big(\mathcal H_{\rho,c}^{(\D,a,b)}(\gamma_1;g)-\mathcal H_{\rho,c}^{(\D,a,b)}(\gamma_2;g)\Big).$$
\end{proposition}
\begin{remark}
If $b\in\partial\D$, or if $b=0$ and $c$ is antipodal to $a$, then there is a choice of $\Gamma$ and $\varphi$ such that the $I^{(\D,a,b)}_{\rho,c}$-minimizer belongs to $\hat{\mathcal X}$. See Remark \ref{rmk:gamma0inXhat}.
\end{remark}
\begin{remark}
The proof of Proposition \ref{prop:laplacians} uses Theorem \ref{thm:polalv}, a Polyakov-Alvarez type formula for a smooth conformal change of metric proved in our companion paper \cite{K24}. The regularity assumptions at $a$, that is, that $\overline\D$ locally has the geometry of a smooth slit domain and that $\gamma$ is smoothly attached at $a$, are imposed to guarantee smoothness of the change of metric used in the proof.
\end{remark}
\subsection*{Acknowledgments}
First and foremost, I would like to thank my supervisor Fredrik Viklund for numerous inspiring and helpful discussions, and for his guidance throughout the process of writing this paper. I would also like to thank Yilin Wang, Alan Sola, and Vladislav Guskov for several discussions about the project and their comments on the draft. Finally, I thank Lukas Schoug for a helpful discussion on flow-lines and Tim Mesikepp for a conversation that inspired Lemma \ref{lemma:welding}. This work was supported by a grant from the Knut and Alice Wallenberg foundation. Part of this work was carried out at the Simons Laufer Mathematical Sciences Institute, while participating in the program The Analysis and Geometry of Random Spaces.
\section{Preliminaries}\label{section:preliminaries}
\subsection{The chordal and radial Loewner equations}
We here provide some of the details regarding the chordal and radial Loewner equations that were omitted in Section \ref{section:introduction}. For further details, we refer the reader to \cite{L05} and \cite{K17}. A compact set $K\subset\overline \H$ is called a half-plane hull if $\H\setminus K$ is simply connected and $K=\overline{K\cap \H}$. For each half-plane hull there is a unique conformal map $g_K:\H\setminus K\to\H$, called the mapping-out function of $\gamma_t$, satisfying hydrodynamic normalization, that is
$$g_K(z)=z+a_Kz^{-1}+o(z^{-1}),\ \text{as }z\to\infty.$$
A family $(K_t)$ of continuously growing half-plane hulls is said to be parametrized by half-plane capacity if $a_{K_t}=2t$ for all $t$. Solving (\ref{eq:chordalloewner}) for an arbitrary real-valued and continuous function $t\mapsto W_t$ yields a family of conformal maps $g_t:=g_{K_t}:\H\setminus K_t\to\H$, satisfying hydrodynamic normalization, where $(K_t)$ is a family of continuously and locally growing half-plane hulls parametrized by half-plane capacity.\par
In a similar manner, a compact set $K\subset\overline \D$ is called a disk hull if $0\notin K$, $\D\setminus K$ is simply connected, and $K=\overline{K\cap\D}$. The mapping-out function of a disk hull $K$ is the unique conformal map $g_K:\D\setminus K\to\D$ satisfying $g_K(0)=0$ and $g_K'(0)>0$. A family $(K_t)$ of continuously growing disk hulls is said to be parametrized according to conformal radius if $g_{K_t}'(0)=e^{t}$ for all $t$. Solving the radial Loewner equation (\ref{eq:radialLoewner}) for an arbitrary continuous function $t\mapsto W_t=e^{iw_t}\in\partial\D$ gives a family of mapping-out functions $(g_t:=g_{K_t})$ corresponding to a family $(K_t)$ of continuously and locally growing disk hulls parametrized by conformal radius. The radial SLE$_\kappa$, $\kappa\geq 0$, from $1$ to $0$ in $\D$ is the random curve with radial driving process $w_t=\sqrt{\kappa}B_t$, where $B_t$ is a one-dimensional standard Brownian motion.\par 
We will use $\cdot$ to denote time-derivatives, e.g., $\dot g_t(z)=\partial_t g_t(z)$. 
\subsection{SLE\texorpdfstring{$_\kappa(\rho)$}{} 
processes}\label{section:slekapparho} 
Let $\kappa\geq 0$ and $\rho\in\R$. 
\begin{definition}\label{def:cslekapparho}The chordal SLE$_\kappa(\rho)$ in $\H$, starting at $0$, with reference point $\infty$ and force point $z_0\in\overline\H\setminus\{0\}$ is the random curve whose driving process $W_t$ satisfies the SDE
\begin{align*}
dW_t&=\Re\frac{\rho}{W_t-z_t}dt+\sqrt{\kappa}dB_t,\quad W_0=0,
\end{align*}
where $z_t=g_t(z_0)$, 
defined up to the time $\tau_{0+}=\lim_{\vare\to 0+}\tau_\vare$, where $\tau_\vare=\inf\{t>0:|W_t-z_t|=\vare\}$,
and $B_t$ is standard one-dimensional Brownian motion. 
\end{definition}
\begin{definition} The radial SLE$_\kappa(\rho)$-processes in $\D$, starting at $1$, with reference point $0$ and force point $z_0\in\partial \D$ is the random curve whose driving process $W_t$ satisfies the SDE
\begin{align}\label{eq:radialSDE}
dW_t&=-\bigg(\frac{\rho}{2}W_t\frac{W_t+z_t}{W_t-z_t}+\frac{\kappa}{2}W_t\bigg)dt+i\sqrt{\kappa}W_tdB_t,\quad W_0=1,
\end{align}
where $z_t=g_t(z_0)$, defined up to $\tau_{0_+}=\lim_{\vare\to 0+}\tau_\vare$, where $\tau_\vare=\inf\{t:|W_t-z_t|=\vare\}$, and $B_t$ is a standard one-dimensional Brownian motion. 
\end{definition}
\begin{remark}
Note that, if one changes coordinates by $W_t=e^{iw_t}$ and $z_t=e^{iv_t}$, where $w_t$ and $v_t$ are the unique continuous functions with $w_0=0$ and $v_0\in(0,2\pi)$, then, using Itô's formula, (\ref{eq:radialSDE}) transforms into (\ref{eq:radialdrive}).
\end{remark}
\begin{remark}
One can, for some values of $\kappa$ and $\rho$, define the SLE$_\kappa(\rho)$ process past $\tau_{0+}$, up until the so-called continuation threshold. However, when $\kappa$ is small and $\rho>-2$, as in Theorem \ref{thm:ldp}, we have $\tau_{0+}=\infty$ a.s. Hence, the definition above will suffice for our purposes.\end{remark}
\begin{remark}
Often, the random curve in Definition \ref{def:cslekapparho} is called the chordal SLE$_\kappa(\rho)$ in $\H$ from $0$ \emph{to} $\infty$ with force point $z_0$. However, for some values of $\kappa$ and $\rho$, the chordal SLE$_\kappa(\rho)$ is a.s. bounded. To avoid confusion, we refer to $\infty$ as the reference point rather than the end-point. The same convention is used in the radial setting.
\end{remark}
\subsection{Harmonic measure}
We recall some of the basic properties of harmonic measure. The harmonic measure of a Borel set $E\subset\partial \D$ with respect to $0$ and $\D$ is
$$\omega(0,E,\D)=\frac{m(E)}{2\pi},$$
where $m(E)$ is the ``Lebesgue measure'' on $\partial \D$. For a simply connected domain $D$, $z\in\partial D$ and $E\subset \partial D$ (here we consider $\partial D$ as the set of prime ends of $D$) we define
$$\omega(z,E,D):=\omega(0,\varphi(E),\D),$$
where $\varphi:D\to\D$ is a conformal map with $\varphi(z)=0$. Then the harmonic measure is, by construction, a conformally invariant probability measure on $\partial D$. The harmonic measure can be considered as a completely deterministic object, but it also has a useful probabilistic interpretation: $\omega(z,E,D)$ is the probability that a planar Brownian motion, started at $z$, exits $D$ at $E$. This interpretation also applies when $D$ is not simply connected. We will sometimes abuse notation in the following way: if $K$ is a compact subset of $\C\setminus D$, then we will write
$\omega(z,K,D)$ for $\omega(z,K\cap \partial D,D)$.\par
We will sometimes refer to monotonicity of harmonic measure. By this we mean the following:
Suppose that $D_1$ and $D_2$ are two domains with $z\in D_1\cap D_2$, and that $E_1\subset \partial D_1$ and $E_2\subset \partial D_2$ are such that any path in $\C$ starting at $z$ which exits $D_1$ at $E_1$ also exits $D_2$ at $E_2$. Then, the probabilistic interpretation of harmonic measure gives
$$\omega(z,E_1,D_1)\leq \omega(z,E_2,D_2).$$
We will also encounter the harmonic measure with respect to a boundary point. If $D$ is a domain with smooth boundary at $x\in\partial D$  and $E\subset \partial D$ we define
$$\omega_x(E,D) :=\lim_{\vare\to 0+}\frac{1}{\vare}\omega(x+\vare n_x,E,D)$$
where $n_x$ is the inner unit normal at $x$. If $\varphi:D\to D'$ is a conformal map where $D'$ is a domain with smooth boundary at $x'=\varphi(x)$ then 
$$\omega_{x'}(\varphi(E),D')=\omega_{x}(E,D)|\varphi'(x)|^{-1}.$$
If $D$ is a domain such that $0\in\partial (1/D)$, where $1/D:=\{1/z:z\in D\}$, and $\partial (1/D)$ is smooth at $0$, we define 
$$\omega_{\infty}(E,D):=\omega_{0}(1/E,1/D),$$
for $\partial E\subset \partial D$. If $D$ and $D'$ are two such domains and $\varphi:D\to D'$ is a conformal map fixing $\infty$ this gives
$$\omega_{\infty}(\varphi(E),D')=\omega_{\infty}(E,D)|\varphi'(\infty)|^{-1},$$
where we use the convention $\varphi'(\infty):=\psi'(0)$, with $\psi(z)=1/\varphi(1/z)$ (we use this convention for derivatives at $\infty$ throughout). 
Finally, we mention, that if $I=[a,b]\subset \R$, then 
\begin{alignat*}{3}\omega(z,I,\H)&=\frac{1}{\pi}\arg\bigg(\frac{z-a}{z-b}\bigg),\quad && z\in\H,\\
\omega_x(I,\H)&=\frac{1}{\pi}\frac{b-a}{(x-a)(x-b)},\quad && x\in\R\setminus I,\\
\omega_\infty(I,\H)&=\frac{1}{\pi}(b-a).&&
\end{alignat*}
\subsection{Chordal Loewner energy}
The chordal Loewner energy of a curve $\gamma:(0,T)\to\H$, $T\in(0,\infty]$, with $\gamma(0+)=0$ and reference point $\infty$, parametrized by half-plane capacity, is the Dirichlet energy of its driving function $W$, 
$$I^C(\gamma)=\begin{cases}\frac{1}{2}\int_0^T \dot W_t^2 dt, &W\text{ abs. cont. on $(0,t]$ for all $t\in(0,T),$}\\
\infty, &\text{otherwise.}\end{cases}$$
Here $C$ denotes the chordal setting $(\H;0,\infty)$. 
If $I^C(\gamma)<\infty$, then $\gamma$ is simple. One could, in principle, also talk about the chordal Loewner energy of a family of continuously and locally growing half-plane hulls $(K_t)_{t}$. However, unless $(K_t)_t$ corresponds to a simple curve, we have $I^C((K_t)_t)=\infty$, so we can restrict to studying (simple) curves. It is not hard to show that $I^C$ is scale invariant. This allows one to define the Loewner energy of a curve $\gamma$ in a simply connected domain $D$ from $a\in\partial D$ with reference point $b\in\partial D\setminus\{a\}$ by using a conformal map $\varphi:D\to \H$, with $\varphi(a)=0$, $\varphi(b)=\infty$
$$I^{(D;a,b)}(\gamma):=I^{(\H;0,\infty)}(\varphi(\gamma))=I^C(\varphi(\gamma)).$$
The unique minimizer of $I^C$ is $i\R^+$ (since the corresponding driving function is $W\equiv 0$) and hence the unique minimizer of $I^{(D;a,b)}$ is the hyperbolic geodesic from $a$ to $b$. 
The Loewner energy has the additive property
$$I^{(\H;0,\infty)}(\gamma)=I^{(\H;0,\infty)}(\gamma_t)+I^{(\H\setminus\gamma_t;\gamma(t),\infty)}(\gamma_{[t,T)}).$$
Therefore the Loewner energy of a bounded curve $\gamma:(0,T]\to\H$ equals the Loewner energy of the ``completed curve'' $\gamma\cup\eta$, that is,
$$I^{C}(\gamma)=I^C(\gamma\cup\eta),$$
where $\eta$ is the hyperbolic geodesic in $\H\setminus\gamma$ from $\gamma(T)$ to $\infty$.\par
Consider $z_0\in\H$. By \cite[Proposition 3.1]{W19a}
\begin{equation}\inf_{\gamma\ni z_0} I^C(\gamma)=-8\log\sin\arg(z_0)=-8\log\sin(\pi\omega(z_0,(-\infty,0],\H))\label{eq:anglebound}\end{equation}where the infimum is taken over all simple curves passing through $z_0$. The infimum is attained for the curve SLE$_0(-8)$ with force point at $z_0$.
\subsection{Large deviation principles}\label{section:expapprox}
We give an overview of the concepts from large deviations theory that we will use, and refer the reader to \cite{DZ10} for details. Let $\mathcal Y$ be a regular Hausdorff topological space. A function $J:\mathcal Y\to [0,\infty]$ is called a rate function if it is lower-semicontinuous, that is, if the sub-level sets $\{y\in\mathcal Y:J(y)\leq c\}$, $c\geq 0$, are closed. If all sub-level sets are compact we say that $J$ is a good rate function. Let $\mathcal B$ be the Borel $\sigma$-algebra on $\mathcal Y$ and let $(P_\vare)_{\vare>0}$ be a family of probability measures on $(\mathcal Y,\mathcal B)$. We say that $(P_\vare)_{\vare>0}$ satisfies the large deviations principle with rate function $J$ if 
$$\liminf_{\vare\to 0+} \vare\log P_\vare(O)\geq -\inf_{y\in O}J(y),\quad \limsup_{\vare\to 0+} \vare\log P_\vare(F)\leq -\inf_{y\in F}J(y),$$
for all open sets $O$ and closed sets $F$.\par
When showing the large deviation principle for the SLE$_\kappa(\rho)$ driving process we will use theory of exponential approximations. The main idea is that one can obtain an LDP on a family of measures by first showing LDPs for families of measures which are sufficiently good approximations of the original family.
\begin{definition}Let $(\mathcal Y,d)$ be a metric space and $(P_{m,\vare})_{\vare>0}$, $m\in\N$, and $(P_\vare)_{\vare>0}$ be families of probability measures on $(\mathcal Y,\mathcal B)$. We say that $((P_{m,\vare})_\vare)_m$ is an exponentially good approximation of $(P_\vare)_\vare$ if there, for every $\vare>0$, exists a probability space $(\Omega,\mathcal F, \mu_{\vare})$ and random variables $(Z_{\vare})$, $( Z_{m,\vare})$, $m\in\N$, on $(\Omega,\mathcal F, \mu_{\vare})$, with marginal distributions $P_\vare$ and $P_{m,\vare}$ respectively, such that
$$\lim_{m\to\infty}\limsup_{\vare\to 0}\vare\log \mu_{\vare,m}(d(Z_\vare,Z_{\vare,m})>\delta)=-\infty,$$
for each $\delta>0$. 
If the approximating sequence is constant, that is, $P_{m,\vare}=\tilde P_\vare$ for all $m$, then we say that $(P_\vare)$ and $(\tilde P_\vare)$ are exponentially equivalent.
\end{definition}
\begin{thmCite}[{{\cite[Theorem 4.2.16]{DZ10}}}] Suppose that for every $m$, the family of measures $(P_{m,\vare})$ satisfies the LDP with rate function $J_m$ and that $(P_{m,\vare})$ is an exponentially good approximation of $(P_\vare)$. If $J$ is a good rate function and for every closed set $F$,
$$\inf_{y\in F} J(y)\leq \limsup_{m\to\infty}\inf_{y\in F} J_m(y),$$ then $(P_\vare)$ satisfies the LDP with rate function $J$. \label{thm:expapprox}
\end{thmCite}
\subsection{\texorpdfstring{$\zeta$}{}-regularized determinants of Laplacians}  \label{section:prellapl}
Let $(M,g)$ be a Riemannian 2-manifold with $\partial M\neq 0$. The (positive) Dirichlet Laplace-Beltrami operator on $M$, originally defined on smooth functions with compact support, can be extended to a self-adjoint operator using a Friedrichs extension. Denote this operator by $\Delta_{(M,g)}$. The spectrum of $\Delta_{(M,g)}$ is discrete and positive, and the eigenvalues, which can be ordered in non-decreasing order, 
$$0<\lambda_1\leq \lambda_2\leq \lambda_3\leq ... ,$$
satisfy Weyl's law
\begin{equation}\lambda_n\sim \frac{n}{4\pi\Vol(M,g)},\quad \text{as }n\to\infty.\label{eq:weyl}
\end{equation}
Hence, the determinant of $\Delta_{(M,g)}$ is not defined in the classical sense. The $\zeta$-regularized determinant of $\Delta_{(M,g)}$ is defined using the spectral $\zeta$-function
$$\zeta_{(M,g)}(s)=\sum_{n\geq 1}\lambda_n^{-s},\ \Re s>1.$$
Note that the right-hand side converges by (\ref{eq:weyl}). A computation shows that the spectral $\zeta$-function can be rewritten as 
$$\zeta_{(M,g)}(s)=\frac{1}{\Gamma(s)}\int_{0}^\infty t^{s-1}\Tr(e^{-t\Delta_{(M,g)}})dt,\ \Re s>1.$$
Under certain regularity assumptions on $(M,g)$ (which we will make precise below), one can, using a short-time asymptotic expansion of the heat trace $\Tr(e^{-t\Delta_{(M,g)}})$, show that $\zeta_{(M,g)}$ can be analytically continued to a neighborhood of the origin. This allows one to define the $\zeta$-regularized determinant of $\Delta_{(M,g)}$ by 
$$\zdet \Delta_{(M,g)}:=e^{-\zeta_{(M,g)}'(0)}$$
motivated by the formal computation
$$\zeta'_{(M,g)}(s)=-\sum_{n\geq 1}\log\lambda_n \lambda^{-s},\ \Re s>1\leadsto ``\zeta_{(M,g)}'(0)=-\log\prod_{n\geq 1}\lambda_n.\text{''}$$
\begin{definition}\label{def:cpd}Let $M$ be a compact surface with boundary $\partial M\neq \varnothing$, with finitely many distinct marked points $p_1,...,p_n\in\partial M$, a smooth structure and a smooth Riemannian metric $g$ on $M\setminus\{p_1,...,p_n\}$. We say that $(M,g,(p_j),(\beta_j))$ is a curvilinear polygonal domain with corners $\beta_j\pi\in(0,2\pi]$, $j=1,...,n$, if there exists, for each $j=1,...,n$ an open neighborhood $U_j\ni p_j$ and a homeomorphism $\varphi_j:U_j^\circ \to V_j\subset\C$, extending continuously to $U_j$ with $\varphi_j(p_j)=0$ satisfying the following:
\begin{enumerate}[label=(\roman*)]
\item The boundary $\varphi_j(\partial M\cap U_j)$ (viewed in the sense of prime ends) is rectifiable. Let $\gamma:(a,b)\to \C$ be the arc-length parametrization of $\varphi_j(\partial M\cap U_j)$ which is positively oriented and satisfies $\gamma(0)=0$. Then $\gamma|_{(a,0]}$ and $\gamma|_{[0,b)}$ are smooth and form an interior  angle $\beta_j\pi$ at $0$. 
\item The pull-back $(\varphi_j^{-1})^\ast g$ can be expressed as $e^{2\sigma_j}dz^2$ where $\sigma_j\in C^\infty(V_j^\circ)$, and all partial derivatives extend continuously to $\gamma|_{(a,0]}$ and $\gamma|_{[0,b)}$.
\item There is a smooth Jordan curve extension $\Gamma_1$ of $\gamma|_{(a,0]}$ such that $\Gamma_1$ is positively oriented with respect to the bounded component of $\C\setminus\Gamma_1$. Moreover, there exists an extension $\sigma_{j,1}\in C^\infty(V_{j,1})$ of $\sigma_j$, where $V_{j,1}$ is the closure of the bounded component of $\C\setminus\Gamma_1$. 
\item The analogous to the previous condition holds for $\gamma_{[0,b)}$ giving a Jordan curve $\Gamma_2$, a smoothly bounded domain $V_{j,2}$, and $\sigma_{j,2}\in C^\infty(V_{j,2})$.
\end{enumerate}
\end{definition}
\begin{remark}In Definition \ref{def:cpd} ``smooth structure'' refers to a smooth structure in the usual differential geometry sense. That is, a smooth structure on $M\setminus\{p_1,...,p_n\}$ is an atlas of smoothly compatible charts $(U_\alpha,\varphi_\alpha)_\alpha$, where $\varphi_\alpha:U_\alpha\to V_\alpha$ is a homeomorphism and $V_\alpha\subset\overline \H$ relatively open.
\end{remark}
\begin{remark}
Definition \ref{def:cpd} can be generalized to allow for corners with interior angles $\beta_j\pi\in(0,\infty),$ see \cite{K24}. Since we will only encounter curvilinear polygonal domains where the interior angles $\beta_j\pi\in(0,2\pi]$, Definition \ref{def:cpd} will suffice for the purposes of this paper. 
\end{remark}
\begin{definition}\label{def:smooth}Let $(M,g,(p_j),(\beta_j))$ be a curvilinear polygonal domain with angles $\beta_j\pi\in(0,2\pi]$. We say that $\psi:M\to \R$ is smooth, $\psi\in C^\infty(M,g,(p_j),(\beta_j))$, if $\psi\in C^\infty(M\setminus\{p_1,...,p_n\})$ and if there, for every $j=1,...,n$, is a choice of $(\varphi_j,U_j)$, $V_{j,1}$, and $V_{j,2}$ as in Definition \ref{def:cpd} such that all partial derivatives of $\psi\circ\varphi^{-1}$ extend continuously to $\gamma_{(a,0]}$ and $\gamma_{[0,b)}$ (with $\gamma$ as in Definition \ref{def:cpd}) and there exists extensions $\psi_{j,1}\in C^{\infty}(V_{j,1})$ and $\psi_{j,2}\in C^{\infty}(V_{j,2})$ of $\psi\circ\varphi_j^{-1}$.
\end{definition}
In \cite{NRS19}, Nursultanov, Rowlett, and Sher obtained a short-time asymptotic expansion of the heat trace on curvilinear polygonal domains with corners of angles $\beta_j\pi\in(0,2\pi)$, showing that $\zdet\Delta_{(M,g)}$ can be defined for such domains. In the companion paper \cite{K24}, we obtain a short-time asymptotic expansion for curvilinear polygonal domains with corners of arbitrary positive angles ($\beta_j\pi\in(0,\infty)$), showing that $\zdet\Delta_{(M,g)}$ can be defined in such cases as well.\par
The short-time asymptotic expansion of the heat trace is not sufficient to compute the regularized determinant. However, using a short-time asymptotic expansion of $\Tr(\sigma e^{-t\Delta_{(M,g)}})$ one can obtain an explicit formula for $\log\zdet \Delta_{(M,g)}-\log\zdet\Delta_{(M,g_0)}$ where $g=e^{2\sigma}g_0$, for smooth $\sigma$. The first comparison formulas of this type were obtained by Polyakov, in the case of manifolds without boundary \cite{P81}, and Alvarez, in the case of manifolds with a smooth boundary \cite{A83}. The comparison formulas of Polyakov and Alvarez were proved by Osgood, Phillips, and Sarnak in \cite{OPS88}. In \cite{AKR20}, a Polyakov-Alvarez type formula was obtained for curvilinear polygonal domains with angles $\beta_j\pi\in(0,2\pi)$ and in \cite{K24} we show that the same formula holds when $\beta_j\pi\in(0,\infty)$. 
\par 
For a Riemannian surface $(M,g)$, $\Vol_g$ and $\ell_g$ denote the volume and arc-length measures with respect to $g$. Moreover, $K_g$ denotes the Gaussian curvature, $k_g$ the geodesic curvature, and $\partial_n$ the outer unit derivative. We have the following Polyakov-Alvarez comparison formula.  
\begin{thmCite}[{{\cite[Theorem 2]{K24}}}]\label{thm:polalv} Let $(M,g_0,(p_j),(\beta_j))$ be a curvilinear polygonal domain and $\sigma\in C^\infty(M,g_0,(p_j),(\beta_j))$. Define $g=e^{2\sigma}g_0$. Then 
\begin{align*}\log\zdet\Delta_{(M;g_0)}-\log\zdet\Delta_{(M,g)}=&\frac{1}{6\pi}\bigg(\frac{1}{2}\int_M |\nabla_{g_0}\sigma|^2d\Vol_{g_0}+\int_M \sigma K_{g_0}d\Vol_{g_0}+\int_{\partial M}\sigma k_{g_0} d\ell_{g_0}\bigg)\\&+\frac{1}{4\pi}\int_{\partial M}\partial_{n_{g_0}}\sigma d\ell_{g_0} +\frac{1}{12}\sum_{j=1}^n\Big(\frac{1}{\beta_j}-\beta_j\Big)\sigma(p_j).\end{align*}
\end{thmCite}
When working with conformal changes of metric, that is, $g=e^{2\sigma}g_0$ for some smooth $\sigma$, the following transformation rules are useful
\begin{equation}
\begin{alignedat}{2}
\Delta_{g} &= e^{-2\sigma}\Delta_{g_0},\ &\partial_{n_g} &= e^{-\sigma}\partial_{n_{g_0}},\\
d\Vol_{g} &= e^{2\sigma}d\Vol_{g_0},\ & 
d\ell_g &= e^{\sigma}d\ell_{g_0},\\
\Delta_{g_0} \sigma &= e^{2\sigma} K_g-K_{g_0},\quad\quad &\partial_{n_{g_0}}\sigma &= e^{\sigma} k_g - k_{g_0}.
\label{eq:transrulemetric}\end{alignedat}
\end{equation}
For ease of notation, we will drop the subscript specifying the metric when we consider the Euclidean metric, i.e., $\Delta=\Delta_{dz^2}$.
\section{Basic properties of the \texorpdfstring{$\rho$}{}-Loewner energy}\label{section:definition}
In this section, we define the $\rho$-Loewner energy and give some of its basic properties. For ease of notation we let $C=(\H;0,\infty)$ and $R=(\D;1,0)$ denote the two reference settings. 
\begin{definition}\label{def:rhoenergyhp}Fix $z_0\in\overline\H\setminus\{0\}$ and $\rho\in\R$. Let $(K_t)_{t\in [0,T]}$, $T\in[0,\infty)$, be a family of half-plane hulls, parametrized by half-plane capacity and driven by a continuous function $W_t$, with $W_0=0$, such that $z_0\notin K_T$. We define the $\rho$-Loewner energy of $(K_t)_{t\in [0,T]}$ in $\H$ from $0$, with reference point $\infty$ and force point $z_0$, as 
$$I^C_{\rho,z_0}((K_t))=\frac{1}{2}\int_{0}^T\bigg(\dot W_t-\rho\Re\frac{1}{W_t-z_t}\bigg)^2 dt,$$
where $z_t=g_t(z_0)$, if $W_t$ is absolutely continuous and otherwise we set $I^C_{\rho,z_0}((K_t))=\infty$. For a family, $(K_t)_{t\in[0,T)}$, $T\in(0,\infty]$, where $(K_t)_{t\in[0,\tau]}$ is as above, for all $\tau\in[0,T)$, we define
$$I^C_{\rho,z_0}((K_t)_{t\in[0,T)})=\lim_{\tau\to T-}I^C_{\rho,z_0}((K_t)_{t\in[0,\tau]}).$$
\end{definition}
\begin{definition}\label{def:rhoenergyud}Fix $v_0\in(0,2\pi)$ and $\rho\in\R$. Let $(K_t)_{t\in[0,T]}$, $T\in[0,\infty)$, be a family of disk hulls, parametrized by conformal radius and driven by $W_t=e^{iw_t}$, where $w_t$ is continuous with $w_0=0$, such that $e^{iv_0}\notin K_T$. The $\rho$-Loewner energy of $(K_t)_{t\in[0,T]}$ in $\D$ from $1$, with reference point $0$ and force point $e^{iv_0}$, is defined as 
$$I^R_{\rho,e^{iv_0}}((K_t))=\frac{1}{2}\int_0^T\bigg(\dot w_t-\frac{\rho}{2}\cot\Big(\frac{w_t-v_t}{2}\Big)\bigg)^2 dt,$$
where $v_t$ is the unique continuous function with $e^{iv_t}=g_t(e^{iv_0})$, if $w_t$ is absolutely continuous and otherwise we set $I^R_{\rho,e^{iv_0}}((K_t))=\infty$. The definition is extended to families $(K_t)_{t\in[0,T)}$, $T\in(0,\infty]$ as in Definition \ref{def:rhoenergyhp}.
\end{definition}
The integrated formulas for the $\rho$-Loewner energy in the following proposition allows us to compare the $\rho$-Loewner energy to the chordal (or radial) Loewner energy for finite times \emph{strictly before} the force point is hit. 
\begin{proposition}\label{prop:integrated}
For families $(K_t)_{t\in[0,T]}$ of half-plane or disk hulls as in Definitions \ref{def:rhoenergyhp} and \ref{def:rhoenergyud} we have the following
\begin{alignat}{3}
\label{eq:rhoenergyinterior}
I^C_{\rho,z_0}((K_t))&=I^C((K_t))+\rho\log\frac{\sin\theta_T}{\sin \theta_0}-\frac{\rho(8+\rho)}{8}\log\frac{|g_T'(z_0)|y_T}{y_0},&& z_0\in \H,\\
\label{eq:rhoenergyboundary}
I^C_{\rho,z_0}((K_t))&=I^C((K_t))-\rho\log\frac{|W_T-z_T|}{|z_0|}-\frac{\rho(4+\rho)}{4}\log |g_T'(z_0)|,&& z_0\in \R\setminus\{0\},\\
 \label{eq:rhoenergydisk}
 I^R_{\rho,z_0}((K_t))&=I^R((K_t))-\rho\log\frac{| W_T-z_T|}{|W_0-z_0|}-\frac{\rho(4+\rho)}{8}\log(|g_T'(z_0)|^2|g_T'(0)|),\ && z_0\in\partial \D,
\end{alignat}
where $I^C$ denotes the chordal Loewner energy and $$I^R((K_t)):=\begin{cases}\frac{1}{2}\int_{0}^T\dot w_t^2dt &\text{if } w_t \text{ abs. cont.,}\\ \infty &\text{otherwise},\end{cases} $$ is the radial Loewner energy. For $z_0\in\H$ we write $z_t=x_t+iy_t$, $x_t,y_t\in\R$, and $\theta_t=\arg(z_t)$.
\end{proposition}
Note that (\ref{eq:rhoenergyinterior}-\ref{eq:rhoenergydisk}) hold even when the $\rho$-Loewner energy is infinite.
\begin{proof}We show (\ref{eq:rhoenergyinterior}). Identities (\ref{eq:rhoenergyboundary}) and (\ref{eq:rhoenergydisk}) can be shown using the same technique.\par
Fix $z_0\in\H$. If $W_t$ is not absolutely continuous, then $I^C(\gamma)=\infty$ and $I^C_{\rho,z_0}(\gamma)= \infty$, by definition. So, we may assume that $W_t$ is absolutely continuous. We then have
\begin{align*}\frac{1}{2}\int_0^T&\Big(\dot W_t-\rho\Re\frac{1}{W_t-z_t}\Big)^2dt=I^C(\gamma)+\int_0^T\Big(\frac{\rho^2}{2}\frac{(W_t-x_t)^2}{|W_t-z_t|^4}-\rho\dot W_t\frac{W_t-x_t}{|W_t-z_t|^2}\Big)dt.
\end{align*}
Hence, it suffices to show that 
\begin{align}\int_{0}^T\bigg(\frac{\rho^2}{2}\frac{(W_t-x_t)^2}{|W_t-z_t|^4}-&\rho\dot W_t\frac{W_t-x_t}{|W_t-z_t|^2}\bigg)dt=\rho\log\frac{\sin\theta_T}{\sin\theta_0}-\frac{\rho(8+\rho)}{8}\log\frac{|g_T'(z_0)|y_T}{y_0}.
\label{eq:integrated}\end{align}
The Loewner equation gives
\begin{align*}
\dot x_t=2\frac{x_t-W_t}{|W_t-z_t|^2},\quad \dot y_t=-2\frac{y_t}{|W_t-z_t|^2},\quad
\dot g_t'(z_0)=-\frac{2g_t'(z_0)}{(g_t(z_0)-W_t)^2}.
\end{align*}
which in turn yields
\begin{align*}
\partial_t(\log y_t)=-\frac{2}{|W_t-z_t|^2},
\quad \partial_t(\log|g_t'(z_0)|)=\Re\partial_t\log g_t'(z_0)=-2\frac{(W_t-x_t)^2-y_t^2}{|z_t-W_t|^4},
\end{align*}
Since $W_t$ is absolutely continuous, $z_t\in C^1$, and $|W_t-z_t|$ is bounded below, we have that $\log|W_t-z_t|$ is absolutely continuous with
$$\partial_t \log|W_t-z_t|=\dot W_t\frac{W_t-x_t}{|W_t-z_t|^2}+2\frac{(W_t-x_t)^2-y_t^2}{|W_t-z_t|^4}\text{ a.e.}$$By combining the above and integrating we find (\ref{eq:integrated}).
\end{proof}
Proposition \ref{prop:integrated} immediately shows that $I^C_{\rho,z_0}((K_t)_{t\in[0,T]})<\infty$ implies that $(K_t)_{t\in[0,T]}$ corresponds to a simple curve (note that this assumes $z_0\notin K_T$). The same conclusion can be drawn in the disk setting in view of the following proposition.
\begin{proposition}\label{prop:coorchangeradialenergy}
Fix $v_0\in(0,2\pi)$ and let $(K_t^R)_{t\in[0,T]}$ be a family of disk hulls as in Definition \ref{def:rhoenergyud}. Let $\varphi_0:\H\to\D$ be a conformal map with $\varphi_0^{-1}(e^{iv_0})=\infty$, $\varphi_0^{-1}(1)=0$. Then $(K_{t(s)}^C):=(\varphi^{-1}_0(K_{t(s)}^{R}))$, where $s\mapsto t(s)$ is the appropriate re-parametrization, is a family of half-plane hulls as in Definition \ref{def:rhoenergyhp} with respect to the point $z_0^C:=\varphi_0^{-1}(0)$. The reverse also holds, i.e., every such family of half-plane hulls induces a family of disk hulls. Moreover, under this change of coordinates
$$I^R_{\rho,e^{iv_0}}((K_t^R)_t)=I^C_{-6-\rho,z_0}((K_s^C)_s).$$
\end{proposition}
This can be seen as a deterministic analog of \cite[ Theorem 3]{SW05}, and the proof is very similar.
\begin{proof}
Throughout this proof, superscripts $C$ and $R$ will be used to distinguish between variables, functions, etc., belonging to the chordal and radial settings respectively. Most of these computations are given in \cite{SW05}, but for the sake of completeness we present them here as well. Let $(g_t^R)$ be the family of mapping-out functions corresponding to $(K_t^R)$. Since $(K_t^C)$ is a family of continuously growing half-plane hulls ($e^{iv_0}\notin K_T^R$), there is a re-parametrization $t=t(s):[0,S]\to[0,T]$ such that $(K_{t(s)}^C)$ is parametrized by half-plane capacity. Let $g_s^C$ be the mapping-out function of $K_{t(s)}^C$. We denote $\varphi_s:=g_{t(s)}^R\circ\varphi_0\circ(g_s^C)^{-1}:\H\to\D$. It can be seen that $\varphi_s$ must be of the form,
$\varphi_s(z)=\lambda_s\frac{z-z_s}{z-\overline z_s},$ where $z_s=g_s^C(z_0)$, for some $\lambda_s\in\partial \D$. Note that $\lambda_s=\varphi_s(\infty)=g_{t(s)}^R(e^{iv_0})$. Loewner's theorem shows that $(g_s^C)$ satisfies the chordal Loewner equation with $W_s^C=\varphi_s^{-1}(e^{iw_{t(s)}})$. 
Since $(g_t^R)'(0)=e^t$ we have 
$$\frac{dt}{ds}=\partial_s\log|(\varphi_0^{-1})'(0)(g_s^C)'(z_0)\varphi_s'(z_s)| = \partial_s\log|(g_s^C)'(z_0)|-\partial_s\log y_s = \frac{4y_s^2}{|z_s-W_s^C|^4},$$
where the third equality follows from the chordal Loewner equation. Since $$w_{t(s)}=-i\log(\varphi_s(W_s^C)),$$ we have that $w_t$ is absolutely continuous if and only if $W_s^C$ is absolutely continuous, and if they are then
$$\partial_s w_{t(s)}=-i(\partial_s\log\lambda_s+\partial_s\log(W_s^C-z_s)-\partial_s\log(W_s^c-\overline z_s)) \text{ a.e.}$$
From the radial Loewner equation, we have
$\partial_s\lambda_{s}=\lambda_{s}\frac{W_{t(s)}^R+\lambda_{s}}{W_{t(s)}^R-\lambda_{s}}\frac{dt}{ds}$. Thus, $$\partial_s\log\lambda_{s}=\frac{\varphi_s(W_s^C)+\lambda_s}{\varphi_s(W_s^C)-\lambda_s}\frac{dt}{ds}=i\frac{W_s^C-x_s}{y_s}\frac{dt}{ds}.$$
Moreover, the chordal Loewner equation gives
\begin{align*}
\partial_s\log (W_{s}^C-z_{s})&=\frac{\dot W_{s}^C}{W_s^C-z_s}+\frac{2}{(W_s^C-z_s)^2},
\end{align*}
and the same holds when $z_s$ is replaced by $\overline z_s$. Finally, observe that 
$$\cot\bigg(\frac{w_{t(s)}-v_{t(s)}}{2}\bigg)=-\frac{W_{s}^C-x_s}{y_s}.$$
When combining all of the above, we find that
$$\int_{0}^T\bigg(\dot w_t - \rho\cot\bigg(\frac{w_s-v_s}{2}\bigg)\bigg)^2 ds = \int_0^S\bigg(\dot W_s^C+(6+\rho)\frac{W_s^C-x_s}{|W_s^C-z_s|^2}\bigg)^2ds,$$
whenever $W_s^C$ and $w_t$ are absolutely continuous. This completes the proof.
\end{proof}
Since we now know that any hull family, on a closed interval $[0,T]$, with infinite chordal Loewner energy also has infinite $\rho$-Loewner energy, we may from now on restrict our attention to hull families corresponding to simple curves. \par
In most cases, curves in $\H$ ($\D$) will be parametrized according to half-plane capacity (conformal radius) but in general, we will consider curves $\gamma$ up to re-parametrization. When we write $I^C_{\rho,z_0}(\gamma)$ ($I^R_{\rho,e^{iv_0}}(\gamma)$) it is implicit that the curve $\gamma$ is (re-)parametrized by half-plane capacity (conformal radius).\par For a force point $x_0\in\R\setminus\{0\}$ we have the following analog of Proposition \ref{prop:coorchangeradialenergy}.
\begin{proposition}\label{prop:coorchangechordenergy}
Fix $x_0\in\R\setminus\{0\}$. Let $\gamma:(0,T]\to\H$ be a simple curve with $\gamma(0+)=0$ parametrized according to half-plane capacity and let $\varphi:\H\to\H$ be a conformal map satisfying $\varphi(0)=0$ and $\varphi(x_0)=\infty$. Then
\begin{equation}I^C_{\rho,x_0}(\gamma)=I^C_{-6-\rho,\varphi(\infty)}(\varphi(\gamma)).\label{eq:equalitychordalchordal}\end{equation}
\end{proposition}
The proof is similar to that of Proposition \ref{prop:coorchangeradialenergy} and is left to the reader.\par
Let $D\subsetneq \C$ be a simply connected domain, $a\in \partial D$, $b\in\partial D\setminus\{a\}$, and $c\in \overline D\setminus\{a,b\}$ (here we consider $\partial D$ in terms of prime ends). We define the $\rho$-Loewner energy of a simple curve $\gamma\in D\setminus\{c\}$, starting at $a$, with reference point $b$ and force point $c\in\overline D\setminus\{a,b\}$, as
$$I^{(D;a,b)}_{\rho,c}(\gamma):=I^C_{\rho,\varphi(c)}(\varphi(\gamma)),$$
where $\varphi:D\to \H$ is a conformal map with $\varphi(a)=0$ and $\varphi(b)=\infty$.
This is well-defined since $I^C_{\rho,z_0}(\gamma)=I^C_{\rho,\lambda z_0}(\lambda\gamma)$ for any $\lambda>0$. 
Suppose instead that $a\in\partial D$, $b\in D$, and $c\in\partial D\setminus\{a\}$. Then we define the $\rho$-Loewner energy of a simple curve $\gamma$ starting at $a$, with reference point $b$ and force point $c$, as
$$I^{(D;a,b)}_{\rho,c}(\gamma):=I^R_{\rho,\varphi(c)}(\varphi(\gamma)),$$
where $\varphi:D\to \D$ is \emph{the} conformal map with $\varphi(a)=1$ and $\varphi(b)=0$.
Propositions \ref{prop:coorchangeradialenergy} and \ref{prop:coorchangechordenergy} can now be summarized by (\ref{eq:coorchange}).  
The $\rho$-Loewner energy satisfies the same type of additive property as the chordal Loewner energy, that is
$$I^{(D;a,b)}_{\rho,c}(\gamma_T)=I^{(D;a,b)}_{\rho,c}(\gamma_t)+I^{(D\setminus\gamma_t;\gamma(t),b)}_{\rho,c}(\gamma_{[t,T]}),\ t\in(0,T).$$
\section{A large deviation principle for SLE\texorpdfstring{$_\kappa(\rho)$}{}} \label{section:ldp}
The goal of this section is to prove Theorem \ref{thm:ldpinf}, which includes Theorem \ref{thm:ldp}. 
The proof follows the same outline as the proof of the LDP for chordal SLE$_\kappa$ in \cite{PW21}. Using an LDP on the driving process on a finite time interval (obtained in Section \ref{subsection:ldpdrive}) and certain continuity properties of the Loewner map we obtain an LDP on SLE$_\kappa(\rho)$ stopped at a finite time (see Section \ref{subsection:ldpfinite}). We then show, in Section \ref{section:ldpinftime}, the large deviation principle on the full SLE$_\kappa(\rho)$ curve by using estimates on return probabilities (which are shown in Appendix \ref{section:escape}). These steps can be carried out, more or less exactly as in \cite{PW21}. Therefore, the bulk of the work will lie in showing the LDP for the driving process of SLE$_\kappa(\rho)$, and establishing the return probability estimates for SLE$_\kappa(\rho)$.\par
Throughout this section $\P$ denotes the standard Wiener measure and $B_t$ denotes a one-dimensional standard Brownian motion. 
\subsection{LDP on driving process}\label{subsection:ldpdrive}
The radial and chordal settings will be treated in very similar ways. Whenever we wish to distinguish the two cases we will use a symbol $X\in\{C,R\}$, often as a superscript, where $C$ and $R$ denote $(\H,0,\infty)$ and $(\D,1,0)$ respectively.
We define $$H^C(x) := \frac{2}{x}\quad \text{ and }\quad H^R(x):=\cot\bigg(\frac{x}{2}\bigg).$$ Fix $\rho>-2$, $T>0$, and $v_0\in\R\setminus\{0\}$ for $X=C$ and $v_0\in(0,2\pi)$ for $X=R$. Let $C_0([0,T])$ denote the space of real-valued continuous functions $t\mapsto f_t$ on $[0,T]$ satisfying $f_0=0$, endowed with the supremum norm. For each $f\in C_0([0,T])$, there is a unique pair $(w^\rho,v^\rho)$ of continuous functions satisfying 
\begin{align}
w^\rho_t&=\frac{\rho}{2}\int_0^{t}H^X(w^\rho_s-v^\rho_s) ds+f_{t},\label{eq:wrho}\\
v^\rho_t&=-\int_0^{t }H^X(w^\rho_s-v^\rho_s)ds+v_0,\label{eq:vrho}
\end{align}
for $t<\tau_{0+}\land T$ where $\tau_{0+}=\tau_{0+}(f)=\lim_{\vare\to\infty}\tau_{\vare}$ and $\tau_{\vare}:=\inf\{t:|H^X(w^\rho_t-v^\rho_t)|\geq 1/\vare\}$. Indeed, since $H^X$ is Lipschitz  on $\{x:|H^X(x)|\leq 1/\vare\}$ there is a unique solution to (\ref{eq:wrho},\ \ref{eq:vrho}) up to time $\tau_\vare$, for every $\vare>0$. We define, for $\vare>0$,
$$C^X_{\vare,\rho}=\{f\in C_0([0,T]):\tau_{\vare}=\infty\},\quad C^X_{0+,\rho}=\{f\in C_0([0,T]):\tau_{0+}=\infty\},$$
and 
$$\mathcal W^\rho=\mathcal W^\rho(T):C_0([0,T])\to C_0([0,T])$$
by $\mathcal W^\rho(f)=w^\rho(f)$, if $f\in C^X_{0+,\rho}$, and by $\mathcal W^\rho(f)\equiv 0$ if $f\in C_0([0,T])\setminus C^X_{0+,\rho} $. One can, using standard arguments, show that $\mathcal W^\rho$ is continuous on $C^X_{0+,\rho}$. It follows from Lemma \ref{lemma:martingale}, proved below, that $\sqrt{\kappa}B|_{[0,T]}\in C^X_{0+,\rho}$ a.s. for sufficiently small $\kappa$. Therefore, $$\mathcal W^\rho(\sqrt{\kappa}B|_{[0,T]})=w^\rho(\sqrt{\kappa}B|_{[0,T]}), \text{ a.s.}$$ That is, $\mathcal W^\rho(\sqrt{\kappa}B|_{[0,T]})$ coincides a.s. with the driving process of SLE$_\kappa(\rho)$ when $\kappa$ is small.
We introduce a functional $I^X_{\rho,v_0}:C_0([0,T])\to[0,\infty]$ defined by
\begin{equation}I^X_{\rho,v_0}(w)=\frac{1}{2}\int_{0}^{T}\bigg(\dot w_t -\frac{\rho}{2} H^X(w_t-v_t)\bigg)^2dt\label{eq:energydrive}\end{equation}
for all $w\in \mathcal W^\rho(C^X_{0+,\rho})$ which are absolutely continuous, where $v$ is the solution of (\ref{eq:vrho}) given $w$, and otherwise we set $I^X_{\rho,v_0}(w)=\infty$. Observe that, for $f\in C_{0+,\rho}^X$, we have $I_D(f)=I^{X}_{\rho,v_0}(w^\rho(f))$, where $I_D$ denotes the Dirichlet energy.
\begin{proposition}\label{proposition:ldpdrive}
The laws of $\mathcal W^\rho(\sqrt{\kappa}B|_{[0,T]})$ and $\mathcal W^{\kappa+\rho}(\sqrt{\kappa}B|_{[0,T]})$ satisfy the large deviation principle with good rate function $I^X_{\rho,v_0}$ as $\kappa\to 0+$ and $\rho>-2$ is fixed.
\end{proposition}
Before presenting the proof of Proposition \ref{proposition:ldpdrive} we provide a few lemmas. Define
$$K^C(x)=x^2/4\quad\text{ and }\quad K^R(x)=\sin^2(x/2),$$
and note that 
$$(\log K^X(x))' = H^X(x)\quad\text{ and }\quad |H^X(x)|\geq 1/\vare\implies K^X(x)\leq \vare^2.$$
\begin{lemma}\label{lemma:infiniteenergydrive}
For every $M>0$ there exists $\vare>0$ such that $I^X_{\rho,v_0}(w)< M$ implies $w\in \mathcal W^\rho(C^X_{\vare,\rho})$.
\end{lemma}
\begin{proof}Since $I^X_{\rho,v_0}(w)<\infty$ implies $w\in \mathcal W^\rho(C_{0+,\rho}^X)$, it suffices to prove that there exists $\vare>0$ such that $I^X_{\rho,v_0}(w)\geq M$ for all $w\in \mathcal W^\rho(C_{0+,\rho}^X\setminus C_{\vare,\rho}^X).$\par
Suppose that $f\in C^X_{0+,\rho}\setminus C^X_{\delta,\rho}$, for some $\delta>0$, is absolutely continuous, and denote by $w=w^\rho(f)$ and $v=v^\rho(f)$. Then $\tau_{\delta}(f)\leq T$ implying that
$K^X(w_{\tau_\delta}-v_{\tau_\delta})\leq \delta^2.$
We may assume that $\tau_\delta=T$ (since $I^X_{\rho,v_0}(w)\geq I^X_{\rho,v_0}(w_{[0,\tau_\delta]})$). 
Observe that
$$\partial_t \log K^X(w_t-v_t) = \dot w_t H^X(w_t-v_t)+(H^X(w_t-v_t))^2.$$
Therefore,
$$(H^X(w_t-v_t))^2 \leq (H^X(w_t-v_t))^2+(H^X(w_t-v_t)+\dot w_t)^2 =\dot w_t^2+2\partial_t\log K^X(w_t-v_t),$$
which, when plugged into (\ref{eq:energydrive}), gives
\begin{align*}I^X_{\rho,v_0}(w) &\geq \min(1,\tfrac{\rho+2}{2})\bigg(\min(1,\tfrac{\rho+2}{2}) I_D(w)-|\rho|\log\frac{K^X(w_T-v_T)}{K^X(w_0-v_0)}\bigg)\\&\geq -|\rho|\min(1,\tfrac{\rho+2}{2})\log\frac{\delta^2}{K^X(-v_0)}.\end{align*} By setting $\vare=\delta$ sufficiently small, the right-hand side becomes larger than $M$, showing that $I^X_{\rho,v_0}(w)\geq M$ for all $w\in \mathcal W^\rho(C^X_{0+,\rho}\setminus C^X_{\vare,\rho})$. 
\end{proof}
\begin{lemma}\label{lemma:driveenergygoodratefunction} $I^X_{\rho,v_0}:C_0([0,T])\to [0,\infty]$ is a good rate function.
\end{lemma}
\begin{proof}
Let $M\in[0,\infty)$. We wish to show that the sub-level set
$$E_M :=\{w\in C_0([0,T]): I^X_{\rho,v_0}(w)\leq M\}$$
is compact. Take a sequence $(w_n)\subset E_M\subset \mathcal W^\rho(C^X_{0+,\rho})$. Recall that $\mathcal W^\rho:C^X_{0+,\rho}\to C_0([0,T])$ is continuous. Let, for each $n$, $f_n\in C^X_{0+,\rho}$ be such that $w_n=\mathcal W^\rho(f_n)$. Then, $I_D(f_n)=I^{X}_{\rho,v_0}(w_n)\leq M$. Since $I_D:C_0([0,T])\to[0,\infty]$ is a good rate function there exists a subsequence $(f_{n_k})$ converging to $f\in C_0([0,T])$ with $I_D(f)\leq M$. Lemma \ref{lemma:infiniteenergydrive} then implies that $f\in C^X_{0+,\rho}$. By continuity of $\mathcal W^\rho$ on $C^X_{0+,\rho}$, $w_{n_k}$ converges to $w=\mathcal W^\rho(f)$ and since $I^X_{\rho,v_0}(w)=I_D(f)\leq M$, this shows that $E_M$ is compact. 
\end{proof}
We will show Proposition \ref{proposition:ldpdrive} using exponentially good approximations. For each $\vare>0$, let $H^X_\vare:\R\to\R$ be a Lipschitz continuous function satisfying $H^X_\vare(x)=H^X(x)$ whenever $|H^X(x)|\leq 1/\vare$. Let $L_\vare$ denote the corresponding Lipschitz constant. For each $f\in C_0([0,T])$ there exists a unique solution $(w^{\rho,\vare},v^{\rho,\vare})=(w^{\rho,\vare}(f),v^{\rho,\vare}(f))\in C_0([0,T])\times C([0,T])$ to
\begin{align}w^{\rho,\vare}_t&=\frac{\rho}{2}\int_0^t H^X_\vare(w^{\rho,\vare}_s-v^{\rho,\vare}_s)ds+f_t\label{eq:wrhoeps}\\
v^{\rho,\vare}_t&=-\int_0^t H^X_\vare(w^{\rho,\vare}_s-v^{\rho,\vare}_s)ds+v_0.\label{eq:vrhoeps}
\end{align}
Note that $w^{\rho,\vare}:C_0([0,T])\to C_0([0,T])$ is continuous and bijective since for each $w\in C_0([0,T])$ there is a unique $v$ solving (\ref{eq:vrhoeps}), and for this pair $(w,v)$ we may solve (\ref{eq:wrhoeps}) for $f$. It follows from the contraction principle and Schilder's theorem (see, e.g., \cite[Theorem 4.2.1 and Theorem 5.2.3]{DZ10}) that the process $w^{\rho,\vare}(\sqrt{\kappa}B|_{[0,T]})$, $t\in[0,T]$ satisfies the large deviation principle with good rate function $$I^X_{\rho,v_0,\vare}(w)=\begin{cases}\frac{1}{2}\int_0^T(\dot w_t -\frac{\rho}{2}H^X_\vare(w_t-v_t))^2dt& \text{ if }w \text{ abs. cont.,}\\ \infty&\text{ otherwise.}\end{cases}$$ 
Note that, for $f\in C^X_{\vare,\rho}$, we have $\mathcal W^{\rho}(f)=w^{\rho,\vare}(f)$ and $I^X_{\rho,v_0}(\mathcal W^\rho(f))=I^X_{\rho,v_0,\vare}(w^{\rho,\vare}(f))$. 
\begin{lemma}\label{lemma:martingale} Fix $\rho>-2$ and $v_0$ as above. Then there exists $\vare_0>0$ such that, for all $0<\kappa<\rho+2$ and $0<\vare<\vare_0$
\begin{equation}\P[\sqrt{\kappa}B|_{[0,T]}\notin C^X_{\vare,\rho}]\leq c\vare^{2\frac{\rho+2}{\kappa}-1}\label{eq:matringalebnd1}\end{equation} where $c=c(\kappa,\rho,v_0,T)$, for which $\lim_{\kappa\to 0+} \kappa\log c(\kappa,\rho,v_0,T)$ exists and is finite. In particular,
\begin{equation}\limsup_{\kappa\to 0+}\kappa\log\P[\sqrt{\kappa}B|_{[0,T]}\notin C^X_{\vare,\rho}]\to-\infty\text{ as }\vare\to 0+.\label{eq:matringalebnd2}\end{equation}
\end{lemma}
\begin{proof}
Define $\tilde w_t^{\kappa,\rho} := v^{\rho}_{t/\kappa}(\sqrt\kappa B)-w^{\rho}_{t/\kappa}(\sqrt\kappa B)$, $t\in[0,\kappa T]$, so that
$$d\tilde w_t^{\kappa,\rho} = \frac{\rho+2}{2\kappa}H^X(\tilde w_t^{\kappa,\rho})dt+ d\tilde B_{t}$$
where $\tilde B_t=-\sqrt{\kappa}B_{t/\kappa}$ is a standard Brownian motion. We introduce a local martingale $F_{a,v_0}(\tilde w_t^{\kappa,\rho})+t$ where $a=(\rho+2)/\kappa$, and 
$$F_{a,v_0}(x)=-2\int_{v_0}^x\int_{v_0}^t(K^X(u))^{a}du(K^X(t))^{-a}dt$$
for $x\in(0,2\pi)$, when $X=R$, and, for $\pm x>0$, when $X=C$ and $\pm v_0>0$. Indeed, $F_{a,v_0}(\tilde w_t^{\kappa,\rho})+t$ is a local martingale since
\begin{align*}F'_{a,v_0}(x)&=-2\int_{v_0}^x(K^X(u))^{a}du(K^X(x))^{-a},\\
F''_{a,v_0}(x)&=-2+2a\int_{v_0}^x(K^X(u))^{a}du(K^X(x))^{-a-1}(K^X)'(x),\end{align*}
so that $d(F_{a,v_0}(\tilde w_t^{\kappa,\rho})+t)=F'_{a,v_0}(\tilde w_t^{\kappa,\rho})d\tilde B_t$ by Itô's formula. Let $\delta,M>0$ be such that $K^X(v_0)\in(\delta,M)$. We define a stopping time $\tau^{\kappa,\rho}_{\delta,M}= \kappa T\land\tau_\delta^{\kappa,\rho}\land \tau_M^{\kappa,\rho}$, where $$\tau_x^{\kappa,\rho}=\inf\{t:K^X(\tilde w_t^{\kappa,\rho})=x\}.$$ By the condition on $\delta$ and $M$, $F_{a,v_0}(\tilde w_{t\land \tau^{\kappa,\rho}_{\delta,M}}^{\kappa,\rho})+t\land S^{\kappa,\rho}_{\delta,M}$ is uniformly bounded, and therefore a martingale. Hence,
\begin{equation}0=\E[F_{a,v_0}(\tilde w_0^{\kappa,\rho})+0]\leq \E[F_{a,v_0}(\tilde w^{\kappa,\rho}_{\kappa T\land S^{\kappa,\rho}_{\delta,M}})]+\kappa T.\label{eq:martingale}\end{equation}
Observe that $F_{a,v_0}(x)\leq 0$, and that
$$\E[F_{a,v_0}(\tilde w^{\kappa,\rho}_{\kappa T\land S^{\kappa,\rho}_{\delta,M}})|\tau_{\delta}^{\kappa,\rho}\leq \kappa T\land\tau_M^{\kappa,\rho}]\leq F_{a,\tilde v_0^X}(x_\delta^X),$$
where 
$$\tilde v_0^X=\begin{cases}|v_0|, & X=C,\\ v_0\land(2\pi-v_0),& X=R,
\end{cases}\qquad\text{ and } \qquad x_\delta^X=\begin{cases}2\sqrt{\delta}, & X=C,\\ 2\arcsin\sqrt{\delta},& X=R.
\end{cases}$$ 
Combining this with (\ref{eq:martingale}) we obtain
$$\P[\tau_\delta^{\kappa,\rho}\leq \kappa T\land  \tau_M^{\kappa,\rho}]\leq \frac{\kappa T}{-F_{a,\tilde v_0^X}(x^X_\delta)}.$$
One can see that $\P[\tau_M^{\kappa,\rho}<\tau_\delta^{\kappa,\rho}\land\kappa T]\to 0$ as $M\to\infty$ (in the case $X=R$ this is trivial) since the drift of $\tilde w^{\kappa,\rho}_t$ is bounded for $t<\tau_\delta^{\kappa,\rho}$. Therefore,
$$\P[\tau^{\kappa,\rho}_\delta\leq \kappa T]\leq \frac{\kappa T}{-F_{a,\tilde v_0}(x_\delta^X)}.$$
It remains to bound $-F_{a,\tilde v_0}(x^X_\delta)$ from below. For this, we use that
$$\frac{x^2}{\pi^2}\leq K^R(x)\leq \frac{x^2}{4},\ x\in(0,\pi].$$
A computation shows that for $0<x<\tilde v_0^X/3$ and $a\geq 1$ we have $$-F_{a,\tilde v_0}(x)\geq \frac{c_1^a (\tilde v_0^X)^{2a+1}x^{-2a+1}}{4a^2-1},$$ for a constant $c_1>0$. Now set $\delta=\vare^2$. Since $|H^X(x)|\leq 1/\vare$ implies $K^X(x)\leq \delta$ we find
\begin{equation}\P[\sqrt{\kappa}B_{[0,T]}\notin C^X_{\vare,\rho}]\leq \P[\tau^{\kappa,\rho}_{\delta}\leq \kappa T]\leq c_2(\kappa,\rho)\frac{T\vare^{2a-1}}{(\tilde v_0^X)^{2a+1}},\label{eq:explicitconstant}\end{equation}
where $\lim_{\kappa\to 0+}c_2(\kappa,\rho)\in \R$. Here we have used that $x_\delta^X\leq c_3\sqrt{\delta}=c_3\vare$ for a constant $c_3>0$. This concludes the proof.
\end{proof}
\begin{lemma}\label{lemma:continuityinrho} Fix $\vare>0$. There exists $\kappa_0>0$ such that $\|w^{\kappa+\rho}(f)- w^{\rho}(f)\|_\infty\leq\frac{|\kappa|T}{2\vare}e^{\frac{|\rho+\kappa|+2}{2}L_{\vare/2} T}$, where $L_\vare$ is the Lipschitz constant of $H^X_\vare$, for all $f\in C^X_{\vare,\rho}$ and $\kappa\in(0,\kappa_0)$. 
\end{lemma}
\begin{proof}
Let $\tilde T(\kappa) = T\land \inf\{t:|H^X(w^{\kappa+\rho}_t-v^{\kappa+\rho}_t)|\geq 2/\vare\}$. Then for $t\in[0,\tilde T(\kappa)]$
\begin{align*}|w^\rho_t&-w^{\kappa+\rho}_t|+|v^\rho_t-v^{\kappa+\rho}_t|\\
\leq& \frac{|\rho+\kappa|+2}{2}\int_0^t|H^X(w_s^\rho-v_s^\rho)-H^X(w_s^{\kappa+\rho}-v_s^{\kappa+\rho})|ds+\frac{|\kappa|}{2}\int_0^t |H^X(w^\rho_s-v_s^\rho)|ds\\
\leq& \frac{|\rho+\kappa|+2}{2}L_{\vare/2}\int_0^t|(w_s^\rho-v_s^\rho)-(w_s^{\rho+\kappa}-v_s^{\rho+\kappa})|ds+\frac{|\kappa| T}{2\vare}.
\end{align*}
Therefore Grönwall's lemma shows 
\begin{equation}|w^\rho_t-w_t^{\rho+\kappa}|+|v_t^\rho-v_t^{\rho+\kappa}|\leq \frac{|\kappa| T}{2\vare} e^{\frac{|\rho+\kappa|+2}{2}L_{\vare/2} T},\label{eq:gronwall}\end{equation} for $t\in[0,\tilde T(\kappa)]$. Since $|H^X(w_t^\rho-v_t^\rho)|\leq 1/\vare$ for all $t\in [0,T]$ there exists a $\kappa_0$, which may be chosen uniformly over all $f\in C^X_{\vare,\rho}$, such that (\ref{eq:gronwall}) forces $\tilde T(\kappa)=T$ for all $\kappa\in(0,\kappa_0)$. Hence, (\ref{eq:gronwall}) holds for all $t\in[0,T]$ whenever $\kappa\in(0,\kappa_0)$.
\end{proof}
\begin{proof}[Proof of Proposition \ref{proposition:ldpdrive}]For each $n=1,2,...$, we have
\begin{align*}\P[\|\mathcal W^\rho(\sqrt{\kappa}B|_{[0,T]})-w^{\rho,1/n}(\sqrt{\kappa}B|_{[0,T]})\|_\infty>\delta]&\leq \P[\|\mathcal W^\rho(\sqrt{\kappa}B|_{[0,T]})-w^{\rho,1/n}(\sqrt{\kappa}B|_{[0,T]})\|_\infty> 0]\\&\leq \P[\sqrt{\kappa}B|_{[0,T]}\notin C^X_{1/n,\rho}].
\end{align*}
By applying Lemma \ref{lemma:martingale}, we deduce that (the law of) $w^{\rho,1/n}(\sqrt{\kappa}B_t)$ is an exponentially good approximation of (the law of) $\mathcal W^{\rho}(\sqrt{\kappa{B_t}})$. Moreover, we claim that for any closed set $F$
$$\inf_{w\in F}I^X_{\rho,v_0}(w)\leq \limsup_{n\to\infty}\inf_{w\in F}I^X_{\rho,v_0,1/n}(w).$$
Suppose $F$ were a closed set for which this did not hold. Let $M=\inf_{w\in F}I^X_{\rho,v_0}(w)$ so that 
$$\limsup_{n\to\infty}\inf_{w\in F}I^X_{\rho,v_0,1/n}(w)<M.$$ Let $\vare>0$ be as in Lemma \ref{lemma:infiniteenergydrive}, and let $\tilde F=F\cap \mathcal W^\rho(C^X_{\vare,\rho})$. Then 
$$\limsup_{n\to\infty}\inf_{w\in F}I^X_{\rho,v_0,1/n}(w)=\limsup_{n\to\infty}\inf_{w\in \tilde F}I^X_{\rho,v_0,1/n}(w)=\inf_{w\in \tilde F}I^X_{\rho,v_0}(w)$$
since $I^X_{\rho,v_0,1/n}(w)=I^X_{\rho,v_0}(w)$, for all $w\in \mathcal W^\rho(C^X_{1/n,\rho})$ whenever $n\geq 1/\vare$. This contradicts the assumption on $F$ and establishes the claim. By applying Theorem \ref{thm:expapprox}, we obtain the LDP for $\mathcal W^\rho(\sqrt{\kappa}B_t)$.\par
To obtain the LDP for $\mathcal W^{\kappa+\rho}(\sqrt{\kappa}B|_{[0,T]})$ we, in light of Theorem \ref{thm:expapprox}, only have to show that it is exponentially equivalent to $\mathcal W^\rho(\sqrt{\kappa}B|_{[0,T]})$. Let $\delta>0$ and fix $\vare>0$. By Lemma \ref{lemma:continuityinrho} we have that there exists a $\kappa_0$, depending on $\vare$ but not on $\delta$, such that $\|\mathcal W^{\kappa+\rho}(f)-\mathcal W^{\rho}(f)\|_\infty\leq\delta$ for all $\kappa\in(0,\kappa_0)$ and $f\in C^X_{\vare,\rho}$. Thus,
$$\limsup_{\kappa\to 0+}\kappa\log \P[\|\mathcal W^{\kappa+\rho}(\sqrt{\kappa}B|_{[0,T]})-\mathcal W^{\rho}(\sqrt{\kappa}B|_{[0,T]})\|_\infty>\delta]\leq \limsup_{\kappa\to 0+}\kappa\log \P[\sqrt{\kappa}B|_{[0,T]}\notin C^X_{\vare,\rho}].$$
By applying Lemma \ref{lemma:martingale}, and then letting $\vare\to 0+$ we find that
$$\limsup_{\kappa\to 0+}\kappa\log \P[\|\mathcal W^{\kappa+\rho}(\sqrt{\kappa}B_t)-\mathcal W^{\rho}(\sqrt{\kappa}B_t)\|_\infty>\delta]=-\infty,$$
for all $\delta>0$.
This shows that $\mathcal W^{\kappa+\rho}(\sqrt{\kappa}B_t)$ is exponentially equivalent to $\mathcal W^\rho(\sqrt{\kappa}B_t)$.
\end{proof}
\subsection{LDP on finite time curves}\label{subsection:ldpfinite}
Let $\mathcal C^R$ denote the (compact) space of compact subsets of $\overline\D$ endowed with the Hausdorff distance and the topology induced by it. Fix a conformal map $\varphi:\D\to\H$. Let $\mathcal C^C=\varphi(\mathcal C^R)$, endowed with the distance and topology induced by $\varphi$. The induced Hausdorff distance on $\mathcal C^C$ depends on the choice of $\varphi$, however, the induced topology (which we refer to simply as the Hausdorff topology) will not. We will denote the (induced) Hausdorff distance by $d^h:\mathcal C^X\times \mathcal C^X\to \R$.\par
Fix $T\in(0,\infty)$. Let $\mathcal K^C$ and $\mathcal K^R$ denote the space of half-plane hulls and disk hulls respectively. Further, let
$$\mathcal K^C_T=\{K\in\mathcal K^C:\hcap(K)=2T\}\quad \text{ and }\quad K^R_=\{K\in \mathcal K^R:\crad(\D\setminus K,0)=e^{-T}\}.$$ We will consider two topologies on $\mathcal K^X$: the Hausdorff topology and the topology induced by Carathéodory kernel convergence. The former is simply the subspace topology induced by the inclusion $\iota:\mathcal K^X_T\hookrightarrow \mathcal C^X$. The latter is the topology where $(K_n)\subset \mathcal K_T^X$ converges to $K\in \mathcal K_T^X$ if and only if the inverse mapping-out functions $(g_{K_n}^{-1})$ converge uniformly on compacts to $(g_{K}^{-1})$.\par
Throughout this section we fix, in addition to $T$, $\rho>-2$ and $z_0=v_0\in\R\setminus\{0\}$, if $X=C$, and $z_0=e^{iv_0}$, $v_0\in(0,2\pi)$, if $X=R$. The goal of this section is to show the following large deviation principle.
\begin{proposition}\label{prop:finitetimeLDP} The $X$-SLE$_\kappa(\rho)$ process and $X$-SLE$_\kappa(\kappa+\rho)$  process, stopped at time $T$ (with standard parametrization), satisfy the large deviation principle with good rate function $$I^X_{\rho,z_0}:\mathcal K^X_T\to[0,\infty]$$ with respect to the Hausdorff topology on $\mathcal K_T^X$.
\end{proposition}
Since $\tau_{0+}=\infty$ a.s. for the $X$-SLE$_\kappa(\rho)$ when $\rho>-2$ and $\kappa$ is sufficiently small (recall Lemma \ref{lemma:martingale}), we have that an $X$-SLE$_\kappa(\rho)$ stopped at time $T\land \tau_{0+}$ is a.s. the image of $\mathcal W^\rho(\sqrt{\kappa}B|_{[0,T]})$ under the Loewner map
$$\mathcal L^X_T:C_0([0,T])\to \mathcal K^X_T,$$
$$w\mapsto K_T.$$
The Loewner map $\mathcal L^X_T$ is continuous when $\mathcal K^X_T$ is endowed with the Carathéodory topology (see \cite[Proposition 6.1]{K17} for the chordal case and \cite[Proposition 6.1]{MS16b} for the radial case), but is not continuous when $\mathcal  K^X_T$ is endowed with the Hausdorff topology. Lemmas \ref{lemma:HausCara} and \ref{lemma:contfinitetime} were shown in the chordal setting in \cite{PW21} and as the proofs in the radial setting are almost identical we omit them.
\begin{lemma}[{{Chordal case: \cite[Lemma 2.3]{PW21}}}] \label{lemma:HausCara}
Let $(K_n)_n\in\mathcal K^X$ be a sequence of $X$-hulls converging to $K\in\mathcal K^X$ in the Carathéodory topology and to $\tilde K\in\mathcal C$, with $\infty\notin\tilde K$ if $X=C$ and $0\notin\tilde K$ if $X=R$, in the Hausdorff topology. If $X=R$, then $\D\setminus K$ coincides with the connected component of $\D\setminus\tilde K$ containing $0$. If $X=C$, then  $\H\setminus K$ coincides with the unbounded connected component of $\H\setminus\tilde K$. 
\end{lemma}
\begin{lemma}[{{Chordal case: \cite[Lemma 2.4]{PW21}}}]
Let $F$ be a Hausdorff-closed subset of $\mathcal K_T^X.$ If $$w\in\overline{(\mathcal L^X_T)^{-1}(F)}\setminus(\mathcal L^X_T)^{-1}(F)$$ then $\mathcal L^X_T(w)$ has non-empty interior. Similarly, if $O$ is a Hausdorff-open subset of $\mathcal K^X_{T}$ and $w\in(\mathcal L^X_T)^{-1}(O)\setminus((\mathcal L^X_T)^{-1}(O))^\circ$ then $\mathcal L^X_T(w)$ has non-empty interior.\label{lemma:contfinitetime}
\end{lemma}
\begin{lemma}\label{lemma:loewnermapinfinite}
Let $O\subset\mathcal K_{T}^X(z_0)$ be Hausdorff-open and $F\subset\mathcal K_{T}^X(z_0)$ be Hausdorff-closed. Then 
$$\inf_{w\in((\mathcal L^X_T)^{-1}(O))^\circ}I^X_{\rho,v_0}(w)=\inf_{w\in(\mathcal L^X_T)^{-1}(O)}I^X_{\rho,v_0}(w),$$
$$\inf_{w\in\overline{(\mathcal L^X_T)^{-1}(F)}}I^X_{\rho,v_0}(w)=\inf_{w\in(\mathcal L^{X}_T)^{-1}(F)}I^X_{\rho,v_0}(w).$$
\end{lemma}
\begin{proof}
Any $w\in\overline{(\mathcal L^X_T)^{-1}(F)}\setminus(\mathcal L^X_T)^{-1}(F)$ (and similarly $w\in(\mathcal L^{X}_T)^{-1}(O)\setminus(\mathcal L^X_T)^{-1}(O)^\circ$) corresponds to an $X$-hull with non-empty interior $K$ by Lemma \ref{lemma:contfinitetime}. If $w\notin\mathcal W^\rho(C_{0+,\rho}^X)$ then $I^X_{\rho,v_0}(w)=\infty$ by definition. If $w\notin\mathcal W^\rho(C_{0+,\rho}^X)$, then $z_0\notin K_T$ and $I^X_{\rho,v_0}(w)=I^X_{\rho,v_0}((K_t)_t)=\infty$ by the discussion in Section \ref{section:definition}.
\end{proof}
\begin{lemma}\label{lemma:goodratefinitetime}
$I^X_{\rho,z_0}:\mathcal K_T^X\to [0,\infty]$ is a good rate function.
\end{lemma}
\begin{proof}
Fix $M\in[0,\infty)$. We wish to show that the sub-level set 
$$E_M :=\{K\in \mathcal K_T^X: I^X_{\rho,z_0}(K)\leq M\}$$
is compact. Note that all $K\in E_M$ are simple curves (this follows from the discussion in Section \ref{section:definition}). Take a sequence $(\gamma_n)_n\subset E_M$. Let $w_n=(\mathcal L^X_T)^{-1}(\gamma_n)\in \mathcal W^\rho(C^X_{0+,\rho})$ be the corresponding driving functions, and note that $I^X_{\rho,v_0}(w_n)=I^X_{\rho,z_0}(\gamma_n)$. Since $I^X_{\rho,v_0}:C_0([0,T])\to [0,\infty]$ is a good rate function, there exists a subsequence $(w_{n_k})$ converging to some $w\in \mathcal W^\rho(C^X_{0+,\rho})$ with $I^X_{\rho,v_0}(w)\leq M$. Therefore, $K=\mathcal L_T^X(w)$ must be a simple curve $\gamma=K$, with $I^X_{\rho,z_0}(\gamma)=I^X_{\rho,v_0}(w)$. Now, let $F_{k_0}=\overline{\{\gamma_{n_k}\}_{k=k_0}^\infty}$, for $k_0\in\mathbb N$, where the closure is taken in $\mathcal K_{T}^X$ w.r.t. the Hausdorff topology. We now claim that $\gamma\in F_{k_0}$ for all $k_0$. Suppose the opposite. Then there exists $k_0$ such that $w\in \overline{(\mathcal L^{X}_T)^{-1}(F_{k_0})}\setminus (\mathcal L^{X}_T)^{-1}(F_{k_0})$. By Lemma \ref{lemma:contfinitetime}, $\gamma$ has non-empty interior, a contradiction. Therefore, since $\gamma\in F_{k_0}$ for all $k_0\in\mathbb N$, there is a further subsequence $(\gamma_{n_{k_l}})$ converging to $\gamma\in E_M$, as desired.
\end{proof}

\begin{proof}[Proof of Proposition \ref{prop:finitetimeLDP}] Let $\gamma^{\kappa,\rho}$ denote an $X$-SLE$_\kappa(\rho)$, and let $O$ be a Hausdorff-open subset of $\mathcal K_T^X$. Then by Proposition \ref{proposition:ldpdrive} and Lemma \ref{lemma:loewnermapinfinite} we have
\begin{align*}\liminf_{\kappa\to 0+}\kappa\log\P[\gamma_{T}^{\kappa,\rho}\in O]&=\liminf_{\kappa\to 0+}\kappa\log\P[\mathcal W^\rho(\sqrt{\kappa}B|_{[0,T]})\in(\mathcal L_T^X)^{-1}(O)]\\
&\geq\liminf_{\kappa\to 0+}\kappa\log\P[\mathcal W^\rho(\sqrt{\kappa}B|_{[0,T]})\in((\mathcal L_T^X)^{-1}(O))^\circ]\\
&\geq -\inf_{w\in (\mathcal L_T^X)^{-1}(O))^\circ}I^X_{\rho,v_0}(w)\\
&=-\inf_{w\in (\mathcal L_T^X)^{-1}(O)}I^X_{\rho,v_0}(w)\\
&=-\inf_{\gamma\in O}I^X_{\rho,z_0}(\gamma)
\end{align*}The closed sets and the case $\gamma^{\kappa,\kappa+\rho}$ can be treated in the same way. This finishes the proof.
\end{proof}
\subsection{LDP on infinite time curves}\label{section:ldpinftime}
Let $\mathcal X^C\subset \mathcal C^C$ and $\mathcal X^R\subset \mathcal C^R$ denote the spaces of simple curves in $\H$ from $0$ to $\infty$ and in $\D$ from $1$ to $0$ respectively, endowed with the subspace topology. Let
\begin{equation*}D^X=\begin{cases}\H,& X=C,\\ \D,& X=R,\end{cases}\qquad a^X=\begin{cases}0,& X=C,\\ 1,& X=R,\end{cases}\qquad b^X=\begin{cases}\infty,& X=C,\\ 0,& X=R.\end{cases}
\end{equation*}
Throughout this section we fix $\rho>-2,$ and $z_0\in\partial D^X\setminus \{a^X\}$. For $r>0$, let $$C_r=\{z\in D^X:d^h(\{z\},\{b^X\})=r\},$$ where $d^h$ is the (induced) Hausdorff distance (see Section \ref{subsection:ldpfinite}). For a curve $\gamma\in\mathcal X^X$, let $\hat\tau_r=\inf\{s:\gamma(s)\in C_r\}$.
If $\kappa<\min(2(\rho+2),4)$ then the $X$-SLE$_{\kappa}(\rho)$ curve is simple and approaches $b^X$. The first follows from absolute continuity with respect to SLE$_\kappa$ up to $T\land \tau_{0+}$ for all $T\geq 0$, and the second was shown in the chordal and radial settings in \cite[Theorem 1.3]{MS16} and \cite[Proposition 3.30]{MS17} respectively. 
For $\kappa<\min(2(\rho+2),4)$, let $\P^{\kappa,\rho}$ be the $X$-SLE$_\kappa(\rho)$ probability measure on $\mathcal X^X$.\par
The main ingredient needed to extend the large deviation principle from finite to infinite time is the following lemma.
\begin{lemma}\label{lemma:returnestimate} Let $R,M>0$. Then there exists an $r>0$ such that 
\begin{enumerate}[label=(\alph*)]
\item$\inf\{I^X_{\rho,z_0}(\gamma):\gamma\in\mathcal X^X, \gamma_{[\hat\tau_r,\infty)}\cap  C_R\neq \varnothing\}\geq M$,
\item$\limsup_{\kappa\to 0+}\kappa\log\P^{\kappa,\rho}[\gamma_{[\hat\tau_r,\infty)}\cap  C_R\neq \varnothing]\leq -M,$
\item[(b')]$\limsup_{\kappa\to 0+}\kappa\log\P^{\kappa,\kappa+\rho}[\gamma_{[\hat\tau_r,\infty)}\cap  C_R\neq \varnothing]\leq -M.$
\end{enumerate}
\end{lemma}
\begin{proof}In the case $X=R$ this coincides exactly with the statement in Proposition \ref{prop:escaperadial}. In the case $X=C$, this is essentially the same statement as Proposition \ref{prop:chordalescape}, only expressed in a different metric. 
\end{proof}
\begin{corollary}\label{cor:returnestimate}
Let $\tilde R,M>0$. Then there exists $T>0$ such that 
\begin{enumerate}[label=(\alph*)]
\item$\inf\{I^X_{\rho,z_0}(\gamma):\gamma\in\mathcal X^X,d^h(\gamma_T,\gamma)\geq \tilde R\}\geq M,$
\item$\limsup_{\kappa\to 0+}\kappa\log\P^{\kappa,\rho}[d^h(\gamma_T,\gamma)\geq \tilde R]\leq -M,$
\item[(b')]$\limsup_{\kappa\to 0+}\kappa\log\P^{\kappa,\kappa+\rho}[d^h(\gamma_T,\gamma)\geq \tilde R]\leq -M.$
\end{enumerate}
\end{corollary}
\begin{proof}
Let $r$ be as in Lemma \ref{lemma:returnestimate} with $R=\tilde R/2$. We claim that $T=T_r$ is sufficient, where $T_r$ is a deterministic upper bound on $\hat\tau_r$ (such upper bounds exist both in the radial and chordal setting). Suppose $\gamma\in\mathcal X^X$ such that $\gamma_{[\hat\tau_r,\infty]}\cap C_{R}=\varnothing$, then
$$d^h(\gamma_T,\gamma)\leq d^h(\gamma_{\hat\tau_r},\gamma)< r+R< \tilde R.$$ We therefore conclude that
$$\{\gamma\in\mathcal X^X:d^h(\gamma,\gamma_T)\subset \tilde R\}\leq\{\gamma\in\mathcal X^X:\gamma_{[\hat\tau_r,\infty)}\cap C_{R}\neq \varnothing\}.$$
The result now follows from Lemma \ref{lemma:returnestimate}.  
\end{proof}
\begin{theorem} The $\rho$-Loewner energy $I^X_{\rho,z_0}:\mathcal X^X\to[0,\infty]$ is a good rate function.
\end{theorem}
\begin{proof}Let $M\in[0,\infty)$ and consider $E_M=\{\gamma\in\mathcal X^X:I^X_{\rho,z_0}(\gamma)\leq M \}.$ Since $\mathcal C^X$ is compact we have that $E_M$ is compact if it is closed as a subset of $\mathcal C^X$.\par
Let $(\gamma_n)\subset E_M$ be a sequence converging to $K\in\mathcal C^X$, and let $(w^n)\subset C_0([0,\infty))$ be the sequence of corresponding driving functions. Since $I^X_{\rho,v_0}:C_0([0,T])\to[0,\infty]$ is a good rate function for all $T>0$ (see Lemma \ref{lemma:driveenergygoodratefunction}) there is a subsequence $(w_{n_k})$ converging uniformly on compact subsets of $[0,\infty)$ to some $w\in C_0([0,\infty))$. It follows that $I^X_{\rho,v_0}(w|_{[0,T]})\leq M$ for all $T\in(0,\infty)$. Therefore, $w$ encodes a simple curve $\gamma$ with $I^X_{\rho,z_0}(\gamma)\leq M$. Additionally, Corollary \ref{lemma:returnestimate} shows that $\gamma(T)\to b^X$ as $T\to\infty$. Hence $\gamma\in E_M\subset \mathcal X^X$. Note that, the continuity of $\mathcal L_T$ on simple curves (recall Lemma \ref{lemma:contfinitetime}) shows that $\gamma^{n_k}_{T}\to\gamma_{T}$ as $n\to\infty$ for each fixed $T>0$ (as in the proof of Lemma \ref{lemma:goodratefinitetime}).
We now wish to show that $\gamma^{n_k}\to \gamma$, i.e., $K=\gamma$. If $K\neq \gamma$, then $\vare:=d^h(K,\gamma)>0$. Fix $T>0$ such that $I^X_{\rho,z_0}(\eta)> M$ for all $\eta\in \mathcal X^X$ with $d^h(\eta_T,\eta)\geq\vare/4$ (such a $T$ exists by Corollary \ref{cor:returnestimate}). Next, fix $k$ such that $d^h(\gamma^{n_k},K)<\vare/4$ and $d^h(\gamma^{n_k}_T,\gamma_T)<\vare/4$. Then,
$$d^h(K,\gamma)\leq d^h(K,\gamma^{n_k})+d^h(\gamma^{n_k},\gamma^{n_k}_T)+d^h(\gamma^{n_k}_T,\gamma_T)+d^h(\gamma_T,\gamma)<\vare,$$
a contradiction. Therefore, $\gamma_{n_k}\to K=\gamma\in E_M$. We deduce that $E_M$ is closed. 
\end{proof}
\begin{theorem}\label{thm:ldpinf}The $X$-SLE$_\kappa(\rho)$ and $X$-SLE$_\kappa(\kappa+\rho)$ processes satisfy the large deviation principle, as $\kappa\to 0+$, with respect to the Hausdorff topology, with good rate function $I^X_{\rho,z_0}$. That is, for any Hausdorff-open subset $O\subset\mathcal X^X$ and Hausdorff-closed subset $F\subset\mathcal X^X$ we have
\begin{align}\liminf_{\kappa\to 0+}\kappa\log\P^{\kappa,\rho}[\gamma\in O]&\geq -\inf_{\gamma\in O}I^X_{\rho,z_0}(\gamma),\label{eq:ldpinfopen}\\
\limsup_{\kappa\to 0+}\kappa\log\P^{\kappa,\rho}[\gamma\in F]&\leq -\inf_{\gamma\in F}I^X_{\rho,z_0}(\gamma)\label{eq:ldpinfclosed}
\end{align}
and the same holds when $\rho$ is replaced by $\kappa+\rho$.
\end{theorem}
\begin{proof}For $K\in \mathcal C^X$ and $r>0$ we use $B^h(K,r)\subset \mathcal C^X$ to denote the Hausdorff-open ball centered at $K$ and of radius $r$.\par
We start with the closed sets. Let $F\subset \mathcal X^X$ be a Hausdorff-closed set and let $\overline F$ denote its closure in $\mathcal C^X$. Note that $\overline F$ is compact. Moreover, denote 
$$N=\inf_{\gamma\in F}I^X_{\rho,z_0}(\gamma)\in[0,\infty].$$
If $N=0$, then (\ref{eq:ldpinfclosed}) is trivial, so we may assume that $N>0$. 
Fix $M\in(0,N)$.
Since $I^X_{\rho,z_0}$ is a good rate function on $\mathcal X^X$ there exists, for every $K\in\overline{{ F}}$, an $\vare_K>0$ such that 
\begin{equation}\gamma\in B^h(K,3\vare_K)\cap \mathcal X^X\implies I^X_{\rho,z_0}(\gamma)\geq M.\label{eq:vareK}\end{equation}
Since $\{ B^h(K,\vare_K)\}_{K\in\overline F}$ is an open cover of the compact set $\overline F$ we know that there is a finite sub-cover $\{B^h(K_1,\vare_1),...,B^h(K_n,\vare_n)\},$
where we write $\vare_j=\vare_{K_j}$. 
We deduce that 
$$\limsup_{\kappa\to 0+}\kappa\log\P^{\kappa,\rho}[\gamma\in F]\leq \max \{\limsup_{\kappa\to 0+}\kappa\log\P^{\kappa,\rho}[\gamma\in \overline{B^h(K_j,\vare_j)}\cap \mathcal X^X]\}_{j=1}^n.$$
For every $j=1,...,n$, let $T_j>0$ be chosen according to Corollary \ref{cor:returnestimate} with $M$ and $\tilde R=\vare_j$. By the triangle inequality we see that $$\gamma\in \overline{ B^h(K_j,\vare_j)}\cap \mathcal X^X \implies \gamma_{T_j}\in\overline{ B^h(K_j,2\vare_j)}\cap \mathcal K^X_{T_j}\text{ or } d^h(\gamma,\gamma_{T_j})\geq \vare_j.$$ Hence
$\limsup_{\kappa\to 0+}\kappa\log \P^{\kappa,\rho}[\gamma\in \overline{ B^h(K_j,\vare_j)}\cap \mathcal X^X]$
is bounded above by the maximum of
\begin{align}&\limsup_{\kappa\to 0+}\kappa\log\P^{\kappa,\rho}[\gamma_{T_j}\in \overline {B^h(K_j,2\vare_j)}\cap \mathcal K^X_{T_j}]\leq -\inf_{\gamma_{T_j}\in\overline {B^h(K_j,2\vare_j)}\cap \mathcal K^X_{T_j}}I^X_{\rho,z_0}(\gamma_{T_j}),\label{eq:infldp1}\\
&\limsup_{\kappa\to 0+}\kappa\log\P^{\kappa,\rho}[d^h(\gamma,\gamma_{T_j})\geq \vare_j]\leq - M.\label{eq:infldp2}\end{align}
The bound (\ref{eq:infldp1}) follows from Proposition \ref{prop:finitetimeLDP} and the fact that $\overline {B^h(K_j,2\vare_j)}\cap \mathcal K^X_{T_j}$ is a Hausdorff-closed subset of $\mathcal K^X_{T_j}$. The bound (\ref{eq:infldp2}) follows directly from Corollary \ref{cor:returnestimate}. We now show that \begin{equation}\inf_{\gamma_{T_j}\in\overline{ B^h(K_j,2\vare_j)} \cap \mathcal K^X_{T_j}}I^X_{\rho,z_0}(\gamma_{T_j})\geq M.\label{eq:geqM}\end{equation} Take $\gamma_{T_j}\in\overline{B^h(K_j,2\vare_j)} \cap \mathcal K^X_{T_j}$ and denote its $I^X_{\rho,z_0}$-optimal continuation by $\tilde\gamma$. If $d^h(\gamma_{T_j},\tilde\gamma)\geq \vare_j$, then the choice of $T_j$ gives that $I^X_{\rho,z_0}(\tilde\gamma)\geq M$. If instead $d^h(\gamma_{T_j},\tilde\gamma)<\vare$, then the triangle inequality implies that $d^h(\tilde\gamma,K_j)< 3\vare_j$ and the choice of $\vare_j$ then implies that $I^X_{\rho,z_0}(\tilde\gamma)\geq M$. Since $I^X_{\rho,z_0}(\gamma_{T_j})=I^X_{\rho,z_0}(\tilde\gamma)$, this shows (\ref{eq:geqM}) and we deduce that  
$$\limsup_{\kappa\to 0+}\kappa\log\P^{\kappa,\rho}[\gamma^\kappa\in F]\leq - M.$$
Taking the limit $ M\to N-$ we obtain (\ref{eq:ldpinfclosed}).\par
We move on to the open sets. Let $O$ be a Hausdorff-open subset of $\mathcal X^X$ and now let $$N=\inf_{\gamma\in O}I^X_{\rho,z_0}(\gamma).$$ If $N=\infty$, then (\ref{eq:ldpinfopen}) is trivial, so we may assume that $N<\infty$. 
Fix $\vare>0$. There exists a $\gamma^\vare\in O$ such that $I^X_{\rho,z_0}(\gamma^\vare)\leq N+\vare$. Moreover, there exists a $\delta=\delta(\vare)>0$ such that $B^h(\gamma^\vare,2\delta)\cap \mathcal X^X\subset O$. Let $T$ be as in Corollary \ref{cor:returnestimate} with $M=N+2\vare$ and $\tilde R=\delta$. Now, for $\gamma\in\mathcal X^X$
$$\gamma_{T}\in B^h(\gamma^\vare,\delta)\cap \mathcal K_T^X\implies \gamma\in B^h(\gamma^\vare,2\delta)\text{ or } d^h(\gamma_T,\gamma)\geq \delta,$$
by the triangle inequality.
Therefore,
$$ \P^{\kappa,\rho}[\gamma\in B^h(\gamma^\vare,2\delta)\cap \mathcal X^X]\geq \P^{\kappa,\rho}[\gamma_T\in B^h(\gamma^\vare,\delta)\cap \mathcal K^X_T]-\P^{\kappa,\rho}[d^h(\gamma,\gamma_T)\geq \delta].$$
Since $B^h(\gamma^\vare,\delta)\cap \mathcal K_T^X$ is a Hausdorff-open subset of $\mathcal K_T^X$, and since we must have $d^h(\gamma^\vare,\gamma_T^\vare)< \delta$ (by the choice of $T$), Proposition \ref{prop:finitetimeLDP} shows
$$\liminf_{\kappa\to 0+}\kappa\log\P^{\kappa,\rho}[\gamma_T\in B^h(\gamma^\vare,\delta)\cap \mathcal K_T^X]\geq - I^X_{\rho,z_0}(\gamma^\vare_T)\geq -(N+\vare).$$
By Corollary \ref{cor:returnestimate}
$$\limsup_{\kappa\to 0+}\kappa\log\P^{\kappa,\rho}[d^h(\gamma_T,\gamma)\geq \delta]\leq -(N+2\vare).$$
Combining the above we find,
$$\liminf_{\kappa\to 0+}\kappa\log\P^{\kappa,\rho}[\gamma\in O]\geq
-(N+\vare),$$
and by taking the limit $\vare\to 0+$ we obtain (\ref{eq:ldpinfopen}). The same arguments hold when $\rho$ is replaced by $\kappa+\rho$. 
\end{proof}
\section{The minimizers -- SLE\texorpdfstring{$_0(\rho)$}{}}\label{section:minimizers}
It is a direct consequence of the definition of the $\rho$-Loewner energy, that its unique minimizer is the SLE$_0(\rho)$ curve. In this section, we will study these curves. As we will see in Section \ref{section:sle0chordal}, some of these curves have been studied extensively before, but not, to our knowledge, under the name SLE$_0(\rho)$ \cite{LMR10,LR16}. SLE$_0(\rho)$ curves have also appeared in \cite{M23}, in the context of optimization problems for the Loewner energy, and in \cite{ABKM20}, where multichordal SLE$_0$ are described as SLE$_0(\overline\rho)$.\par
We will see that in both the radial and chordal setting the SLE$_0(\rho)$ curves exhibit three phases: a force point hitting phase, a boundary hitting phase, and a reference point hitting phase. This is illustrated in Figure \ref{fig:sle0}.\par In Section \ref{section:flowline} we will show that the interpretation of SLE$_\kappa(\rho)$, $\kappa>0$, as generalized flow-lines of the GFF (see, e.g., \cite{MS16,MS17}) has the expected deterministic analog when $\kappa=0$: SLE$_0(\rho)$ is a flow-line, in the classical sense, of the appropriate harmonic field. In \cite[Section 4.4]{ABKM20}, a similar statement was shown for multichordal SLE$_0$, and in \cite{VW20} the authors proved a finite Loewner energy analog of the SLE-GFF flow-line coupling.\par 
We proceed by studying chordal SLE$_0(\rho)$ with a boundary force point, and radial SLE$_0(\rho)$ with a boundary force point (equivalent to chordal SLE$_0(\rho)$ with an interior force point) separately in Section \ref{section:sle0chordal} and Section \ref{section:sle0radial}. In Section \ref{section:sle0radial} we also define and study a whole-plane variant of SLE$_0(\rho)$ starting at $\infty$, with reference point $0$ and force point $\infty$, when $\rho\leq-2$. 
\subsection{Flow-line property}\label{section:flowline}
For $z_0\in\C$, let $M_{z_0}$ be the logarithmic Riemann surface centered at $z_0$. By this we mean that $M_{z_0}=\{(r,\theta):\ r>0,\ \theta\in\R\}$  endowed with the complex structure induced by $\pi_{z_0}:M_{z_0}\to \C$, $(r,\theta)\mapsto z_0+re^{i\theta}$. For convenience we use the notation $z=z_0+re^{i\theta}$ for $z=(r,\theta)\in M_0$, so that $M_0$ is (locally) identified with $\C$. The main point is that $\arg(\cdot-z_0):M_{z_0}\to \R$ is (with the identification above) a single valued function.\par
We say that a $C^1$ curve $\eta$, parametrized by arc-length, is a flow line of a field $h$ if 
$$\eta'(t)=e^{ih(\eta(t))}.$$
\begin{proposition}\label{prop:flowline} Let $\eta^{\rho,z_0}$ denote an SLE$_0(\rho)$ in $\H,$ starting at $0$, with reference point $\infty$ and force point $z_0\in\overline\H\setminus\{0\},$ re-parametrized by arc-length. If $z_0\in\R\setminus\{0\}$, let $h_{\rho,z_0}:\H\to\R$ be the field defined by
$$h_{\rho,z_0}(z)=\begin{cases}
\pi(1+\tfrac{\rho}{2})-\arg(z)-\tfrac{\rho}{2}\arg(z-z_0),& z_0>0,\\
\pi-\arg(z)-\tfrac{\rho}{2}\arg(z-z_0),& z_0<0,\end{cases}$$ 
where the branches of $\arg(z)$ and $\arg(z-z_0)$ are chosen so that they take values in $(0,\pi)$. Then $\eta^{\rho,z_0}$ a flow-line of $h_{\rho,z_0}$.\par
If $z_0=re^{i\theta}\in\H$, $\theta\in(0,\pi)$, let $h_{\rho,z_0}:\H_{z_0}\to \R$, where $\H_{z_0}:=\pi_{z_0}^{-1}(\H)\subset M_{z_0}$, be the field defined by
$$h_{\rho,z_0}(z)=\pi-\arg(z)-\tfrac{\rho}{4}(\arg(z-z_0)+\arg(z-\bar z_0)),$$
where the branches of $\arg(z)$ and $\arg(z-\overline z_0)$ are chosen so that they take values in $(0,\pi).$ Then the lift of $\eta^{\rho,z_0}$ by $\pi_{z_0}$, starting at $(r,\theta+\pi)$, is a flow-line of $h_{\rho,z_0}$.
\end{proposition}
\begin{proof}
Let $\gamma^{\rho,z_0}$ denote $\eta^{\rho,z_0}$ parametrized by half-plane capacity. Since the driving function of $\gamma^{\rho,z_0}$ satisfies
$$\dot W_t = \Re\frac{\rho}{W_t-z_t}$$
and the right-hand side is continuous in $t$, we have that $\gamma^{\rho,z_0}$ is $C^1$ away from its end-points \cite{W14}. Therefore, for any $t<s$
$$\arg((\gamma^{\rho,z_0})'(t))=\lim_{w\to g_s(\gamma^{\rho,z_0}(t)-)}\arg((g_s^{-1})'(w)),$$
where $g_s$ is the mapping-out function of $\gamma_s^{\rho,z_0}$. We have, for $w_0\in\H\setminus\gamma^{\rho,z_0}$,
$$\partial_t \log(w_t-W_t)=\frac{1}{w_t-W_t}\bigg(\frac{2}{w_t-W_t}-\Re\frac{\rho}{W_t-z_t}\bigg),\quad \partial_t \log g_t'(w_0)=-\frac{2}{(w_t-W_t)^2},$$
$$\partial_t \log(w_t-z_t)=-\frac{2}{(w_t-W_t)(z_t-W_t)},\quad \partial_t \log(w_t-\overline z_t)=-\frac{2}{(w_t-W_t)(\overline z_t-W_t)},$$
where $z_t=g_t(z_0)$, $w_t=g_t(w_0)$, and $W_t$ is the driving function of $\gamma^{\rho,z_0}.$ By combining this and recalling that $\arg=\Im\log$, we find
\begin{equation}\partial_t (h_{\rho,z_t-W_t}(w_t-W_t) - \arg(g_t'(w_0)))=0.\label{eq:flowlinemarkov}\end{equation}
In the case where $z_0\in\H$ we make sense of $h_{\rho,z_t-W_t}(w_t-W_t)$ by lifting $g_t$ to a conformal map from $M_{z_0}$ to $M_{z_t}$. Since $\arg(g_0'(w_0))=0$, we thus have, 
$$g_t'(w_0)=h_{\rho,z_t-W_t}(w_t-W_t)-h_{\rho,z_0}(w_0).$$
Therefore, if $x_0\in\R\setminus\{0\}$,
\begin{align*}\arg((\gamma^{\rho,z_0})'(t))&=\lim_{w_s\to g_s(\gamma^{\rho,z_0}(t)-)}\arg((g_s^{-1})'(w_s)) \\&= \lim_{w_s\to g_s(\gamma^{\rho,z_0}(t)-)}(h_{\rho,z_0}(w_0) - h_{\rho,z_s-W_s}(w_s-W_s)) \\&= h_{\rho,z_0}(\gamma^{\rho,z_0}(t)),\end{align*}
since $h_{\rho,z_0}$ vanishes on $\R^-$ if $z_0>0$ and on $(z_0,0)$ if $z_0<0$. This shows that 
$$(\eta^{\rho,z_0})'(t)=e^{ih_{\rho,z_0}(\eta^{\rho,z_0}(t))},$$
so that $\eta^{\rho,z_0}$ is a flow-line of $h_{\rho,z_0}$.
If $z_0=re^{i\theta}\in\H$, let $\tilde\gamma^{\rho,z_0}$ denote the lift of $\gamma^{\rho,z_0}$ starting at $(r,\theta+\pi)$. Then $h_{\rho,z_0}$ vanishes along the lift of $\R^-$ which has $(r,\theta+\pi)$ as an end-point. We therefore deduce, in the same way as above, that
$$\arg((\gamma^{\rho,z_0})'(t))=\arg((\tilde\gamma^{\rho,z_0})'(t))= h_{\rho,z_0}(\tilde\gamma^{\rho,z_0}(t)).$$
We deduce that the arc-length parametrization, $\tilde\eta^{\rho,z_0}$, of $\tilde\gamma^{\rho,z_0}$ satisfies
$$(\tilde\eta^{\rho,z_0})'(t)=e^{ih_{\rho,z_0}(\tilde\eta^{\rho,z_0}(t))},$$
and it is therefore a flow-line of $h_{\rho,z_0}$.
\end{proof}
\subsection{Chordal SLE\texorpdfstring{$_0(\rho)$}{} with boundary force point}\label{section:sle0chordal}
The chordal SLE$_\kappa(\rho)$ with a boundary force point exhibits the following (probabilistic) self-similarity property: if one maps out an initial part of the curve and then rescales the picture so that the image of the force point is mapped to its initial location one gets a new SLE$_\kappa(\rho)$ with the same force point. When $\kappa=0$ this self-similarity property becomes deterministic.\par
Deterministically self-similar Loewner curves have been studied before in \cite{LMR10} and \cite{LR16}. In \cite{LR16} it was shown that the only curves $\gamma\in C^3$ satisfying 
$$\gamma=r(t)(g_t(\gamma_{t})-W_t),$$
for all $t$ and some $r(t)>0$, are curves driven by 
\begin{enumerate}[label=(\roman*)]
\item $W_t=0$, which corresponds to $\gamma=i\R^{+}$ (as a set),
\item $W_t=ct$, where $c\neq 0$, which corresponds to a curve $\gamma$ which approaches $\infty$ horizontally,
\item $W_t=c\sqrt{t+\tau}-c\sqrt{\tau}$, where $c\neq 0$ and $\tau>0$, corresponding to a curve which approaches $\infty$ at a certain angle $\theta_1(c)\in(0,\pi)$,
\item $W_t=c\sqrt{\tau}-c\sqrt{\tau-t}$, where $|c|\geq 4$ and $\tau>0$, corresponding to a curve which hits the real line at an angle $\theta_2(c)\in [0,\pi)$, and
\item $W_t=c\sqrt{\tau}-c\sqrt{\tau-t}$, where $|c|\in(0,4)$ and $\tau>0$, corresponding to a logarithmic spiral.
\end{enumerate}
These driving functions were also considered in \cite{KNK04}.
\begin{proposition}\label{prop:ChordalSLE0}
The driving function of SLE$_0(\rho)$ with force point $x_0>0$ is 
$$W^{\rho,x_0}_t=\begin{cases}\frac{\rho}{\rho+2}x_0-\frac{\rho}{\rho+2}\sqrt{x_0^2+2(2+\rho)t},& \rho\neq -2,\\
\frac{2}{x_0}t,& \rho=-2.
\end{cases}
$$
When $x_0<0$ we have $W_t^{\rho,x_0}=-W_t^{\rho,-x_0}.$ 
\end{proposition}
\begin{proof}
Assume that $x_0>0$. By definition, the driving function of SLE$_0(\rho)$ satisfies
\begin{align}dW_t^{\rho,x_0}&=\frac{\rho}{W_t^{\rho,x_0}-x_t}dt,\quad W_0^{\rho,x_0}=0.\label{eq:sle0drive}
\end{align}
This gives the separable equation
$$d(x_t-W_t^{\rho,x_0})=\frac{2+\rho}{x_t-W_t^{\rho,x_0}}dt$$
which is solved by $$x_t-W_t^{\rho,x_0} = \sqrt{x_0^2+2(2+\rho)t}.$$
Plugging this into (\ref{eq:sle0drive}) and integrating yields the expression of the driving function. 
The case when $x_0<0$ follows by symmetry.
\end{proof}
\begin{corollary}
Fix $x_0>0$ and let $\gamma^{\rho,x_0}$ be the SLE$_0(\rho)$ with force point $x_0$. Then
\begin{enumerate}[label=(\roman*)]
\item If $\rho\in(-\infty,-4]$, then $\gamma^{\rho,x_0}(\tau)=x_0$ and the (outer) hitting angle is 
$\alpha\pi = \pi\frac{4+\rho}{2+\rho}$. Moreover the hitting time is $\tau = -\frac{x_0^2}{2(2+\rho)}.$
\item If $\rho\in(-4,-2)$, then $\gamma^{\rho,x_0}(\tau)=- \frac{2}{\rho+2}x_0>x_0$ and the (outer) hitting angle is 
$\alpha\pi = \pi\frac{4+\rho}{2}$. Moreover the hitting time is $\tau = -\frac{x_0^2}{2(2+\rho)}.$
\item If $\rho=-2$, then $\gamma^{\rho,x_0}$ approaches $\infty$ asymptotic to $\{x+i\pi x_0:x\in\R^+\}.$
\item If $\rho\in(-2,\infty)$, then $\gamma^{\rho,x_0}$ approaches $\infty$ at an angle $\alpha\pi = \pi\frac{2+\rho}{4+\rho}.$
\end{enumerate}
\label{cor:rhoflow}
\end{corollary}
\begin{proof}These statements all follow from \cite{KNK04} after an appropriate translation and scaling.
\end{proof}
\begin{remark}\label{remark:ray}For $\rho>-2$ one can consider the situation where the force point is placed infinitesimally close to the origin, at $0+$ (or similarly $0-$). From Proposition \ref{prop:ChordalSLE0} we can see that the driving function then becomes
$$W_t^{\rho,0+}=-\rho\sqrt{\frac{2t}{\rho+2}}$$
which corresponds to a ray from $0$ to $\infty$ with angle $\alpha\pi=\pi\frac{2+\rho}{4+\rho}$ (see, e.g., \cite{KNK04}).
\end{remark}
\subsection{Radial SLE\texorpdfstring{$_0(\rho)$}{} with boundary force point}\label{section:sle0radial}
The radial SLE$_0(\rho)$ curves are a bit harder to study. However, we can easily obtain their driving functions. 
\begin{proposition} \label{prop:radialsledrive} The driving function of radial SLE$_0(\rho)$ with force point at $e^{iv_0}$, $v_0\in(0,2\pi)$, is 
\begin{equation}w_t^{\rho,v_0} = \begin{cases}-\frac{\rho}{\rho+2}\Big(2\arccos \Big(\cos\frac{v_0}{2}e^{-\frac{\rho+2}{4}t}\Big) -v_0\Big),& \rho\neq -2,\\
t\cot\frac{v_0}{2},&\rho =-2.
\end{cases}\label{eq:radialdriving}\end{equation}
\end{proposition}
\begin{proof}
The driving function of SLE$_0(\rho)$ satisfies
\begin{align}
dw_t^{\rho,v_0} & = \frac{\rho}{2}\cot\bigg(\frac{w_t^{\rho,v_0}-v_t}{2}\bigg)dt\label{eq:radialsle0drive}
\end{align}
so that $v_t-w_t^{\rho,v_0}$ satisfies the separable equation
$$d(v_t-w_t^{\rho,v_0})=\frac{\rho+2}{2}\cot\bigg(\frac{v_t-w^{\rho,v_0}_t}{2}\bigg)dt,$$ solved by
$$v_t-w_t^{\rho,v_0} = 2\arccos \Big(\cos\tfrac{v_0}{2}e^{-\frac{\rho+2}{4}t}\Big).$$
Plugging this into (\ref{eq:radialsle0drive}) and integrating gives the expression for (\ref{eq:radialdriving}).
\end{proof}
\begin{proposition}
For all $\rho>-2$ the radial SLE$_0(\rho)$ comes arbitrarily close to $0$.
\end{proposition}
\begin{proof} The expression for the radial driving function tells us that the SLE$_0(\rho)$ is not stopped at a finite time for $\rho>-2$.
\end{proof}
\begin{remark}As a consequence of Corollary \ref{cor:infenergy}, which will be shown in Section \ref{section:finiteenergyradial}, this means that, for each $\rho>-2$, the SLE$_0(\rho)$ also approaches $0$.
\end{remark}
Now consider $\rho=-2$. The proof of Proposition \ref{prop:radialsledrive} shows that
$v_t-w_t^{-2,v_0}=v_0,$ for all $t$. This implies that the radial SLE$_0(-2)$, or equivalently, chordal SLE$_0(-4)$ is self-similar. Therefore we can easily compute its chordal driving function. 
\begin{proposition}\label{prop:logspiral}
The chordal SLE$_0(-4)$ with force point $z_0\in\H$ has driving function
$$W_t^{-4,z_0} = 2\cos\theta_0(|z_0|-\sqrt{|z_0|^2-4t}).$$
Hence, the chordal SLE$_0(-4)$ is a logarithmic spiral approaching $z_0$ (unless $\theta_0= \pi/2$). 
\end{proposition}
\begin{proof}
We saw that the driving function of the radial SLE$_0(-2)$, which we denote here by $\gamma^R$, satisfies 
$v_t-w^{-2,v_0}_t=v_0,$ for all $t$. This is equivalent to
$$\omega(0,(\gamma^R_t)^+\cup a,\D\setminus\gamma^R_t)=\omega(0,a,\D),$$ where $a=\{e^{i\theta}:\theta\in[0,v_0]\}$ for all $t$. By conformal invariance of harmonic measure, this gives for chordal SLE$_0(-4)$, denoted here by $\gamma^C$, 
$$\omega(z_0,(\gamma^C_t)^+\cup [0,\infty),\H\setminus\gamma^C_t)=\omega(z_0,[0,\infty),\H),$$ for all $t$, which is equivalent to $\sin\theta_t=\sin\theta_0$ for all $t$.
We have, from the definition of the chordal SLE$_0(-4)$ and the chordal Loewner equation, that
$$\dot W_t^{-4,z_0}=-4\frac{W_t^{-4,z_0}-x_t}{|W_t^{-4,z_0}-z_t|^2},\qquad \dot x_t=2\frac{x_t-W_t^{-4,z_0}}{|W_t^{-4,z_0}-z_t|^2}.$$
Using $\sin\theta_t=\sin\theta_0$ we find 
$$\dot W_t^{-4,z_0}=4\frac{\cos\theta_0}{|W_t^{-4,z_0}-z_t|},\qquad \partial_t\log|W_t^{-4,z_0}-z_t|=-\frac{2}{|W_t^{-4,z_0}-z_t|^2},$$
so that 
$$|W_t^{-4,z_0}-z_t|=\sqrt{|z_0|^2-4t},\qquad W_t^{-4,z_0}=2\cos\theta_0(|z_0|-\sqrt{|z_0|^2-4t}).$$ 
In \cite{KNK04}, it was shown that the curve driven by $W_t=2\sqrt{k(1-t)},$ $k\in(0,4)$, is a logarithmic spiral approaching $\sqrt{k}+\sqrt{k-4}.$ After an appropriate translation, scaling, and reflection, this shows that the SLE$_0(\rho)$ curve is a logarithmic spiral approaching $z_0$. 
\end{proof}
In order to understand radial SLE$_0(\rho)$ with $\rho<-2$ we note that the expression (\ref{eq:radialdriving}) for $w_t^{\rho,v_0}$ can be extended to the interval $(-\infty,T)$ where $T\in \R\cup\{\infty\}$ such that $\pm e^{\frac{\rho+2}{4}T}=\cos\frac{v_0}{2}$ (where the sign depends on whether $v_0\in(0,\pi)$ or $v_0\in(\pi,2\pi)$). That is,
$$w_t^{\rho,v_0}(t)=-\frac{2\rho}{\rho+2}\arccos\Big(\pm e^{\frac{\rho+2}{4}(T-t)}\Big)+C=\frac{2\rho}{\rho+2}\arcsin\Big(\pm e^{\frac{\rho+2}{4}(T-t)}\Big)+\tilde C,\ t\in(-\infty,T)$$
where $C,$ and $\tilde C$ are constants. (This can also be done for $\rho=-2$. See Remark \ref{rmk:wholeplane-2}.) With this in mind, it is natural to make the following definition.
\begin{definition}\label{def:wholeplane}Fix $\rho\in(-\infty,-2)$. The whole-plane SLE$_0(\rho)$, started at $\infty$ in the direction $e^{i\theta}$, of conformal radius $e^T\in(-\infty,\infty)$, positive orientation, and with reference point $0$ and force point $\infty$, is the whole-plane Loewner chain with driving function $e^{iw_t}$ where
$$w_t = \frac{2\rho}{\rho+2}\arcsin(e^{\frac{\rho+2}{4}(T-t)})+\theta,\quad t\in(-\infty,T).$$
Similarly, the same object but with negative orientation is defined as the whole-plane Loewner chain with driving function
$$w_t = -\frac{2\rho}{\rho+2}\arcsin(e^{\frac{\rho+2}{4}(T-t)})+\theta,\quad t\in(-\infty,T).$$
\end{definition}
\begin{remark}\label{remark:wholeplane}This definition is natural since it has the ``deterministic domain Markov property'' that we expect from a whole-plane SLE$_0(\rho)$, that is, upon mapping out the initial part of the curve, we get a radial SLE$_0(\rho)$ with force point at the image of $\infty$. 
It follows from this property that the whole-plane SLE$_0(\rho)$ is a simple curve on $(-\infty,T)$ since we know that the radial SLE$_0(\rho)$ is a simple curve (since it has finite chordal Loewner energy for all times strictly before the force point is swallowed).  
\end{remark}
\begin{proposition}\label{prop:wholeplane}Fix $\rho\in(-\infty,-2)$. Let $\eta$ denote an arc-length parametrization of the whole-plane SLE$_0(\rho)$ starting at $\infty$ in the direction $1$, of any conformal radius, positive orientation, and with reference point $0$ and force point $\infty$. Then $\eta^{\rho,z_0}$ has a lift $\tilde\eta$ to $M_0$, such that $\arg(\tilde\eta(t))\to 0$ as $t\to -\infty$. The curve $\tilde\eta(t)$ is a flow-line of $h(z)=\frac{6+\rho}{4}\arg(z)+\pi$. This further implies that $\eta$ is the image of $\{x+iy_0:x\in(x_0,\infty)\}$, for some $y_0>0$ and $x_0\in[-\infty,\infty)$, under $z\mapsto z^{-\frac{4}{2+\rho}}$ where the branch is chosen so that 
$\arg((x+iy_0)^{-\frac{4}{\rho+2}})\to 0$ as $x\to +\infty$.
\end{proposition}
\begin{proof}Let $\gamma$ denote $\eta$ parametrized by conformal radius. Let $g_t:\C\setminus \gamma_t\to \D$ with $g_t(0)=0$, $g_t'(0)=e^{t}$,
be its family of mapping-out functions. From the theory of whole-plane Loewner evolution, we have that $g_t(z) e^{-t}\to z$ as $t\to-\infty$ uniformly on compacts.\par
There is a unique continuous function $t\mapsto v_t\in \R$, such that, for every $t\in(-\infty,0)$ we have that $g_t(\gamma_{[t,T)})$ is a radial SLE$_0(\rho)$ from $e^{iw_t}$, with reference point $0$ and force point $e^{iv_t}$, and $v_t\to \pi$ as $t\to -\infty$. Once we show that $\gamma$ is ``well behaved" in the limit $t\to -\infty$ we will be able to say that $e^{iv_t} = g_t(\infty)$, but for now, we will simply treat $v_t$ as a continuous function with the given properties. For $t\in(-\infty,T)$, let $\varphi_t:\D\to\H$ be the conformal map with $\varphi_t(e^{iv_t})=\infty$, $\varphi_t(e^{iw_t})=0$, and $|\varphi'_t(0)|=1$. A computation gives $$\varphi_t(z)=-e^{i(v_t-w_t)/2}\frac{z-e^{iw_t}}{z-e^{iv_t}}.$$ Denote by $z_t=\varphi(0)=-e^{-i(v_t-w_t)/2}$. Further, let $\psi_t(z)=i(z-z_t)e^{-t}$, so that $\psi_t$ maps $\H$ onto $\{x+iy:x<e^{-t}\sin((v_t-w_t)/2),y\in\R\}$.\par
Let $\tilde\rho=-6-\rho$. For every $t\in(-\infty,T)$ we know that a lift of $\gamma_{(t,T)}$ to $M_0$, which we denote by $\tilde\gamma_{(t,T)}$, reparametrized by arc-length, is a flow-line of the field $$h_t(z)=h_{z_t,\tilde\rho}(\varphi_t\circ g_t(z)))-\arg((\varphi_t\circ g_t)'(z))$$ defined on $\pi_0^{-1}(\C\setminus \gamma_t)$, where $h_{z_t,\tilde\rho}$ is as defined in Proposition \ref{prop:flowline}. Here, as in the proof of Proposition \ref{prop:flowline}, we make sense of $h_t(z)$
by considering a lift of $\varphi_t\circ g_t$, mapping $\pi_0^{-1}(\C\setminus \gamma_t)$ onto $\pi_{z_t}^{-1}(\H)$.
Moreover, from (\ref{eq:flowlinemarkov}), we obtain that $h_t$ and $h_s$ coincide on the intersection of their domains. We thus obtain a field $h$ defined on all of $M_0$ coinciding with $h_t$ for each $t\in(-\infty,T)$ wherever the latter is defined. 
We now show that $h(z)=-\frac{\tilde\rho}{4}\arg(z)+\pi$. We have 
\begin{align}
\begin{split}h(z)&=h_{z_t,\tilde\rho}(\varphi_t\circ g_t(z))-\arg((\varphi_t\circ g_t)'(z)\\&=\lim_{t\to -\infty}(h_{z_t,\tilde\rho}(\varphi_t\circ g_t(z))-\arg((\varphi_t\circ g_t)'(z)))\\
&=\lim_{t\to -\infty}(\hat h_t(H_t(z))-\arg(H_t'(z))),\end{split}\label{eq:wholeplanefield}\end{align}
where $\hat h_t(z)=h_{z_t,\tilde\rho}(\psi^{-1}(z))-\arg((\psi^{-1})'(z))$ and $H_t(z)=\psi_t\circ \varphi_t \circ g_t$. Using that $g_t(z)e^{t}\to z$ uniformly on compacts as $t\to-\infty$ we find that
$H_t(z)\to 2z$ uniformly on compacts as $t\to-\infty$. Moreover,
$$\hat h_t(z)=\frac{3\pi}{2}-\arg(z_t-ie^{t}z)-\frac{\tilde\rho}{4}(\arg(-ize^t)+\arg(2i\Im z_t -ize^t)),$$
where the first and second arg-terms take values in $(0,\pi)$. Since $H_t(z)e^t\to 0$ and $z_t\to i$ as $t\to -\infty$, 
\begin{align*}&\lim_{t\to -\infty}\arg(2i\Im z_t -iH_t(\gamma(s))e^t)=\frac{\pi}{2},\\
&\lim_{t\to -\infty}\arg(z_t-ie^{t}H_t(z))=\frac{\pi}{2},\\
&\lim_{t\to -\infty}\arg(H_t'(z))=0,\\
&\lim_{t\to -\infty}\arg(-iH_t(z)e^t)=\arg(z)-\frac{\pi}{2}.
\end{align*}
Hence $h(z)=-\frac{\tilde\rho}{4}\arg(z)+\pi$.\par
To show the second part of the statement, consider $(\tilde\eta)^{1+\frac{\tilde\rho}{4}}$. Then $\tilde\eta'(t)=e^{ih(\tilde\eta(t))}$ and the chain rule yields 
$$\arg(((\tilde\eta(s))^{1+\frac{\tilde\rho}{4}})')=\pi.$$
Therefore, $(\tilde\eta)^{1+\frac{\tilde\rho}{4}}$ is (as a set) contained in a horizontal line on $M_0$. Thus, $\eta$ is the image of $\{x+iy_0:x\in(x_0,\infty)\}$, for some $y_0\in\R$ and $x_0\in[-\infty,\infty)$, under some branch of $z\mapsto z^{-\frac{4}{2+\rho}}$. Since $\gamma$ starts in the direction $1$, meaning that $w_t\to 0$ as $t\to -\infty$, the branch must be chosen so that 
$$\arg((x+iy_0)^{-\frac{4}{2+\rho}})\to 0,\quad\text{as }x\to +\infty.$$ Moreover, if $\rho\in(-\infty,-4]$, then $x_0=-\infty$, that is, $\gamma$ is a simple curve both starting and ending at $\infty$. If $\rho\in(-4,-2)$, then $x_0>-\infty$, that is, $\gamma$ is a curve starting at $\infty$ and ending at a self-intersection. Additionally, since $\gamma$ has positive orientation, the harmonic measure of the right side of the curve $(v_t-w_t)/(2\pi)\in(0,1/2)$ for all $t$. Hence, we must have $y_0>0$.
\end{proof}
\begin{remark}Since the driving function of SLE$_0(-2)$ is linear, one can similarly define whole-plane SLE$_0(-2)$ to be a whole-plane Loewner chain with driving function 
$$w_t = t\cot\frac{v_0}{2}+\theta,\quad t\in\R$$
for some $v_0\in(0,2\pi)$ and $\theta\in [0,2\pi)$. With Proposition \ref{prop:wholeplane} in mind one can guess that whole-plane SLE$_0(-2)$ should be a flow-line of $h(z)=\arg(z)+C$ for some $C$ (if we set $C=\pi$ the flow-lines will be rays from $\infty$ to $0$, which is clearly not what we want). Indeed, one can quite easily see, using Proposition \ref{prop:logspiral} and \cite[Proposition 3.3 and Figure 3]{LMR10}, that this is in fact the case. Using the explicit computations of \cite{LMR10,KNK04} one finds that $C=\frac{3\pi}{2}-\frac{v_0}{2}\in(\frac{\pi}{2},\frac{3\pi}{2})$. (Note that $C=\pi/2$ produces concentric circular flow-lines which are counter clockwise oriented, circles centered at $0$, and that $C=3\pi/2$ produces circles of the opposite orientation.) \label{rmk:wholeplane-2} 
\end{remark}
\begin{corollary}Suppose $v_0\in(0,\pi)\cup(\pi,2\pi)$. Let $\gamma:(0,T)\to\D$ be the radial SLE$_0(\rho)$ starting at $0$, with reference point $0$ and force point $e^{iv_0}$. Then $\gamma$ is a simple curve and can be continuously extended to $T$. If $\rho\in(-4,-2)$, then $\gamma(T-)\in\gamma_{(0,T)}\cup \partial \D\setminus\{e^{iv_0}\}$ such that $\D\setminus\gamma$ separates $e^{iv_0}$ from $0$. Moreover, the component of $\D\setminus\gamma$ containing $0$ has interior angle $\pi\frac{4+\rho}{2}$ at $\gamma(T-)$ (unless $\gamma(T-)=1$, in which case the interior angle is $\pi\frac{4+\rho}{4}$). If $\rho\leq-4$, then $\gamma(T-)=e^{iv_0}$ and the component of $\D\setminus\gamma$ containing $0$ has an interior angle $\pi\frac{\rho+4}{\rho+2}$ at $e^{iv_0}$.
\end{corollary}
\begin{proof}The topological properties of $\gamma$ follow directly from Proposition \ref{prop:wholeplane} and Remark \ref{remark:wholeplane}. When $\rho\leq -4$, the size of the intersection angle also follows easily from Proposition \ref{prop:wholeplane} since the whole-plane SLE$_0(\rho)$ forms an angle $2\pi\frac{\rho+4}{\rho+2}$ at $\infty$. For $\rho\in(-4,-2)$, we can find the intersection angle by using Proposition \ref{prop:flowline} in the following way. By changing coordinates we may instead consider $\hat\gamma$, a chordal SLE$_0(-6-\rho)$ in $\H$ with force point $z_0$. Suppose $\arg(z_0)\in(0,\pi/2)$. Then $\hat\gamma$ separates $z_0$ from $\infty$ by winding around $z_0$ clockwise. We may assume that $\hat\gamma$ ends upon hitting $\R^+$ (for otherwise we may achieve this by mapping out a portion of $\hat\gamma$). Now observe that $h_{-\rho-6,z_0}$ takes the value $-\pi\frac{\rho+4}{2}$ at the endpoint of the lift of $\hat\gamma$ starting at $(r,\theta+\pi)$. From this we deduce that the intersection angle is as claimed.
\end{proof}
\section{Finite energy curves and Dirichlet energy formulas}\label{section:finiteenergy}
This section is devoted to studying curves of finite $\rho$-Loewner energy when the force point is on the boundary and $\rho>-2$ as well as proving Theorems \ref{thm:radial} and \ref{thm:chordal}.\par 
In Section \ref{section:finiteenergyradial} we study fully grown curves $\gamma\subset\D$ starting at $1$, with respect to the reference point $0$ and force point $z_0\in\partial\D\setminus\{1\}$ (we will make precise what this means). We prove that $I^R_{\rho,z_0}(\gamma)<\infty$, $\rho>-2$, if and only if $\gamma$ approaches $0$ (that is, ends at $0$) and $I^{(\D;1,e^{iv_0})}(\gamma)<\infty$. Therefore, for a fully grown curve $\gamma$ and $\rho_1,\rho_2>-2$, we have 
$$I^R_{\rho_1,z_0}(\gamma)<\infty \iff I^R_{\rho_2,z_0}(\gamma)<\infty.$$
In Section \ref{section:globalchordal}, study fully grown curves $\gamma\subset\H$ starting at $0$, with respect to the reference point $\infty$ and force point $x_0>0$. We show that $I^C_{\rho,x_0}(\gamma)<\infty$ only if $\gamma$ is transient and approaches $\infty$ at an angle $\alpha\pi=\frac{\rho+2}{\rho+4}\pi$ (in the sense of Proposition \ref{proposition:angleapproach}). Therefore, for a fully grown curve $\gamma$ and $\rho_1,\rho_2>-2$, with $\rho_1\neq \rho_2$, we have
$$I^C_{\rho_1,x_0}(\gamma)<\infty \implies I^C_{\rho_2,x_0}(\gamma)=\infty.$$
Theorems \ref{thm:radial} and \ref{thm:chordal} are proved in Sections \ref{section:finiteenergyradial} and \ref{section:globalchordal} respectively. 
\subsection{Radial setting}\label{section:finiteenergyradial}
In order to benefit from the previous knowledge about the chordal Loewner energy (e.g., the bound (\ref{eq:anglebound}) and the Dirichlet energy formula (\ref{eq:chorDirichlet})) we change coordinates and instead consider the chordal setting with an interior force point and $\rho<-4$. All of our findings can be translated back to the radial setting using an appropriate conformal map. Throughout this section we fix $z_0=x_0+iy_0\in\H$. Let $\mathcal X^C_{z_0}$ be the class of simple curves $\gamma:(0,T)\to\H\setminus\{z_0\}$, with $\gamma(0+)=0$, which are maximal in the sense that $T=\tau_{0+}=\lim_{\vare\to 0+}\tau_{\vare}$ where $\tau_\vare=\inf\{t\in(0,T):|W_t-z_t|\leq \vare\}$. Note that this allows for $T=\infty$. We call $\gamma\in\mathcal X^C_{z_0}$  a fully grown curve with respect to $\infty$ and $z_0$.\par
The following proposition gives upper and lower bounds of the $\rho$-Loewner energy in terms of the first two terms of the right-hand side of (\ref{eq:rhoenergyinterior}).
\begin{proposition}\label{prop:radialbounds} Suppose $\rho<-4$ and suppose $\gamma:(0,T]\to\H\setminus\{z_0\}$ is a simple curve with $\gamma(0)=0$. Then
\begin{align*}I^C_{\rho,z_0}(\gamma)&\leq \max(1,-\tfrac{4+\rho}{4})\bigg(\max(1,-\tfrac{4+\rho}{4})I^C(\gamma)+\rho\log\frac{\sin\theta_T}{\sin\theta_0}\bigg),\\
I^C_{\rho,z_0}(\gamma)&\geq \min(1,-\tfrac{4+\rho}{4})\bigg(\min(1,-\tfrac{4+\rho}{4})I^C(\gamma)+\rho\log\frac{\sin\theta_T}{\sin\theta_0}\bigg).
\end{align*}
\end{proposition}
\begin{proof}
We estimate $|g_t'(z_0)|$ in terms of the chordal Loewner energy as in \cite{FS17}. That is, we have
\begin{align*}\partial_t \log|g_t'(z_0)|y_t &= -4\frac{(W_t-x_t)^2}{|W_t-z_t|^4}=2\partial_t \log\sin\theta_t+4\frac{(W_t-x_t)^2}{|W_t-z_t|^4}+2\dot W_t\frac{W_t-x_t}{|W_t-z_t|^2}\\&\geq 2\partial_t \log\sin\theta_t-\frac{1}{4}\dot W_t^2,
\end{align*}
where the equalities follow from the chordal Loewner equation. This shows
$$2\log\frac{\sin\theta_T}{\sin\theta_0}-\frac{1}{2}I^C(\gamma)\leq\log\frac{|g_T'(z_0)|y_t}{y_0}\leq 0.$$
Plugging this into (\ref{eq:rhoenergyinterior}) one obtains the desired bounds.
\end{proof}
\begin{corollary}
\label{cor:infenergy}
Let $\gamma\in\mathcal X^C_{z_0}$ and let $\rho<-4$. Then,
\begin{enumerate}[label=(\alph*)]
\item If $I^C(\gamma)=\infty$, then $I^C_{\rho,z_0}(\gamma)=\infty$.
\item If $\gamma$ is unbounded, then $I^C_{\rho,z_0}(\gamma)=\infty$.
\end{enumerate}
\end{corollary}
\begin{proof}Note that $\min(1,-\frac{4+\rho}{4})>0$ and recall that $I^C_{\rho,z_0}(\gamma_t)$ is increasing in $t$. Statement (a) holds since $\rho\log\sin\theta_t$ is non-negative. Thus, if $I^C(\gamma_t)$ diverges, then $I^C_{\rho,z_0}(\gamma_t)$ must also do so. Consider statement (b). We claim that if $\gamma$ is unbounded, then for every $\vare>0$ there is a $t\in(0,T)$ such that
$\sin \theta_t < \vare.$ If so, then (b) follows, since $I^C(\gamma_t)$ is non-negative. We now prove the claim. Observe that $$\frac{\theta_t}{\pi}=\omega(z_t,(-\infty,W_t],\H)=\omega(z_0,\R^-\cup\gamma^-_t,\H\setminus\gamma_t),$$
$$\frac{\pi-\theta_t}{\pi}=\omega(z_t,[W_t,\infty),\H)=\omega(z_0,\R^+\cup\gamma^+_t,\H\setminus\gamma_t).$$
For $R>|z_0|$, let $T_R=\inf\{t:|\gamma(t)-x_0|\geq R\}$ and $D_R=\{z\in\H:|z-x_0|<R\}$. Note that $D_R\setminus\gamma_{T_R}$ consists of two components, $D_R^-$ to the left of $\gamma$ and $D_R^+$ to the right of $\gamma$. Suppose $z_0\in D_R^+$. Then, by monotonicity of harmonic measure
$$\omega(z_0,\R^-\cup\gamma^-_t,\H\setminus\gamma_t)\leq \omega(z_0,\partial D_R\cap \H,D_R^+)\leq \omega(z_0,\partial D_R\cap \H,D_R).$$
By symmetry, we have $\omega(z_0,\R^+\cup\gamma^+_t,\H\setminus\gamma_t)\leq \omega(z_0,\partial D_R\cap \H,D_R)$ if $z_0\in D_R^-$. Since
$$\omega(z_0,\partial D_R\cap \H,D_R)\to 0\quad\text{as }R\to\infty$$ (this can be seen by an explicit computation) it follows that $\sin\theta_{T_R}\to 0$ as $R\to\infty$.
\end{proof}
Suppose $\gamma\in\mathcal X^C_{z_0}$ has finite $\rho$-Loewner energy. Then $T<\infty$ since $T=\infty$ implies that $\gamma$ is unbounded. We can conclude that $\gamma\in\mathcal X^C_{z_0}$ can have finite $\rho$-Loewner energy only if $T<\infty$ and $\gamma$ can be continuously extended by $\gamma(T)=z_0$: 
we can assume that the driving function $W$ can be continuously extended to $T$, for otherwise $I^C(\gamma)=\infty$ which implies $I^C_{\rho,z_0}(\gamma)=\infty$. This also means that $\gamma$ can be extended continuously to $T$ and must be simple on $[0,T]$ (otherwise it could not have finite chordal Loewner energy). By the definition of $T=\tau_{0+}$ and Loewner's theorem, we must have $\gamma(T)=z_0$. 
\begin{proposition}\label{prop:finiteenergy}
Fix $\rho<-4$. Let $\gamma:(0,T)\to\H\setminus\{z_0\}$ be a simple curve with $\gamma(0+)=0$ and with $\gamma(T-)=z_0$. Then 
$I^C_{\rho,z_0}(\gamma)$ is finite if and only if $I^C(\gamma)$ is finite. Moreover, if they are finite then $\sin\theta_t\to 0$ as $t\to T-$.
\end{proposition}  
\begin{proof}
From Corollary \ref{cor:infenergy} we already have that $I^C(\gamma)$ is finite if $I^C_{\rho,z_0}(\gamma)$ is finite. To show the other direction, assume $I^C(\gamma)$ is finite. Note that $g_t(\gamma_{[t,T]})-W_t$ is a curve from $0$ to $z_t-W_t$, with $\arg(z_t-W_t)=\theta_t$. Therefore, (\ref{eq:anglebound}) implies 
$I^C(g_t(\gamma_{[t,T]})-W_t)\geq -8\log\sin\theta_t$. The assumption $I^C(\gamma)<\infty$ gives 
$$I^C(g_t(\gamma_{[t,T]})-W_t)=\frac{1}{2}\int_t^T\dot W_t^2dt\to 0\text{ as } t\to T-.$$
It follows that $-\log\sin\theta_t\to 0$ as $t\to T-$ and by Proposition \ref{prop:radialbounds}, it follows that $I^C_{\rho,z_0}(\gamma)$ is finite. 
\end{proof}
\begin{remark}The statement that $I^C(\gamma)<\infty$ implies that $\sin\theta_t\to 0$ also appears in \cite[Lemma 3.2]{M23}.
\end{remark}
\begin{corollary}\label{cor:chordalradial}Let $\gamma$ be a simple curve from $1$ to $0$ in $\D$. Then
$$\frac{1}{4}I^{(\D;1,-1)}(\gamma)\leq I^R(\gamma)\leq I^{(\D;1,-1)}(\gamma).$$ \end{corollary}
\begin{proof}The corollary follows by changing coordinates to $\H$, applying Proposition \ref{prop:radialbounds} with $\rho=-6$, and $\theta_0=\pi/2$, and finally applying Proposition \ref{prop:finiteenergy}.
\end{proof}
\subsubsection{Proof of Theorem \ref{thm:radial}}
Recall that $\Sigma=\C\setminus\R^+$. Fix $z_0\in\H$ and $\rho>-4$.
Let $\gamma:(0,T)\to\H\setminus \{z_0\}$ be a simple curve with $\gamma(0+)=0,$ $\gamma(T-)=z_0$. Let $\tilde\gamma(t)=(\gamma(t))^2$ and for each $t\in[0,T]$, and let $h_t:\Sigma\setminus\tilde\gamma_t\to\Sigma$ defined by $h_t(z)=(g_t(\sqrt{z})-W_t)^2$. It follows from Proposition \ref{prop:finiteenergy} and (\ref{eq:chorDirichlet}) that
$$I^{(\Sigma;0,z_0^2)}_{-\rho-6,\infty}(\tilde\gamma)<\infty \iff I^{(\Sigma;0,\infty)}(\tilde\gamma)<\infty\iff \mathcal D(h)<\infty.$$
We now aim to show (\ref{eq:globalradial}). Assume, that $I^C_{\rho,z_0}(\gamma)<\infty$. Using the chain rule we obtain
$$\frac{|g_t'(z_0)|y_t}{y_0}=\frac{|h_t'(z_0^2)|\sin\theta_t}{\sin\theta_0},$$
so that (\ref{eq:rhoenergyinterior}) yields
\begin{equation}I^C_{\rho,z_0}(\gamma)=\lim_{t\to T-}\bigg(I^C(\gamma_t)-\frac{\rho^2}{8}\log\frac{\sin\theta_t}{\sin\theta_0}-\frac{\rho(8+\rho)}{8}\log |h'_t(z_0^2)|\bigg).\label{eq:altrep0}
\end{equation}
Proposition \ref{prop:finiteenergy} gives $I^C(\gamma)<\infty$ and $\lim_{t\to T-}\log\sin\theta_t=0$. Thus $\lim_{t\to T-}\log|h'_t(z_0^2)|$ must also exist and 
$$I^C_{\rho,z_0}(\gamma)=I^C(\gamma)+\frac{\rho^2}{8}\log\sin\theta_0-\frac{\rho(8+\rho)}{8}\lim_{t\to T-}\log |h'_t(z_0^2)|\bigg.$$ 
We define for $0\leq s\leq t\leq T$, $h_{s,t}:= h_{t}\circ h_{s}^{-1},$ and $w_t=(z_t-W_t)^2=h_t(z_0^2)$. 
\begin{lemma}\label{lemma:derivativeexists}For a curve $\gamma$ as above, we have that $\lim_{\delta\to 0+} |(h^{-1}_T)'(-\delta)|$ exists and 
\begin{equation}\label{eq:derivativeexists}\lim_{\delta\to 0+} |(h^{-1}_T)'(-\delta)|=\lim_{t\to T-}|(h^{-1}_t)'(w_t)|.\end{equation}
where $w_t=(z_t-W_t)^2.$
\end{lemma}
\begin{proof} 
\begin{figure}
\includegraphics[width=\textwidth]{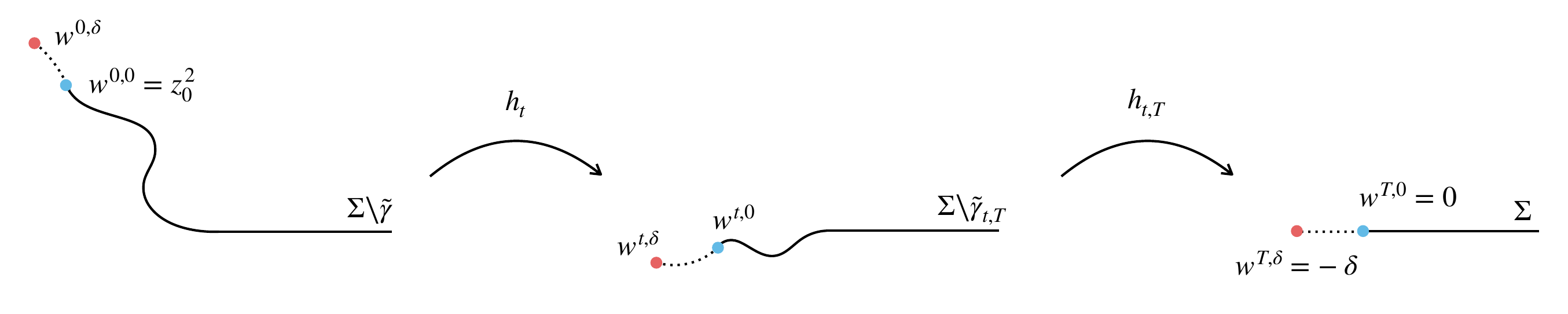}
\caption{This figure illustrates the set-up for the proof of Lemma \ref{lemma:derivativeexists}.}\label{fig:setupGlobalRadial}
\end{figure}
For every $t\in[0,T]$, $\delta\geq 0$, let $w^{t,\delta}=h_{t,T}^{-1}(-\delta)$, $g_{t,T}=g_{T}\circ g_t^{-1}$, and $z^{t,\delta}=g_{t,T}^{-1}(i\sqrt{\delta}+W_T)$ (see Figure \ref{fig:setupGlobalRadial}). By the same estimate as in the proof of Proposition \ref{prop:radialbounds} we have
$$\bigg(\frac{\sin\pi/2}{\sin\arg(z^{t,\delta}-W_t)}\bigg)^2 e^{-\frac{1}{2}I^C(\gamma_{t,T})}\leq|g_{t,T}'(z^{t,\delta})|\frac{\sqrt{\delta}}{\Im z^{t,\delta}}\leq 1,\ \forall t\in[0,T],\ \delta>0.$$ 
Therefore,
$$e^{-\frac{1}{2}I^C(\gamma_{t,T})}\leq |h'_{t,T}(w^{t,\delta})|\leq 1,\ \forall t\in[0,T],\ \delta>0,$$
and hence, $e^{-\frac{1}{2}I^C(\gamma_{t_0,T})}\leq |h'_{t,T}(w^{t,\delta})|\leq 1,$ for all $t\in(t_0,T)$ and $\delta>0$.
By the chain rule $|h_T'(w^{0,\delta})|= |h_t'(w^{0,\delta})||h_{t,T}'(w^{t,\delta})|$, and thus
$$e^{-\frac{1}{2}I^C(\gamma_{t_0,T})}|h_t'(w^{0,\delta})|\leq |h_T'(w^{0,\delta})|\leq |h_t'(w^{0,\delta})|,$$
for all $t\in(t_0,T)$ and $\delta>0$. Taking the limit $\delta\to 0+$ we obtain
\begin{align*}e^{-\frac{1}{2}I^C(\gamma_{t_0,T})}|h_t'(z_0^2)|\leq\liminf_{\delta\to 0+}|h_T'(w^{0,\delta})|\leq\limsup_{\delta\to 0+}|h_T'(w^{0,\delta})|\leq |h_t'(z_0^2)|,\end{align*}
since for each $t<T$, $z_0^2\in \Sigma\setminus\tilde\gamma_t$ so that $h'_t$ is continuous at $z_0^2$. Taking the limit $t\to T-$ followed by $t_0\to T-$ yields
\begin{align*}\lim_{t\to T-}|h_t'(z_0^2)|\leq \liminf_{\delta\to 0+}|h_T'(w^{0,\delta})|\leq\limsup_{\delta\to 0+}|h_T'(w^{0,\delta})|\leq \lim_{t\to T-}|h_t'(z_0^2)|.\end{align*}
Recall from the discussion above that the limit on the left- (and right-)hand side exists. This completes the proof.
\end{proof}
Let $\eta(s):=h_T^{-1}(-s)$, $s\in[0,\infty)$, be the hyperbolic geodesic from $z_0^2$ to $\infty$ in $\Sigma\setminus\tilde \gamma$. Since $$I^{(\Sigma;0,\infty)}(\tilde\gamma\cup\eta)=I^C(\gamma)<\infty,$$ $\tilde\gamma\cup\eta$ is asymptotically smooth (see \cite[Theorem 2.18]{W22}). In particular, this implies that $\lim_{s\to 0}\frac{|\eta(s)-z_0^2|}{\ell(\eta_s)}=1,$ where $\ell(\eta_s)$ denotes the length of $\eta_s$, and as a consequence
\begin{align*}|(h^{-1}_T)'(0)|_{\R^-}:=&\lim_{\delta\to 0+}\frac{|h^{-1}_T(-\delta)-h^{-1}_T(0)|}{\delta}=\lim_{\delta\to 0+}\frac{|\eta(\delta)-z_0^2|}{\delta}=\lim_{\delta\to 0+}\frac{\ell(\eta_\delta)}{\delta}\\=&\lim_{\delta\to 0+}\frac{1}{\delta}\int_0^\delta|(h_T^{-1})'(-s)|ds =\lim_{\delta\to 0+}|(h_T^{-1})'(-\delta)|,\end{align*}
so $|(h^{-1}_T)'(0)|_{\R^-}$ as well as $|h_T'(z_0^2)|_\eta:=\lim_{\delta\to 0+}|h_T(\eta(\delta))|/|\eta(\delta)-\eta(0)|$ exist and we finally deduce
\begin{equation}I^{(\Sigma;0,\infty)}_{\rho, z_0^2}(\tilde\gamma)=I^C_{\rho,z_0}(\gamma)=I^C(\gamma)+\frac{\rho^2}{8}\log\sin\theta_0-\frac{\rho(8+\rho)}{8}\log|h_T'(z_0^2)|_\eta.\label{eq:dirichletradialnonnormalized}\end{equation}
In particular this means that, if $\tilde\gamma^{0}$ denotes the SLE$_0(\rho)$ in $\Sigma$ from $0$, with reference point $\infty$ and force point $z_0^2$, then
$$I^{(\Sigma;0,\infty)}_{\rho,z_0}(\tilde\gamma)=I^{(\Sigma;0,\infty)}(\tilde\gamma)-I^{(\Sigma,0,\infty)}(\tilde\gamma^{0})-\frac{\rho(8+\rho)}{8}\log |H'(z_0^2)|_{\eta},$$
where $H:\Sigma\setminus\tilde\gamma\to\Sigma\setminus \tilde\gamma^{0}$ is the conformal map with $H(z_0^2)=z_0^2$, $H(\infty)=\infty$ and $|H'(\infty)|=1$. Using (\ref{eq:chorDirichlet}) and recalling that
$$I^{(\Sigma;0,z_0^2)}_{-6-\rho,\infty}(\tilde\gamma)=I^{(\Sigma;0,\infty)}_{\rho,z_0^2}(\tilde\gamma)$$
we obtain (\ref{eq:globalradial}). This finishes the proof of Theorem \ref{thm:radial}.
\subsection{Chordal setting}\label{section:globalchordal}
Throughout this section $\rho>-2$ and the force point $x_0>0$. Since $I^C_{\rho,x_0}(\gamma)=I^C_{\rho,-x_0}(-\overline\gamma)$ where $-\overline\gamma$ denotes the reflection of $\gamma$ in the imaginary axis, the assumption on $x_0$ can be imposed without loss of generality. Consider the class $\mathcal X^C_{x_0}$, of simple curves $\gamma:(0,T)\to\H$, with $\gamma(0+)=0$, and maximal in the sense that $T=\tau_{0+}:=\lim_{\vare\to 0+}\tau_{\vare}$ where $\tau_\vare=\inf\{t\in(0,T):|W_t-x_t|\leq \vare\}$. We call $\gamma\in\mathcal X^C_{x_0}$ a fully grown curve with respect to $x_0$ and $\infty$.\par
We have the following analog of Proposition \ref{prop:radialbounds}. As the proof is very similar we omit it.
\begin{proposition} \label{prop:chordalbounds} Suppose $\rho\in(-2,\infty)$ and suppose $\gamma:(0,T]\to\H$ is a simple curve with $\gamma(0)=0$. Then 
\begin{align*}
I^C_{\rho,x_0}(\gamma)& \leq\max(\tfrac{\rho+2}{2},1)\bigg(\max(\tfrac{\rho+2}{2},1)I^C(\gamma)+|\rho|\log\frac{|W_t-x_t|}{|x_0|}\bigg),\\
I^C_{\rho,x_0}(\gamma)& \geq \min(\tfrac{\rho+2}{2},1)\bigg(\min(\tfrac{\rho+2}{2},1)I^C(\gamma)-|\rho|\log\frac{|W_t-x_t|}{|x_0|}\bigg).
\end{align*}
\end{proposition}
This immediately shows that $\tau_{0+}<\infty$ implies that $I^C_{\rho,z_0}(\gamma)=\infty$. So, if $I^C_{\rho,z_0}(\gamma)<\infty$ for a $\gamma\in\mathcal X^C_{x_0}$, then $\gamma$ is unbounded. We saw in Corollary \ref{cor:rhoflow} that the SLE$_0(\rho)$ curve approaches $\infty$ with an angle $\alpha\pi$ where $\alpha=\alpha(\rho)=\frac{2+\rho}{4+\rho}$. For the remainder of this section $\alpha$ refers to $\alpha(\rho)$.
\begin{proposition}Fix $\rho>-2$ and $x_0>0$.  \label{proposition:angleapproach}Let $\gamma\in\mathcal X^C_{x_0}$ and suppose that $I^C_{\rho,x_0}(\gamma)<\infty$. Then for each $0<\alpha_-<\alpha(\rho)<\alpha_+<1$ there is an $R>0$ such that $\gamma\setminus B(0,R)\subset C(\alpha_-,\alpha_+)$ where $$C(\alpha_-,\alpha_+)=\{re^{i\theta}:r>0,\ \theta\in(\alpha_-\pi,\alpha_+\pi)\}.$$
\end{proposition}
\begin{remark}As a consequence of Proposition \ref{prop:chordalescape}, curves of finite $\rho$-Loewner energy are continuous at the end point. This means that we may strengthen the above: for each $\gamma$, $\alpha_-$, and $\alpha_+$ as above there exists a $T>0$ such that $\gamma_{[T,\infty)}\subset C(\alpha_-,\alpha_+)$. 
\end{remark} 
\begin{lemma}\label{lemma:welding}Fix $\rho>-2$ and $x_0>0$, and let $y_0=\frac{-2x_0}{2+\rho}$. Suppose $\gamma:(0,T]\to\H$ driven by $W_t$ and define $r_t:=\frac{W_t-y_t}{x_t-y_t}\in(0,1)$ where $x_t=g_t(x_0)$ and $y_t=g_t(y_0)$. Then,
\begin{equation}I^C_{\rho,x_0}(\gamma)\geq -(2+\rho)\log\frac{1-r_T}{1-r_0}-2\log \frac{r_T}{r_0},\label{eq:weldingineq}\end{equation}
where the right hand side is positive whenever $r_T\neq r_0=1-\alpha$. 
\end{lemma}
\begin{remark}The role of the point $y_0$ in Lemma \ref{lemma:welding} might seem artificial, but by viewing $\gamma$ as the (mapped out) continuation of an SLE$_0(\rho)$ with force point at $0+$ we see that is is quite natural: Recall from Remark \ref{remark:ray} that $\{re^{i\pi\alpha}:r\in[0,R]\}$ is an initial part of the SLE$_0(\rho)$ with force point at $0+$. There exists, for each $x_0>0$ a unique $R>0$ such that there is a (unique) conformal map $$\varphi_{x_0}:\H\setminus\{re^{i\pi\alpha}:r\in[0,R]\}\to\H$$ satisfying $\varphi_{x_0}(Re^{i\pi\alpha})=0$, $\varphi_{x_0}(\infty)=\infty$, $\varphi_{x_0}'(\infty)=1$, and $\varphi_{x_0}(0+)=x_0$. Then $\varphi_{x_0}(0-)=y_0$ (this can be checked by an explicit computation).\label{remark:y0}\end{remark}
\begin{remark}With Remark \ref{remark:y0} in mind, Lemma \ref{lemma:welding} is reminiscent of \cite[Theorem 4.4]{M23} where Mesikepp answers the question: For $X,Y>0$ fixed, what is the minimal chordal Loewner energy required for a curve to have $g_T(0+)-W_T=X$, $W_T-g_T(0-)=Y$? Mesikepp also finds that the optimal energy is attained for a unique curve corresponding to an SLE$_0(-4,-4)$ with force points at $(0-,0+)$. To fully see the analogy with the Lemma above we would have to define, in the natural way, the $\rho$-Loewner energy when the force point $x_0=0+$. If one does this, the above gives a lower bound on the minimal $\rho$-Loewner energy, with respect to $x_0=0+$, required to obtain $g_T(0+)-W_T=X$, $W_t-g_T(0-)=Y$. If one works a little more (e.g., by following the proof of \cite[Theorem 4.4(i)]{M23}), one finds that the bound above is optimal and that the unique curve of optimal energy is an SLE$_0(-4,-4-\rho)$, again with force points $(0-,0+)$. \label{remark:welding}
\end{remark}
\begin{proof}
We may suppose that $W_t$ is absolutely continuous, for otherwise (\ref{eq:weldingineq}) is trivial. We may also assume that $T=\inf\{t\in[0,T]:r_t=r_T\}$. Let $X_t=x_t-W_t$, $Y_t=W_t-y_t$, and note that they, as well as $r_t$, are  absolutely continuous with
$$\dot X_t = \frac{2}{X_t}-\dot W_t,\qquad \dot Y_t = \frac{2}{Y_t}+\dot W_t,\qquad \dot r_t =\frac{\dot W_t+\frac{2}{Y_t}-\frac{2}{X_t}}{X_t+Y_t},\quad \text{ a.e.}$$
Assume $r_T>r_0$. In this case, we may assume that $r_t\geq r_0$ for all $t\in[0,T]$. Let $$E=\{t\in[0,T]:r_t=\sup_{s<t}r_s\}.$$ Note that $\dot r_t\geq 0$ a.e. on $E$, that $E$ is closed, and that $0,T\in E$. Therefore, $[0,T]\setminus E$ consists of countably many disjoint open intervals $I_n=(\alpha_n,\beta_n)$ for which $r_{\alpha_n}=r_{\beta_n}$. We have
$$I^C_{\rho,x_0}(\gamma)\geq \frac{1}{2}\int_E\bigg ( \dot W_t+\frac{\rho}{X_t}\bigg)^2 \mathds 1_{\{\dot r_t>0\}} dt = \frac{1}{2}\int_E\frac{(\dot W_t+\frac{\rho}{X_t})^2}{\dot r_t} \dot r_t \mathds 1_{\{\dot r_t>0\}} dt.$$ 
Under the assumption $\dot r_t>0$ (and $r_t>r_0$), a computation shows that, 
$$\frac{(\dot W_t+\frac{\rho}{X_t})^2}{\dot r_t}\geq 4\bigg(\frac{2+\rho}{1-r_t}-\frac{2}{r_t}\bigg)$$ where equality is attained when $\dot W_t+\frac{\rho}{X_t}=2(\frac{2+\rho}{X_t}-\frac{2}{Y_t})$. Define $f(t)=-2(2+\rho)\log(1-r_t)-4\log r_t$ and observe that $f$ is absolutely continuous with $f'(t)=(2\frac{2+\rho}{1-r_t}-\frac{4}{r_t})\dot r_t$ a.e. Thus,
$$I^C_{\rho,x_0}(\gamma)\geq \int_E f'(t) dt = \int_0^T f'(t)dt$$
since $f(\beta_n)-f(\alpha_n)=\int_{I_n}f'(t)=0$ and $r_{\beta_n}=r_{\alpha_n}$ for all $n$. This shows (\ref{eq:weldingineq}) in the case $r_T>r_0$. The case $r_T<r_0$ is shown in the same way by replacing $r_t$ with $1-r_t$.
\end{proof}
\begin{lemma}\label{lemma:conebound}Let $R>0$ and $\varphi_{x_0}$ be as in Remark \ref{remark:y0}. For each $0<\alpha_-<\alpha<\alpha_+<1$ there is an $M>0$, such that for any unbounded simple curve $\gamma$ in $\H$ starting at $0$ 
$$\varphi_{x_0}^{-1}(\gamma)\not\subset C(\alpha_-,\alpha_+)\implies I^C_{\rho,z_0}(\gamma)\geq M.$$
\end{lemma}
\begin{proof}For $\beta\in(0,1)$ and $\tilde R>0$ and let $s_{\beta,\tilde R}=\{re^{i\pi\beta}:r\in[0,\tilde R]\}.$
Then
\begin{equation}
    \frac{\omega_\infty(s^-_{\beta,\tilde R},\H\setminus s_{\beta,\tilde R})}{\omega_\infty(s_{\beta,\tilde R},\H\setminus s_{\beta,\tilde R})}=1-\beta.\label{eq:harmonicmslit}
\end{equation}
This is easily checked using an (inverse) Schwarz-Christoffel map to map out the slit. Suppose $\gamma$ is an unbounded simple curve in $\H$ staring at $0$ and denote $\tilde\gamma:=\varphi_{x_0}^{-1}(\gamma)$. Suppose further that $\tilde\gamma\not\subset C(\alpha_-,\alpha_+)$ and let $T=\inf\{t:\tilde\gamma(t)\notin C(\alpha_-,\alpha_+)\}$. Assume that $\arg(\tilde\gamma(T))=\pi\alpha_-$. We claim that this implies, in the notation of Lemma \ref{lemma:welding}, that $\gamma$ has  $r_{T}\geq 1-\alpha_-.$ To prove this claim, it suffices to show that $\frac{1}{r_{T}}-1\leq \frac{\alpha_-}{1-\alpha_-}.$ First note that conformal covariance of $\omega_\infty$ gives
\begin{equation}\frac{1}{r_T}-1=\frac{x_T-W_T}{W_T-y_T}=\frac{\omega_\infty(\varphi_{x_0}^{-1}([W_T,x_T]),\varphi_{x_0}^{-1}(\H))}{\omega_\infty(\varphi_{x_0}^{-1}([y_t,W_T]),\varphi_{x_0}^{-1}(\H))}=\frac{\omega_\infty(s_{\alpha,R}^+\cup\tilde\gamma^+_T,\H\setminus(s_{\alpha,R}\cup\tilde\gamma_T))}{\omega_\infty(s_{\alpha,R}^-\cup\tilde\gamma^-_T,\H\setminus(s_{\alpha,R}\cup\tilde\gamma_T))}.\label{eq:harmonicmeasure}\end{equation}
Observe that $s_{\alpha,R}\cup\tilde\gamma_T\cup s_{\alpha_-,\tilde\gamma(T)}$ is a Jordan curve, see Figure \ref{fig:conesetup}.
\begin{figure}
\centering
\includegraphics[width=0.38\textwidth]{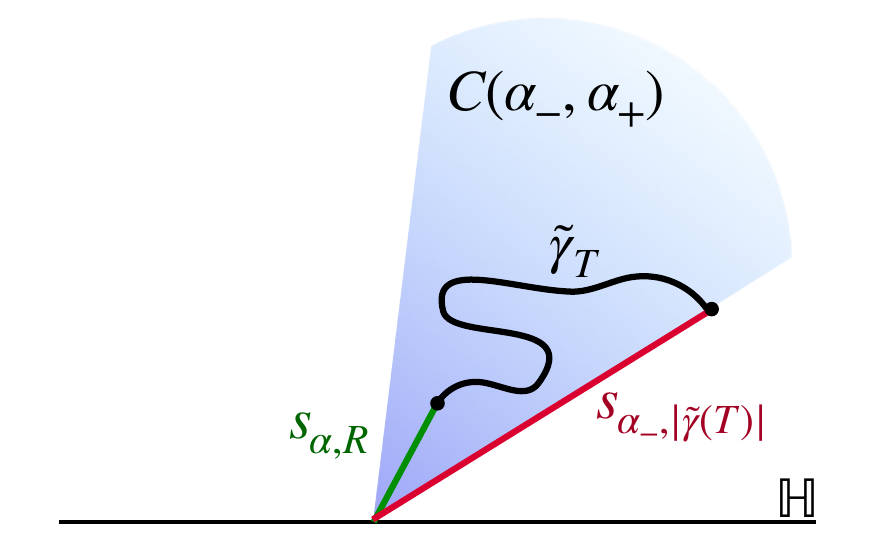}
\caption{Illustration of the set-up in Lemma \ref{lemma:conebound}.\label{fig:conesetup}}
\end{figure}
Monotonicity of harmonic measure then implies
\begin{align*}&\omega_\infty(s_{\alpha,R}^+\cup\tilde\gamma^+_T,\H\setminus(s_{\alpha,R}\cup\tilde\gamma_T))\leq\omega_{\infty}(s_{\alpha_-,\tilde\gamma(T)}^+,\H\setminus s_{\alpha_-,\tilde\gamma(T)}),\\
& \omega_\infty(s_{\alpha,R}^-\cup\tilde\gamma^-_T,\H\setminus(s_{\alpha,R}\cup\tilde\gamma_T))\geq\omega_{\infty}(s_{\alpha_-,\tilde\gamma(T)}^-,\H\setminus s_{\alpha_-,\tilde\gamma(T)}).\end{align*}
Inserting this into (\ref{eq:harmonicmeasure}) and using (\ref{eq:harmonicmslit}) yields $\frac{1}{r_T}-1 \leq \frac{\alpha_-}{1-\alpha_-}$
as desired. By a similar argument, one finds that $\arg(\tilde\gamma(T))=\pi\alpha_+$ implies $r_T\leq 1-\alpha_+$. Applying Lemma \ref{lemma:welding} shows the desired implication with
$$M=\min\Big\{-(2+\rho)\log\frac{\alpha_-}{\alpha}-2\log\frac{1-\alpha_-}{1-\alpha},-(2+\rho)\log\frac{\alpha_+}{\alpha}-2\log\frac{1-\alpha_+}{1-\alpha}\Big\}>0.$$
\end{proof}
\begin{proof}[Proof of Proposition \ref{proposition:angleapproach}]
Suppose $\gamma\in\mathcal X^C_{x_0}$ and that $I^C_{\rho,x_0}(\gamma)<\infty$. Suppose also, for contradiction, that there exists $\alpha_-$ and $\alpha_+$  for which the statement fails. Take $\vare>0$ sufficiently small so that $\tilde\alpha_-:=\alpha_-+\vare>0$, $\tilde\alpha_+:=\alpha_+-\vare<1$. Now, let $M=M(\tilde\alpha_-,\tilde\alpha_+)$ as in Lemma \ref{lemma:conebound}. Since $I^C_{\rho,x_0}(\gamma)<\infty$ there exists a $T>0$ such that $I^C_{\rho,x_T-W_T}(g_T(\gamma_{[T,\infty)})-W_T)<M.$ Therefore, $$\tilde\gamma:=\varphi_{x_T-W_T}^{-1}(g_T(\gamma_{[T,\infty)})-W_T)\subset C(\tilde\alpha_-,\tilde\alpha_+)$$
which implies that 
$$\gamma_{[T,\infty)}\subset g_T^{-1}(\varphi_{x_T-W_T}(C(\tilde\alpha_-,\tilde\alpha_+))+W_T).$$
Since both $\varphi_{x_T-W_T}$ and $g_T$ and are analytic at $\infty$ the rays $\{re^{i\pi\tilde\alpha_-},r>0\}$ and $\{re^{i\pi\tilde\alpha_+},r>0\}$ are mapped by $g_T^{-1}( \varphi_{x_T-W_T}-W_T)$ to smooth curves approaching infinity at the same angle. Hence, there exists an $R>0$ such that 
$$g_T^{-1}(\varphi_{x_T-W_T}(C(\tilde\alpha_-,\tilde\alpha_+))+W_T)\setminus B(0,R)\subset C(\alpha_-,\alpha_+).$$
Since, by choosing $R$ larger, we can ensure that $\gamma_{T}\subset B(0,R)$ this finishes the proof. 
\end{proof}
\subsubsection{Proof of Theorem \ref{thm:chordal}}
Let $\gamma:(0,\infty)\to\Sigma$ be a simple curve from $0$ to $\infty$. Let $h:\Sigma\setminus\gamma\to\Sigma\setminus\R^-$ be a map which is conformal on each component of $\Sigma\setminus\gamma$ mapping the upper (lower) component to $\H$ ($\H^\ast$) with $\infty\mapsto \infty$. For $\beta\in(0,1)$, we define
$$\mathcal D_\beta(\gamma):=\mathcal D_\beta(h)$$ 
whenever the right-hand side (defined in (\ref{eq:renormalizedDir})) exists. Since $|\nabla \log |h'(z)||=|h''(z)/h'(z)|$, and $h$ is unique up to post-composition by translation and scaling on each component, $\mathcal D_\beta(\gamma)$ is well-defined. We will write $\sigma_h := \log|h'|$ for conformal maps $h$. 
\begin{lemma}Fix $\rho>-2$ and $x_0>0$ and let $\gamma^{0}$ be the SLE$_0(\rho)$ from $0$ to $\infty$ in $\Sigma$ with force point $x_0^+$. Then, $\mathcal D_\alpha(\gamma^{\rho,x_0})$ exists and is finite.\label{lemma:direnergy}
\end{lemma}
\begin{proof}
By Remark \ref{remark:ray}, there exists a $\tilde R>0$ and a conformal map $\tilde h:\Sigma\setminus\{re^{i2\alpha\pi}:r\in[0,\tilde R]\}\to\Sigma$ satisfying $$\tilde h(\{re^{i2\alpha\pi}:r\in[\tilde R,\infty)\})=\gamma^{0}\quad \text{ and }\quad \sqrt{\tilde h(z^2)}=z+O(1)\text{ analytic at $\infty$}.$$ 
The second property gives
\begin{align}
|\tilde h(z)|&=|z|+O(|z|^{1/2}),\label{eq:dirbnd3}\\
\log|\tilde h'(Re^{i\theta})| &= O(|z|^{-1/2}),\label{eq:dirbnd2}\\
|\nabla\sigma_{\tilde h} (z)|&=O(|z|^{-3/2}).\label{eq:dirbnd1}
\end{align}
(For the same reason (\ref{eq:dirbnd3})-(\ref{eq:dirbnd1}) also hold for $\tilde h^{-1}$.) Further, let 
$$\hat h:\Sigma\setminus\{re^{i2\alpha\pi}:r\in\R^+\}\to\Sigma\setminus\R^{-}$$
$$\hat h(re^{i\theta}) = \begin{cases}r^{1/(2\alpha)}e^{i\theta/(2\alpha)},&\theta\in(0,2\alpha\pi),\\
-r^{1/(2-2\alpha)}e^{i(\theta-2\alpha\pi)/(2-2\alpha)}, &\theta\in(2\alpha\pi,2\pi),
\end{cases}$$
and set $h=\hat h\circ\tilde h^{-1}$. 
Then $\mathcal D_\alpha(\gamma^{0})=\mathcal D_\alpha(h)$. Since $\gamma^{0}_{[0,T]}$ has finite chordal Loewner energy for every $T<\infty$, we have
$$\int_{B(0,r)\setminus (\gamma^{0}\cup\R^+)}|\nabla\sigma_h(z)|^2dz^2<\infty$$
for every $r>0$. Hence, it suffices to show that
\begin{equation}\lim_{R\to\infty}\bigg(\frac{1}{\pi}\int_{A(r,R)\setminus(\gamma^{0}\cup\R^+)}|\nabla\sigma_h(z)|^2dz^2-c_\alpha\log R\bigg)\label{eq:dirlimit}
\end{equation} exists for some $r$, where $A(r,R)=\{r<|z|<R\}$. We write
$$|\nabla \sigma_h(z)|^2=|\nabla \sigma_{\tilde h^{-1}}(z)|^2+ 2\nabla\sigma_{\tilde h^{-1}}(z)\cdot\nabla\sigma_{\hat h}(\tilde h^{-1}(z)) +|\nabla\sigma_{\hat h}(\tilde h^{-1}(z))|^2,$$
and study the integral over these terms separately.
First of all,
$$\lim_{R\to\infty}\frac{1}{\pi}\int_{A(r,R)\setminus(\gamma^{0}\cup\R^+)}|\nabla\sigma_{\tilde h^{-1}}(z)|^2 dz^2$$
exists and is finite for sufficiently large $r>0$ by (\ref{eq:dirbnd1}). Second, by conformal invariance of the Dirichlet inner product and Stokes' theorem, we have
\begin{align*}\int_{A(r,R)\setminus(\gamma^{0}\cup\R^+)}\nabla\sigma_{\tilde h^{-1}}(z)\cdot \nabla\sigma_{\hat h}(\tilde h^{-1}(z)) dz^2
&=-\int_{\tilde h^{-1}(A(r,R))\setminus(\{re^{i2\alpha}\}\cup\R^+)} \nabla\sigma_{\tilde h}(z)\cdot\nabla\sigma_{\hat h}(z) dz^2\\
 &= -\int_{\Gamma^+_{R}\cup\Gamma^-_{R}} \sigma_{\tilde h}\partial_n\sigma_{\hat h}(z)d\ell,\end{align*}
 where $\Gamma^+_{R}$ and $\Gamma^-_{R}$ are the boundaries of the upper and lower connected components of $\tilde h^{-1}(A(r,R))\setminus(\{re^{i2\alpha}\}\cup\R^+)$ respectively. Note that $r$ can be chosen sufficiently large so that $\tilde h^{-1}$ is conformal on $\C\setminus (\R^+\cup B(0,r))$, and can be analytically extended across $\R^+$ (from both sides separately). This ensures that $\Gamma^\pm_R$ is piece-wise smooth, and that $\sigma_{\tilde h}$ and $\sigma_{\hat h}$ are smooth up to (and including) $\Gamma_R^\pm$, so Stokes' theorem may be used. Note that $\partial_n\sigma_{\hat h}$ vanishes along both sides of $\R^+$ and $\{re^{i2\alpha\pi}:r\in\R^+\}$. Thus, we are left with integrals along $\tilde h^{-1}(\partial B(0,r))$ and $\tilde h^{-1}(\partial B(0,R))$. We have 
$$|\nabla\sigma_{\hat h}(Re^{i\theta})|=\begin{cases}|\frac{1}{2\alpha}-1|\frac{1}{R},&\theta\in(0,2\alpha\pi),\\ |\frac{1}{2-2\alpha}-1|\frac{1}{R},&\theta\in(2\alpha\pi,2\pi), 
\end{cases}$$
and the length of $\tilde h^{-1}(\partial B(0,R))$ is $O(R)$. From this and (\ref{eq:dirbnd2}), we conclude that
$$\lim_{R\to\infty}\int_{A(r,R)\setminus(\gamma^{0}\cup\R^+)}\nabla \sigma_{\tilde h^{-1}}(z)\cdot \nabla\sigma_{\hat h}(\tilde h^{-1}(z)) dz^2=\int_{\tilde h^{-1}(\partial B(0,r))}\sigma_{\tilde h}(z)\partial_n\sigma_{\hat h}(z)d\ell.
$$
Finally, using the computation of $|\nabla\sigma_{\hat h}|$ and (\ref{eq:dirbnd3}) we obtain
$$\bigg(\int_{\tilde h^{-1}(A(r,R))\setminus(\{re^{i2\alpha}\}\cup\R^+)}-\int_{A(r,R)\setminus(\{re^{i2\alpha}\}\cup\R^+)}\bigg)|\nabla\sigma_{\hat h}(z)|^2 dz^2=O(1),\quad\text{as }R\to\infty $$
since the symmetric difference of the sets to be integrated over consists of a bounded part and a part contained in an annulus of radius $R$ and thickness $O(\sqrt{R}).$
Moreover, the second term is easily computed and equals 
$\pi c_\alpha\log\frac{R}{r}.$ We conclude that (\ref{eq:dirlimit}) exists.
\end{proof}
We are now ready to prove Theorem \ref{thm:chordal}. Let $\gamma\subset\Sigma$ be a simple curve from $0$ to $\infty$ such that $\gamma_{[T,\infty)}$ is the $\rho$-Loewner energy optimal continuation of $\gamma_T$. 
Let $h_T:\Sigma\setminus\gamma_T\to\Sigma$ be the conformal map with $h_T(\infty)=\infty$, $h_T(\gamma(T))=0$ and $|h_T'(\infty)|=1$. We may assume that $h_T(x_0^+)\geq x_0$ (for otherwise we may achieve this by increasing $T$). Under this assumption, there exists a $\tilde T\geq 0$ such that
$$\tilde h_{\tilde T}(x_0^+) =  h_T(x_0^+)=:x_T,$$
where $\tilde h_{\tilde T}:\Sigma\setminus\gamma_{\tilde T}^{0}\to\Sigma$ is the conformal map with the same normalization as $h_T$. Note that $H=\tilde h_{\tilde T}^{-1}\circ h_T$ and that (\ref{eq:rhoenergyboundary}) gives
$$I^{(\Sigma;0,\infty)}_{\rho,x_0^+}(\gamma) = I^{(\Sigma;0,\infty)}_{\rho,x_0^+}(\gamma_T)-I^{(\Sigma;0,\infty)}_{\rho,x_0^+}(\gamma^{0}_{\tilde T})=I^{(\Sigma;0,\infty)}(\gamma_T)-I^{(\Sigma;0,\infty)}(\gamma^{0}_{\tilde T})-\frac{\rho(4+\rho)}{4}\log|H'(x_0)|.$$
So, it only remains to show that $I^{\Sigma,0,\infty}(\gamma_T)-I^{\Sigma,0,\infty}(\gamma^{0}_{\tilde T})=\mathcal D_\alpha(\gamma)-\mathcal D_\alpha(\gamma^{0})$ and from (\ref{eq:chorDirichlet}) we already know that
$$I^{(\Sigma;0,\infty)}(\gamma_T) -I^{(\Sigma;0,\infty)}(\gamma^{0}_{\tilde T}) = \frac{1}{\pi}\int_{\Sigma\setminus\gamma}|\nabla \sigma_{h_T}(z)|^2 dz^2- \frac{1}{\pi}\int_{\Sigma\setminus\gamma^{0}}|\nabla \sigma_{\tilde h_{\tilde T}}(z)|^2 dz^2.$$
Let $\tilde\gamma^0:=\tilde h_{\tilde T}(\gamma^0_{[\tilde T,\infty)})= h_{T}(\gamma_{[T,\infty)})$ and let $h_{\tilde\gamma^{0}}:\Sigma\setminus\tilde\gamma^{0}\to\Sigma\setminus\R^-$ be a conformal map, fixing $\infty$ and mapping the components appropriately. See, Figure \ref{fig:chordalsetup}.
\begin{figure}
\includegraphics[width=\linewidth]{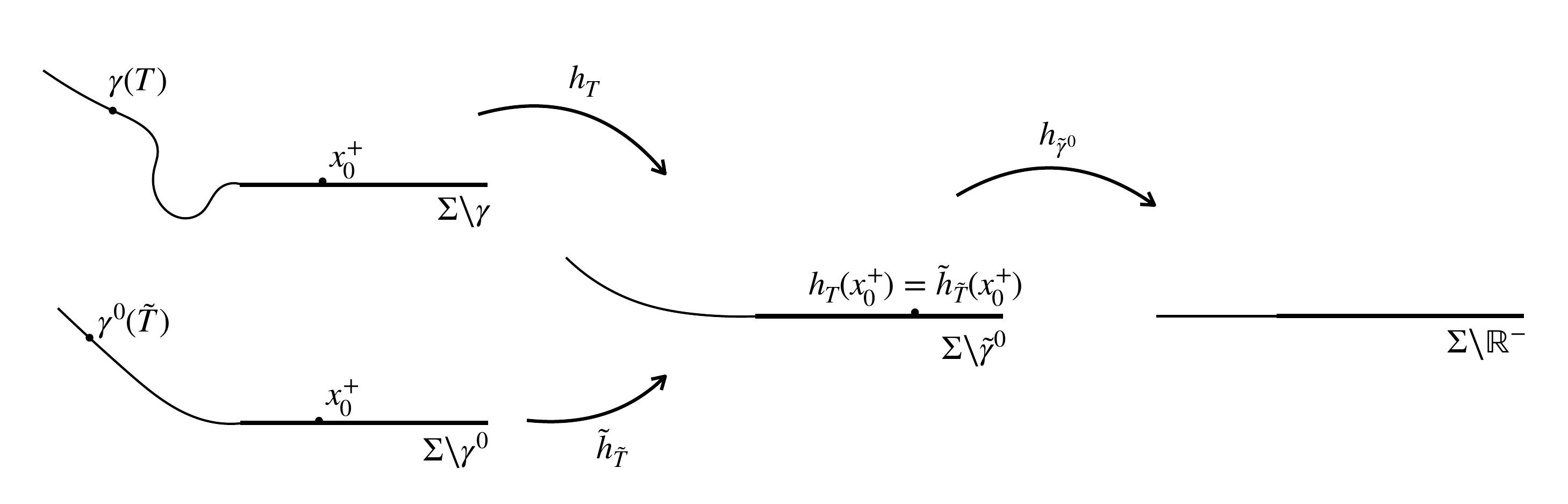}
\caption{Illustration of the conformal maps $h_{T}$, $\tilde h_{\tilde T}$, and $h_{\tilde\gamma^0}$ from the proof of Theorem \ref{thm:chordal}.\label{fig:chordalsetup}}
\end{figure}
Define $$h_{\gamma}:=h_{\tilde\gamma^{0}}\circ h_T\quad\text{ and }\quad h_{\gamma^{0}}:=h_{\tilde\gamma^{0}}\circ \tilde h_{\tilde T}$$ and note that 
$\mathcal D_\alpha(\gamma)=\mathcal D_\alpha(h_{\gamma})$ and $\mathcal D_\alpha(\gamma^{0})=\mathcal D_\alpha(h_{\gamma^{0}})$.
If $I^{(\Sigma,0,\infty)}(\gamma)=\infty$ we have that $\mathcal D_\alpha(\gamma)=\infty$, since for each sufficiently large $R$ we have that 
$$\int_{B(0,R)\setminus(\gamma\cup \R^+)}|\nabla \sigma_{h_T}(z)|^2 dz^2=\infty,\qquad \int_{B(0,R)\setminus(\gamma\cup \R^+)}|\nabla \sigma_{h_{\tilde\gamma^0}}(h_T(z))|^2 dz^2<\infty.$$
Hence, we may assume that $I^{(\Sigma,0,\infty)}(\gamma)<\infty$. As we already know that $\mathcal D_\alpha(\gamma^{0})$ exists and is finite (by Lemma \ref{lemma:direnergy}), we need only to check that
\begin{align}
\lim_{R\to\infty}&\int_{B(0,R)\setminus(\gamma\cup\R^+)}\nabla \sigma_{h_{\tilde\gamma^{0}}}(h_T)\cdot\nabla \sigma_{h_T}dz^2=0\label{eq:Dirichlet2}\\
\lim_{R\to\infty}&\int_{B(0,R)\setminus(\gamma^{0}\cup\R^+)}\nabla \sigma_{h_{\tilde\gamma^{0}}}(\tilde h_{\tilde T})\cdot\nabla \sigma_{\tilde h_{\tilde T}} dz^2=0\label{eq:Dirichlet3}\\
\lim_{R\to\infty}&\Bigg(\int_{B(0,R)\setminus(\gamma\cup\R^+)}|\nabla \sigma_{h_{\tilde\gamma^{0}}}(h_T)|^2dz^2 -\int_{B(0,R)\setminus(\gamma^{0}\cup\R^+)}|\nabla \sigma_{h_{\tilde\gamma^{0}}}( \tilde h_{\tilde T})|^2dz^2\Bigg)=0\label{eq:Dirichlet1}.\end{align}
We start by showing (\ref{eq:Dirichlet2}). Let $\phi\in C^\infty_c(\C)$ is some function with $\phi|_{h_T(B(0,R))}\equiv 1$. It then follows from \cite[Lemma 5.3]{W19a} and Stokes' theorem that
\begin{align*}\int_{B(0,R)\setminus(\gamma\cup\R^+)}\nabla \sigma_{h_{\tilde\gamma^{0}}}(h_T)\cdot\nabla \sigma_{h_T}dz^2&=-\int_{\C\setminus(B(0,R)\cup\gamma\cup\R^+)}\nabla (\phi\sigma_{h_{\tilde\gamma^{0}}})(h_T)\cdot\nabla \sigma_{h_T}dz^2\\
&=\int_{\partial B(0,R)}\sigma_{h_T}\partial_n \sigma_{h_{\tilde\gamma^{0}}}(h_T)d\ell.
\end{align*}
Here we are using that $\sigma_{h_T}$ and $\sigma_{\tilde\gamma^0}(h_T)$ are smooth on $\C\setminus(B(0,R)\cup\R^+)$ and that $\partial_n\sigma_{h_{\tilde\gamma^0}}$ vanishes along both sides of $\R^+$ since
$$\partial_n\sigma_{h_{\tilde\gamma^0}}(re^{i\theta})|_{\theta=0} =\partial_r\arg(h_{\tilde\gamma^0}'(r))=\partial_r(0) =0.$$
Further, since  $\sqrt{h_T(z^2)}=z+O(1)$ is analytic at $\infty$
$$|h_T(Re^{i\theta})|=R+O(1/\sqrt{R}),\qquad |h_T'(Re^{i\theta})|=1+O(1/\sqrt{R}),$$ as $R\to\infty$ and by the proof of Lemma \ref{lemma:direnergy} we have $|\nabla \sigma_{h_{\tilde\gamma^{0}}}(Re^{i\theta})|=O(1/R),$ as $R\to\infty$. Applying these estimates to the right-hand side above shows (\ref{eq:Dirichlet2}). Since (\ref{eq:Dirichlet3}) is a special case of (\ref{eq:Dirichlet2}), we have also shown (\ref{eq:Dirichlet3}). 
We move on to (\ref{eq:Dirichlet1}). Using conformal invariance of the Dirichlet inner product we get
\begin{align*}&\int_{B(0,R)\setminus(\gamma\cup\R^+)}|\nabla \sigma_{h_{\tilde\gamma^{0}}}(h_T(z))|^2dz^2
-\int_{B(0,R)\setminus(\gamma^{0}\cup\R^+))}|\nabla \sigma_{h_{\tilde\gamma^{0}}}(\tilde h_{\tilde T}(z))|^2dz^2\\
=&\bigg(\int_{h_T(B(0,R))\setminus(\tilde\gamma^{0}\cup \R^+)}-\int_{\tilde h_{\tilde T}(B(0,R))\setminus(\tilde\gamma^{0}\cup \R^+)}\bigg)|\nabla\sigma_{h_{\tilde\gamma^{0}}}(z)|^2 dz^2.\end{align*}
Since, $|h_T(Re^{i\theta})|=R+O(1/\sqrt{R})$ and $|\tilde h_{\tilde T}(Re^{i\theta})|=R+O(1/\sqrt{R})$, the symmetric difference $h_T(B(0,R))\triangle \tilde h_{\tilde T}(B(0,R))$ is contained in an annulus of thickness $O(1/\sqrt{R})$ and radius $R$. Combining this with the estimate $|\nabla \sigma_{h_{\tilde\gamma^{0}}}(Re^{i\theta})|=O(1/R)$, we conclude that (\ref{eq:Dirichlet1}) holds. This finishes the proof.
\section{\texorpdfstring{$\zeta$}{}-regularized determinants of Laplacians}\label{section:laplacians}
In this section, we prove Proposition \ref{prop:laplacians}. We will give a detailed proof of the chordal case and an outline of the proof of the radial case (the details in the radial case are almost identical to those in the chordal case). Recall from Section \ref{section:mainresults}, the construction of the smooth slit structures $(\overline\D,\varphi)$. We say that a Riemannian metric $g_0$ on $\D$ is a $(\overline\D,\varphi)$-smooth conformal metric if the following holds:
\begin{itemize}
    \item In the $z$-coordinate $g_0(z)=e^{2\sigma_0(z)}dz^2$, where $\sigma_0\in C^\infty(\overline\D\setminus\{a\}).$
    \item In the $w=\varphi(z)$-coordinate $g_0(\varphi^{-1}(w))=e^{2\hat\sigma_0(w)}dw^2,$ where $\hat\sigma_0\in C^\infty(\D\setminus\Gamma)$ and $\hat\sigma_0$ extends smoothly to both sides of $\Gamma$ separately.
\end{itemize}
Before presenting the proof of Proposition \ref{prop:laplacians} we prove two lemmas. 
\begin{lemma}\label{lemma:cpd}Let $a$, $b$, $c$, $\Gamma$, $\varphi$ be as in Proposition \ref{prop:laplacians}. Then, for each $(\overline \D,\varphi)$-smooth conformal metric $g_0$, $(\overline\D,g_0,(a),(2))$ is a curvilinear polygonal domain. Moreover, if $\gamma\in \hat{\mathcal X}$ then (the closure of) each of the components of $\D\setminus\gamma$ and $\D\setminus(\gamma\cup\eta)$, endowed with $g_0$ along with the appropriate $n$-tuples $(p_j)$ and $(\beta_j)$, are curvilinear polygonal domains. 
\end{lemma}
\begin{proof}
It is immediate from the construction of the smooth slit structure $(\overline\D,\varphi)$ that $g_0$ is a smooth Riemannian metric on $\overline\D\setminus\{a\}$. For the corner point $p_1=a$ we choose $\varphi_1=\varphi$. We check that (i-iv) of Definition \ref{def:cpd} hold. By construction of $\varphi$ (i) holds with $\beta_1=2$. As $g_0$ is a $(\overline \D,\varphi)$-smooth conformal metric (ii) is also immediate. 
Moreover, let $\Gamma_1$ and $\Gamma_2$ be smooth Jordan curves as in (iii) and (iv). Let $g=e^{2\hat\sigma(w)}dw^2$. We may assume that $\Gamma_1,\Gamma_2\subset\D$ and thus $\sigma_1=\hat\sigma$ is already defined on $V_{1,1}^\circ$ and $V_{1,2}^\circ$. Since $g_0$ is $(\overline\D,\varphi)$-smooth it follows that $\hat\sigma=\sigma_1\in C^\infty(V_{1,1})$ and $\hat\sigma=\sigma_1\in C^\infty(V_{1,2})$. Thus $(\overline\D,g_0,(a),(2))$ is a curvilinear polygonal domain.\par 
Let $\gamma\in\hat{\mathcal X}.$ Then, $\gamma$ is $(\overline\D,\varphi)$-smooth. Moreover, as $\eta$ is a hyperbolic geodesic it is smooth, at least away from its endpoints. If $b\in \D$ smoothness of $\eta$ at $b$ is forced by the definition of $\hat{\mathcal X}$ and smoothness at $c$ follows from the fact that a conformal map $\psi:\D\setminus\gamma\to \D$, with $\psi(b)=1$ and $\psi(c)=-1$, maps $\eta$ onto $[-1,1]$ and can be extended conformally to neighbourhood of $c$ using Schwarz reflection. If $b\in\partial\D$, then we may assume, by symmetry that $c$ lies on the counter-clockwise circular arc from $a$ to $b$. We consider the conformal map
$$\psi:\D\setminus\gamma_{[0,T]}\to\Sigma\setminus\{re^{\alpha\pi i}:r\in[0,1]\}$$
with $\alpha=\alpha(\rho)=\frac{2+\rho}{4+\rho}$, $\psi(\gamma(T))=e^{2\alpha\pi i}$, $\psi(b)=\infty$ and $\psi(c)=0+$, where $T$ is chosen so that $\gamma_{[T,\infty)}$ is the $\rho$-Loewner energy optimal continuation of $\gamma_T$. It follows from Remark \ref{remark:ray} that 
$$\psi(\gamma_{[T,\infty)})=\{re^{\alpha\pi i}:r\in[1,\infty)\},\quad \psi(\eta)=\{re^{\alpha\pi i/2}:r\in[0,\infty)\}.$$
Since $\psi$ can be extended to a neighborhood of $b$ (again using Schwarz reflection) we deduce that $\eta$ is smooth at $b$. Using the same type of argument as in the radial case one can show that $\eta$ is smooth at $c$.\par 
Since $(\D,g,(a),(2))$ is a curvilinear polygonal domain and $\gamma$ and $\eta$ are $(\overline\D,\varphi)$-smooth and intersect $\partial\D$ non-tangentially it follows that the components of $\D\setminus\gamma$ and $\D\setminus(\gamma\cup\eta)$ are curvilinear polygonal. 
\end{proof}
It follows from Lemma \ref{lemma:cpd} that, for each choice of $a$, $b$, $c$, $\varphi$, and $(\overline\D,\varphi)$-smooth conformal metric $g_0$, the $\rho$-Loewner potential, with respect to $g_0$, is defined for all curves $\gamma\in\hat{\mathcal X}$.
\begin{lemma}\label{lemma:smooth}
Let $a$, $b$, $c$, $\Gamma$, $\varphi$ be as in Proposition \ref{prop:laplacians}. Let $\gamma:(0,T]\to\D$, with $\gamma(0+)=a$ be a $(\overline \D,\varphi)$-smooth slit, smoothly attached at $a$. Let $g_0$ be a $(\overline\D,\varphi)$-smooth conformal metric. Then $(\overline{\D\setminus\gamma},g_0,(\gamma(T)),(2))$ is a curvilinear polygonal domain (here $\partial(\D\setminus\gamma)$ refers to the set of prime ends). If $\sigma\in C^\infty(\D\setminus\gamma)$ extends smoothly to both sides of $\gamma((0,T])$ and to $\partial\D\setminus\{a\}$, and $\sigma(\varphi^{-1})$ extends smoothly to both sides of $\varphi(\gamma)\cup\Gamma$ locally at $0$, then $\sigma\in C^\infty(\overline{\D\setminus\gamma},g_0,(\gamma(T)),(2))$. 
\end{lemma}
\begin{proof}The proof of Lemma \ref{lemma:cpd} shows that $(\overline{\D\setminus\gamma},g_0,(\gamma(T)),(2))$ is a curvilinear polygonal domain. The $z$-coordinate is smooth in a neigborhood of $p_1=\gamma(T)$. Since $\sigma$ extends smoothly to both sides of $\gamma$ it is immediate that $\sigma$ can be smoothly extended to any $V_{1,1},V_{1,2}\subset\D$ as in Definition \ref{def:smooth}. Since the boundary of $\D\setminus\gamma$ (treated as prime ends) is $(\overline\D,\varphi)$-smooth away from $\gamma(T)$, the inherited smooth structure on $\overline{\D\setminus\gamma}\setminus\{\gamma(T)\}$ is a smooth structure in the classical sense. Since $\sigma$ extends smoothly to both sides of $\gamma_{(0,T)}$ and to $\partial\D\setminus\{a\}$, and $\sigma(\varphi^{-1})$ extends smoothly (locally at $0$) to both sides of $\varphi(\gamma)\cup\Gamma$ we have that $\sigma\in C^\infty(\overline{\D\setminus\gamma}\setminus\{\gamma(T)\})$ with respect to the inherited smooth structure. This shows that $\sigma\in C^\infty(\overline{\D\setminus\gamma},g_0,(\gamma(T)),(2))$.
\end{proof}
We are now ready to prove Proposition \ref{prop:laplacians}. 
\begin{proof}[Proof of Proposition \ref{prop:laplacians}, chordal case]
Fix $\rho>-2$, $a$, $b\in\partial \D$, $c$, $\Gamma$, and $\varphi$ as in Proposition \ref{prop:laplacians}. Consider $\gamma_1,\gamma_2\in\hat{\mathcal X}$ and a $(\overline\D,\varphi)$-smooth conformal metric $g_0=e^{2\sigma_0}dz^2$. Since $\gamma_j$, $\eta_j=\eta(\gamma_j),$ $j=1,2$, are $(\overline\D,\varphi)$-smooth the $\rho$-Loewner potential of $\gamma_j$ with respect to $g_0$ is defined (see, Lemma \ref{lemma:cpd}).\par
Let $\psi:\D\to\Sigma$ be a conformal map with $\psi(a)=0$, $\psi(b)=\infty$. By symmetry, we may assume that $\psi(c)=x_0^+$, where $x_0^+$ is the upper prime end at $x_0>0$. For $j=1,2$, write $\tilde\gamma_j=\psi(\gamma_j)$ and $\tilde\eta_j=\psi(\eta_j)$, $j=1,2$, and let $T_j>0$ be such that $\gamma_j([T_j,\infty))$ is the $I^{(\D,a,b)}_{\rho,c}$-optimal continuation of $\gamma_j([0,T_j])$. Next, let
$$h:\Sigma\setminus\tilde\gamma_1([0,T_1])\to\Sigma\setminus\tilde\gamma_2([0,T_2])$$
be the conformal map with $h(\tilde\gamma_1(T_1))=h(\tilde\gamma_2(T_2))$, $h(\infty)=h(\infty)$, and $h(x_0^+)=x_0^+$. Observe that this has the effect that $\tilde\gamma_2([T_2,\infty))=h(\tilde\gamma_1([T_1,\infty)))$ and $\tilde\eta_2=h(\tilde\eta_1)$, where $\tilde\eta_j=\psi(\eta_j)$ and $\eta_j=\eta(\gamma_j)$. 
Hence, we have, for each component $D$ of $\D\setminus\gamma_1$ or $\D\setminus(\gamma_1\cup\eta_1)$, that
\begin{equation}\log\zdet\Delta_{((\psi^{-1}\circ h\circ\psi)(D),g_0)}=\log\zdet\Delta_{(D,(\psi^{-1}\circ h\circ\psi)^\ast g_0)}.\label{eq:confchange}\end{equation}
Write $H=\psi^{-1}\circ h\circ\psi$ and set $g=H^\ast g_0$. Then $g=e^{2\sigma}g_0$ with $\sigma=\sigma_0(H)+\sigma_H-\sigma_0.$ We wish to apply Theorem \ref{thm:polalv} on the components of $\D\setminus\gamma_1$ and $\D\setminus(\gamma_1\cup\eta_1)$, with the conformal change of metric $g=e^{2\sigma}g_0$. To do so we must, in light of Lemma \ref{lemma:smooth}, show that $\sigma$ extends smoothly to both sides of $\gamma_1((0,T_1])$ and to $\partial\D\setminus\{a\}$, and that $\sigma(\varphi^{-1})$ extends smoothly, locally at $0$, to both sides of $\varphi(\gamma)\cup\Gamma$.
We immediately have that $\sigma$ and extends smoothly to both sides of $\gamma_1((0,T_1])$ and to $\partial\D\setminus\{a\}$ with the possible exception of $\{H^{-1}(a+),H^{-1}(a-)\}$, since $\sigma_0$ is smooth on $\overline\D\setminus\{a\}$ and $H$ is smooth on $\overline\D\setminus\{a+,a-,H^{-1}(a+),H^{-1}(a-))\}$ (here we use Kellogg's theorem, see, e.g., \cite[Theorem II.4.3]{GM05}). Close to $\varphi(a+)$ we have
$$\sigma(\varphi^{-1}(w))=\sigma_0((H\circ\varphi^{-1})(w))+\sigma_{H\circ\varphi^{-1}}(w)-(\sigma_0(\varphi^{-1}(w))+\sigma_{\varphi^{-1}}(w))$$
where $H\circ\varphi^{-1}$ is a smooth conformal map (unless $H$ fixes $a+$) and $\hat\sigma_0=\sigma_0(\varphi^{-1})+\sigma_{\varphi^{-1}}$ is smooth since $g_0$ is smooth. Hence we have that $\sigma$ is $(\overline\D,\varphi)$-smooth at $a+$ (the case where $H$ fixes $a+$ is covered in a similar way). The same type of argument shows that $\sigma$ is $(\overline\D,\varphi)$-smooth at $a-$, $H^{-1}(a+)$, and $H^{-1}(a-)$.\par
Let $D_1$ denote the component of $\D\setminus\gamma_1$ containing $\eta_1$ and let $D_2$ denote the other component. Observe that the $\rho$-Loewner potential of $\gamma_1$ can be re-written as
$$\mathcal H^{(\D;a,b)}_{\rho,c}(\gamma_{1};g_0)=\mathcal H_{(\D,g_0)}(\gamma_1)+\frac{\rho(\rho+4)}{12}\mathcal H_{(D_1,g_0)}(\eta_1),$$
and similarly for $\gamma_2$.
We apply Theorem \ref{thm:polalv}, which along with (\ref{eq:confchange}) yields
\begin{align}\begin{split}12\Big(\mathcal H_{(D_1,g_0)}(\eta_2)-\mathcal H_{(H(D_1),g_0)}(\eta_1)\Big) & = -\Big(\frac{1}{\alpha}-\alpha\Big)\sigma(b) + 2\Big(\frac{2}{\alpha}-\frac{\alpha}{2}\Big)\sigma(b)+2\Big(2-\frac{1}{2}\Big)\sigma(c)\end{split}\label{eq:ChorLapl1}\end{align}
and 
\begin{align}\begin{split}12\Big(\mathcal H_{(\D;g_0)}(\gamma_2)-&\mathcal H_{(\D;g_0)}(\gamma_1)\Big)\\
=&\frac{1}{\pi}\int_{D_1\cup D_2}|\nabla_{g_0}\sigma|^2d\Vol_{g_0}+\frac{2}{\pi}\int_{D_1\cup D_2}\sigma K_{g_0} d\Vol_{g_0} + \frac{2}{\pi}\int_{\partial D_1\cup\partial D_2} \sigma k_{g_0}d\ell_{g_0}\\& +
\frac{3}{\pi}\int_{\partial D_1\cup\partial D_2} \partial_{n_{g_0}} \sigma d\ell_{g_0}+\Big(\frac{1}{\alpha}-\alpha\Big)\sigma(b)+\Big(\frac{1}{1-\alpha}-(1-\alpha)\Big)\sigma(b).\end{split}\label{eq:ChorLapl2} 
\end{align}
Note that
$$\int_{\partial D_1\cup\partial D_2}\partial_{n_{g_0}}\sigma d\ell_{g_0}=\int_{\partial H(D_1)\cup\partial H(D_2)} k_{g_0} d\ell_{g_0}-\int_{\partial D_1\cup\partial D_2}k_{g_0} d\ell_{g_0}=0,$$
since the contributions from integration along the two sides of $\gamma_1$ and $\gamma_2$ cancel.\par
Since we aim to relate the $\rho$-Loewner potential to the $\rho$-Loewner energy, we consider (\ref{eq:ChorLapl2}) in $\psi$-coordinates. Denote by $\tilde g_0=(\psi^{-1})^\ast g_0$, $\tilde g=e^{2\sigma(\psi^{-1})}\tilde g_0$, $\tilde D_1 = \psi(D_1)$, and $\tilde D_2=\psi(D_2)$. We have
$\tilde g_0=e^{\tilde\sigma_0(z)}dz^2$ where $\tilde\sigma_0 = \sigma_0(\psi^{-1})+\sigma_{\psi^{-1}}$ and $\sigma(\psi^{-1})=\tilde\sigma_0(h)+\sigma_h-\tilde\sigma_0$, so that
$$\int_{D_1\cup D_2}|\nabla_{g_0}\sigma|^2d\Vol_{g_0} = \int_{\tilde D_1\cup \tilde D_2}|\nabla_{\tilde g_0}(\tilde\sigma_0(h)+\sigma_h-\tilde\sigma_0)|^2d\Vol_{\tilde g_0}.$$
We wish to compute the right-hand side by expanding the square and studying each term separately. Note that $\tilde\sigma_0(h),$ $\sigma_h$ and $\tilde\sigma_0$ are not necessarily smooth at $$x\in\{0+,0-,h^{-1}(0+),h^{-1}(0-),\infty\}$$ when considered separately. Let $D_\vare=\tilde D_1\cup\tilde D_2\setminus\cup B_{x,\vare}$ where $B_{x,\vare}=B_{(\tilde D_1\cup\tilde D_2,\tilde g_0)}(x,\vare)$ if $x\in\{0+,0-,\infty\}$ and $B_{h^{-1}(0\pm),\vare}=h^{-1}(B_{0\pm,\vare})$. (Here we assume that $0+$ and $0-$ are not fixed by $h$. The other case can be treated in a similar way.) By applying Stokes' theorem on $D_\vare$ and using (\ref{eq:transrulemetric}) we find
\begin{align*}\int_{D_\vare}\nabla_{\tilde g_0}\sigma_h\cdot\nabla_{\tilde g_0}\tilde\sigma_0(h) d\Vol_{\tilde g_0} &= \int_{\partial h(D_\vare)}\tilde \sigma_0 kd\ell - \int_{\partial D_\vare}\tilde \sigma_0(h) k d\ell,\\
\int_{D_\vare}\nabla_{\tilde g_0} \sigma_h\cdot\nabla_{\tilde g_0}\tilde \sigma_0 d\Vol_{\tilde g_0}&= \int_{\partial D_\vare}\sigma_h\partial_{n}\tilde\sigma_0 d\ell+\int_{D_\vare}\sigma_h K_{\tilde g_0} d\Vol_{\tilde g_0},\\
\int_{D_\vare}\nabla_{\tilde g_0}\tilde \sigma_0(h)\cdot\nabla_{\tilde g_0}\tilde \sigma_0(h) d\Vol_{\tilde g_0} &= \int_{\partial h(D_\vare)}\tilde \sigma_0\partial_n\tilde \sigma_0 d\ell+\int_{h(D_\vare)}\tilde \sigma_0  K_{\tilde g_0} d\Vol_{\tilde g_0},\\
\int_{D_\vare}\nabla_{\tilde g_0}\tilde \sigma_0\cdot\nabla_{\tilde g_0}\tilde \sigma_0 d\Vol_{\tilde g_0} &= \int_{\partial D_\vare}\tilde \sigma_0\partial_n\tilde \sigma_0 d\ell+\int_{ D_\vare}\tilde \sigma_0 K_{\tilde g_0} d\Vol_{\tilde g_0},\\
\int_{D_\vare}\nabla_{\tilde g_0}\tilde \sigma_0(h)\cdot\nabla_{\tilde g_0}\tilde \sigma_0 d\Vol_{\tilde g_0} &= \int_{\partial D_\vare}\tilde \sigma_0(h)\partial_n\tilde \sigma_0 d\ell+\int_{ D_\vare}\tilde \sigma_0(h)  K_{\tilde g_0} d\Vol_{\tilde g_0},\end{align*}
(here we consider $\partial D_\vare$ in terms of prime ends).
By combining the above and 
$$\int_{\partial D_\vare}\sigma(\psi^{-1})\partial_n\tilde \sigma_0 d\ell  = \int_{\partial D_\vare} \sigma(\psi^{-1}) k_{\tilde g_0}d\ell_{\tilde g_0} - \int_{\partial D_\vare}\sigma k d\ell,$$
we find that
\begin{align}\begin{split}\int_{D_\vare}|\nabla_{\tilde g_0}\sigma(\psi^{-1})|^2d\Vol_{\tilde g_0} 
=& \int_{D_\vare}|\nabla \sigma_h|^2dz^2 
+2 \int_{\partial D_\vare}\sigma_h k d\ell-2\int_{ D_\vare}\sigma(\psi^{-1}) K_{\tilde g_0} d\Vol_{\tilde g_0}
\\
&-2\int_{\partial D_\vare} \sigma(\psi^{-1}) k_{\tilde g_0}d\ell_{\tilde g_0} +\bigg(\int_{h(D_\vare)}-\int_{D_\vare}\bigg) \tilde \sigma_0 K_{\tilde g_0} d\Vol_{\tilde g_0}
\\
&
+\bigg(\int_{\partial h(D_\vare)}-\int_{\partial D_\vare}\bigg)\tilde \sigma_0k_{\tilde g_0}d\ell_{\tilde g_0} + \bigg(\int_{\partial h(D_\vare)}-\int_{\partial D_\vare}\bigg)\tilde \sigma_0 kd\ell.\end{split}\label{eq:chorlapl3}
\end{align}
We now study the limit as $\vare\to 0+$.
First of all, 
\begin{align*}\int_{D_\vare}|\nabla_{\tilde g_0}\sigma(\psi^{-1})|^2d\Vol_{\tilde g_0}&\to\int_{\tilde D_1\cup\tilde D_2}|\nabla_{\tilde g_0}\sigma|^2d\Vol_{\tilde g_0},\\
\int_{D_\vare}|\nabla\sigma_h|^2dz^2&\to\int_{\tilde D_1\cup\tilde D_2}|\nabla\sigma_h|^2dz^2,\\
\int_{D_\vare}\sigma K_{\tilde g_0}d\Vol_{\tilde g_0}&\to\int_{\tilde D_1\cup\tilde D_2}\sigma K_{\tilde g_0}d\Vol_{\tilde g_0}.\end{align*}
Next, observe that
\begin{equation}\psi^{-1}(z)=b+Cz^{-1/2}+o(z^{-1/2}),\quad (\psi^{-1})'(z)=-\frac{C}{2}z^{-3/2}+o(z^{-3/2}),\quad\text{ as }z\in\overline\Sigma\to\infty\label{eq:psiasymp}\end{equation}
for some $C\neq 0$. Similarly, \cite[Theorem 3]{W65} implies that
$$\varphi\circ\psi^{-1}(z)=Dz+o(z),\quad (\varphi\circ\psi^{-1})'(z)=D+o(1),\quad\text{ as }z\in\overline\Sigma\to 0,$$
for some $D\neq 0$.
Moreover $\psi^{-1}$ is smooth at $h^{-1}(0+)$, and $ h^{-1}(0-)$. Hence,
$$\lim_{\vare\to 0+}\bigg(\int_{h(D_\vare)}-\int_{D_\vare}\bigg)\tilde\sigma_0 K_{\tilde g_0} d\Vol_{\tilde g_0}=0.$$
For each $x\in\{0+,0-,h^{-1}(0+),h^{-1}(0-),\infty\}$, let $C_{x,\vare}= \partial B_{x,\vare}\cap(\tilde D_1\cup\tilde D_2)$ and $L_{x,\vare}=\partial B_{x,\vare}\setminus C_{x,\vare}.$
Then
$$\lim_{\vare\to0+}\bigg(\int_{\partial D_\vare}-\int_{\partial \tilde D_1\cup\partial \tilde D_2}\bigg)\sigma(\psi^{-1}) k_{\tilde g_0}d\ell_{\tilde g_0} =\lim_{\vare\to 0+}\int_{\cup C_{x,\vare}}\sigma(\psi^{-1}) k_{\tilde g_0}d\ell_{\tilde g_0},$$
and the same holds when replacing $\sigma(\psi^{-1})k_{\tilde g_0}d\ell_{\tilde g_0}$ with $\sigma_h kd\ell.$
Similarly,
$$\lim_{\vare\to 0+}\bigg(\int_{\partial h(D_\vare)}-\int_{\partial D_\vare}\bigg)\tilde\sigma_0 k_{\tilde g_0} d\ell_{\tilde g_0} = \lim_{\vare\to 0+}\bigg(\int_{\cup h(C_{x,\vare})}-\int_{\cup C_{x,\vare}}\bigg)\tilde\sigma_0 k_{\tilde g_0} d\ell_{\tilde g_0},$$
which also holds when replacing $k_{\tilde g_0} d\ell_{\tilde g_0}$ with $k d\ell$.
To compute the integrals over $C_{x,\vare}$ we note that $\ell_{\tilde g_0}(C_{x,\vare})\sim \pi\vare$ (since in the limit $C_{x,\vare}$ is a semicircle of $\tilde g_0$-radius $\vare$) and $k_{\tilde g_0}\sim-\frac{1}{\vare}$ (since the semicircle is traversed clockwise). We therefore obtain
\begin{align*}\frac{1}{\pi}\int_{\cup_x C_{x,\vare}}\sigma(\psi^{-1}) k_{g_0}d\ell_{g_0}\to & -\sigma(\psi^{-1}(\infty))-\sigma(\psi^{-1}(0+)) -\sigma(\psi^{-1}(0-))\\& -\sigma(\psi^{-1}(h^{-1}(0+))) -\sigma(\psi^{-1}(h^{-1}(0-))),\end{align*}
as $\vare\to 0+$.
In a similar manner
\begin{align*}
\frac{1}{\pi}\int_{\cup_x C_{x,\vare}}\sigma_h kd\ell\to & -2\sigma_h(\infty)-\sigma_h(0+)-\sigma_h(0-)\\ &-\sigma_h(h^{-1}(0+))-\sigma_h(h^{-1}(0-)),\\
\frac{1}{\pi}\bigg(\int_{\cup h(C_{x,\vare})}-\int_{\cup C_{x,\vare}}\bigg)\tilde\sigma_0 k_{\tilde g_0}d\ell_{\tilde g_0} \to &  -\tfrac{3}{2}\sigma_h(\infty)-\tilde\sigma_0(h(0+))-\tilde\sigma_0(h(0-))\\ &+\tilde\sigma_0(h^{-1}(0+))+\tilde\sigma_0(h^{-1}(0-)),\\
\frac{1}{\pi}\bigg(\int_{\cup h(C_{x,\vare})}-\int_{\cup C_{x,\vare}}\bigg)\tilde \sigma_0 k d\ell \to & 3\sigma_{h}(\infty)-\tilde\sigma_0(h(0+))-\tilde\sigma_0(h(0-))\\ &+\tilde\sigma_0(h^{-1}(0+))+\tilde\sigma_0(h^{-1}(0-)),\end{align*}
where we use the convention $\sigma_h(\infty)=\log|(1/h(1/z))'|_{z=0}$. In the bottom two limits we have used that $h(C_{h^{-1}(0\pm),\vare})=C_{0\pm,\vare}$ and that $\tilde\sigma_0$ is smooth at $h^{-1}(0\pm)$ and $h(0\pm)$. 
Combining the computed limits with (\ref{eq:ChorLapl2}) and (\ref{eq:chorlapl3}) yields
\begin{align}\begin{split}12\Big(\mathcal H_{(\D;g_0)}&(\gamma_2)-\mathcal H_{(\D;g_0)}(\gamma_1)\Big)\\=&\frac{1}{\pi}\int_{\tilde D_1\cup\tilde D_2}|\nabla \sigma_h|^2dz^2 +\frac{2}{\pi}\int_{\partial\tilde D_1\cup\partial\tilde D_2}\sigma_h kd\ell+\frac{(2\alpha-1)^2}{2\alpha(1-\alpha)}\sigma_h(\infty)\end{split}\label{eq:chorlapl4}\end{align}
where we have used that $\sigma(\psi^{-1}(\infty))=\sigma_h(\infty)/2$ which is found by using (\ref{eq:psiasymp}). Since $\tilde\sigma_0$ is smooth at $x_0$ we find that $\sigma(c)=\sigma_h(x_0)$. Then, (\ref{eq:ChorLapl1}) and (\ref{eq:chorlapl4}) imply
\begin{align*}12\Big(&\mathcal H^{(\D;a,b)}_{\rho,c}(\gamma_2;g_0)-\mathcal H^{(\D;a,b)}_{\rho,c}(\gamma_1;g_0)\Big)\\ &=\frac{1}{\pi}\int_{\tilde D_1\cup\tilde D_2}|\nabla \sigma_h|^2dz^2 +\frac{2}{\pi}\int_{\partial\tilde D_1 \cup\partial\tilde D_2}\sigma_h kd\ell+\frac{\rho(\rho+4)}{4}\sigma_h(x_0)+\frac{\rho(8+\rho)}{8}\sigma_h(\infty).
\end{align*}
Now, for the final step, let $h_1:\Sigma\setminus\tilde\gamma_1([0,T_1])\to\Sigma$ with $h_1(\tilde\gamma_1(T_1))=0$, $h_1(\infty)=\infty$ and $h_1'(\infty)=1$. Let $h_2:=h_1\circ h^{-1}:\Sigma\setminus\tilde\gamma([0,T_2])\to\Sigma$. We may assume that $h_2'(\infty)=1$ since this can be achieved by increasing $T_1$ or $T_2$ by an appropriate amount. Then it follows, from the Dirichlet energy formula (\ref{eq:chorDirichlet}) and (\ref{eq:rhoenergyboundary}), that
\begin{align*}I^{(\D;a,b)}_{\rho,c}(\gamma_2)-I^{(\D;a,b)}_{\rho,c}(\gamma_1)=&I^{(\Sigma;0,\infty)}_{\rho,x_0^+}(\tilde\gamma_2)-I^{(\Sigma;0,\infty)}_{\rho,x_0^+}(\tilde\gamma_1)\\
=&\frac{1}{\pi}\int_{\Sigma}|\nabla\sigma_{h_2^{-1}}|^2dz^2-\frac{1}{\pi}\int_{\Sigma}|\nabla\sigma_{h_1^{-1}}|^2dz^2+\frac{\rho(\rho+4)}{4}\log|h'(x_0)|.
\end{align*}
Since, $\sigma_h(\infty)=0$ by assumption, it only remains to show that
\begin{equation}
\int_{\tilde D_1\cup\tilde D_2}|\nabla \sigma_h|^2dz^2 +\int_{\partial\tilde D_1\cup\partial\tilde D_2}\sigma_h kd\ell = \int_{\Sigma}|\nabla\sigma_{h_2^{-1}}|^2dz^2-\int_{\Sigma}|\nabla\sigma_{h_1^{-1}}|^2dz^2\label{eq:chorlapl5}
\end{equation}
Write $\sigma_{h}=\sigma_{h_1}+\sigma_{h_2^{-1}}(h_1)$. We have,
$$\int_{\tilde D_1\cup\tilde D_2}|\nabla\sigma_h|^2dz^2 = \int_{\tilde D_1\cup\tilde D_2}\Big(|\nabla\sigma_{h_2^{-1}}(h_1)|^2-|\nabla\sigma_{h_1}|^2+2\nabla\sigma_{h_1}\cdot\nabla\sigma_h\Big)dz^2.$$
By conformal invariance of the Dirichlet inner product, the 
first two terms on the right-hand side equal the right-hand side of (\ref{eq:chorlapl5}). We apply Stokes' theorem to the third term of the right-hand side and use the change of variable formula for the geodesic curvature (\ref{eq:transrulemetric}), yielding
$$2\int_{\tilde D_1\cup\tilde D_2}\nabla\sigma_{h_1}\cdot\nabla\sigma_hdz^2=2\int_{\partial\tilde D_1\cup\partial\tilde D_2}\sigma_h\partial_n\sigma_{h_1}d\ell = 2\int_{\partial \Sigma}\sigma_h(h_1^{-1}) kd\ell-  2\int_{\partial\tilde D_1\cup\partial\tilde D_2}\sigma_h kd\ell.$$
Since $k=0$ along $\partial\Sigma$, this shows (\ref{eq:chorlapl5}) and finishes the proof.
\end{proof}
\begin{proof}[Proof of Proposition \ref{prop:laplacians}, radial case]
The proof of the radial case of Proposition \ref{prop:laplacians} is very similar. We therefore give an outline of the proof and refer the reader to the chordal proof for details. Fix $\rho>-2$, $a$, $b\in\D$, $c$, $\Gamma$, and $\varphi$ as in Proposition \ref{prop:laplacians}. Fix $\gamma_1,\gamma_2\in\hat{\mathcal X}$ and a $(\overline\D,\varphi)$-smooth conformal metric $g_0=e^{2\sigma_0}dz^2$. Since $\gamma_j$ and $\eta_j=\eta(\gamma_j),$ $j=1,2$, are $(\overline\D,\varphi)$-smooth the $\rho$-Loewner potential of $\gamma_j$ with respect to $g_0$ is defined.\par
As in the chordal case we fix a conformal map $\psi:\D\to\Sigma$, $\psi(a)=0$, $\psi(c)=\infty$ and write $\tilde\gamma_j=\psi(\gamma_j)$ and $\tilde\eta_j=\psi(\eta_j)$, $j=1,2$. Denote by $z_0=\psi(b)$. We let $h:\Sigma\setminus\tilde\gamma_1\to\Sigma\setminus\tilde\gamma_2$ be the conformal map with $h(z_0)=z_0$, $h(\infty)=\infty$, and $h'(\infty)=1$. Then $\tilde\eta_2=h(\tilde\eta_1)$ so that (\ref{eq:confchange}) holds if $D$ is $\D\setminus\gamma_1$ or a component of $\D\setminus(\gamma_1\cup\eta_1)$. We introduce a new metric $g=H^\ast g_0=e^{2\sigma}g_0$ with $H=\psi^{-1}\circ h\circ \psi$ and $\sigma=\sigma_0(H)+\sigma_H-\sigma_0$. By the same type of argument as in the chordal case, we find that $\sigma$ is $(\overline\D,\varphi)$-smooth. We have
$$\mathcal H^{(\D;a,b)}_{\rho,c}(\gamma_j;g_0)=\mathcal H_{(\D;g_0)}(\gamma_j)+\frac{\rho(\rho+4)}{12}\mathcal H_{(\D\setminus\gamma_j,g_0)}(\eta_j)$$
for $j=1,2$. Theorem \ref{thm:polalv} implies
$$12\Big(\mathcal H_{(\D\setminus\gamma_2,g_0)}(\eta_2)-\mathcal H_{(\D\setminus\gamma_1,g_0)}(\eta_1)\Big)=-\bigg(\frac{1}{2}-2\bigg)\sigma(b)+2\bigg(2-\frac{1}{2}\bigg)\sigma(c),$$
and
\begin{align*}12\Big(\mathcal H_{(\D\setminus\gamma_2,g_0)}(\eta_2)-\mathcal H_{(\D\setminus\gamma_1,g_0)}(\eta_1)\Big)=&\frac{1}{\pi}\int_{\D\setminus\gamma_1}|\nabla_{g_0}\sigma|^2d\Vol_{g_0}+\frac{2}{\pi}\int_{\D\setminus\gamma_1}\sigma K_{g_0}d\Vol_{g_0}\\&+\frac{2}{\pi}\int_{\partial(\D\setminus\gamma_1)}\sigma k_{g_0}d\ell_{g_0}+\frac{3}{\pi}\int_{\partial(\D\setminus\gamma_1)}\partial_{n_{g_0}}\sigma d\ell_{g_0} -\frac{3}{2}\sigma(b).\end{align*}
By carrying out the same type of computation as in the chordal case (the computation is almost identical) one finds
$$12\Big(\mathcal H_{(\D\setminus\gamma_2,g_0)}(\eta_2)-\mathcal H_{(\D\setminus\gamma_1,g_0)}(\eta_1)\Big) = \frac{1}{\pi}\int_{\Sigma\setminus\tilde\gamma_1}|\nabla\sigma_h|^2dz^2+\frac{2}{\pi}\int_{\partial\Sigma\setminus\tilde\gamma_1}\sigma_h kd\ell-\frac{3}{2}\sigma_h(\infty)-\frac{3}{2}\sigma_h(z_0).$$
A computation shows that $\sigma_h(\infty)=\sigma(c)/2$, but by the normalization of $h$, $\sigma_h(\infty)=0$. We also have $\sigma_h(z_0)=\sigma(b)$. Combining the above yields,
$$12\Big(\mathcal H^{(\D;a,b)}_{\rho,c}(\gamma_2;g_0)-\mathcal H^{(\D;a,b)}_{\rho,c}(\gamma_2;g_0)\Big)=\frac{1}{\pi}\int_{\Sigma\setminus\tilde\gamma_1}|\nabla\sigma_h|^2dz^2+\frac{2}{\pi}\int_{\partial\Sigma\setminus\tilde\gamma_1}\sigma_h kd\ell +\frac{(\rho+6)(\rho-2)}{8}\sigma_h(z_0).$$
Finally, by letting $h_1:\Sigma\setminus\tilde\gamma_1\to \Sigma$ be the conformal map with $h_1(z_0)=0$, $h_1(\infty)=\infty$, and $h_1'(\infty)=1$ and setting $h_2=h_1\circ h^{-1}$ we find 
\begin{align*}12\Big(\mathcal H^{(\D;a,b)}_{\rho,c}&(\gamma_2;g_0)-\mathcal H^{(\D;a,b)}_{\rho,c}(\gamma_2;g_0)\Big)\\=&\frac{1}{\pi}\int_{\Sigma\setminus\tilde\gamma_2}|\nabla\sigma_{h_2}|^2dz^2-\frac{1}{\pi}\int_{\Sigma\setminus\tilde\gamma_1}|\nabla\sigma_{h_1}|^2dz^2 -\frac{(\rho+6)(\rho-2)}{8}\log\frac{|h_2'(z_0)|}{|h_1'(z_0)|}\\=&I^{(\Sigma;0,z_0)}_{\rho,\infty}(\tilde \gamma_2)-I^{(\Sigma;0,z_0)}_{\rho,\infty}(\tilde \gamma_1) = I^{(\D;a,b)}_{\rho,c}( \gamma_2)-I^{(\D;a,b)}_{\rho,c}( \gamma_1),\end{align*}
where Stokes' theorem is used in the first equality and Theorem \ref{thm:radial} is used in the second. 
\end{proof}
\begin{remark}If $b\in\partial\D$ we can guarantee that the $I^{(\D;a,b)}_{\rho,c}$-minimizer, denoted by $\gamma^0$, is in $\hat{\mathcal X}$ by choosing $\varphi$ in the following way. We may assume $c$ lies on the counter-clockwise circular arc from $a$ to $b$ (the other case is covered by symmetry). Consider $\Sigma_{\alpha}=\Sigma\setminus\{re^{2\alpha\pi i}:r\in[0,1]\}$ and let $\tilde\varphi:\D\to\Sigma_\alpha$ be the conformal map with $\tilde\varphi(a)=e^{2\alpha\pi i}$, $\tilde\varphi(b)=\infty$, and $\tilde\varphi(c)=0+$. Then $$\tilde\varphi(\gamma^0)=\{re^{2\alpha\pi i}:r\in[1,\infty)\}.$$
So if we choose $\varphi(z)=1-\tilde\varphi(z)e^{-2\alpha\pi i}$, then $\gamma^0\in \hat{\mathcal X}(\varphi)$.\par 
If instead $b=0\in\D$ and $a$ and $c$ are antipodal then the $I^{(\D;a,b)}_{\rho,c}$-minimizer, $\gamma^0$, is the line segment from $a$ to $b$. If we then choose $\varphi:\D\to\Sigma$ with $\varphi(a)=0$, $\varphi(c)=\infty$, and $\varphi(b)=-1$ then $\varphi(\gamma^0)=[-1,0]$ and $\varphi(\eta(\gamma^0))=(-\infty,-1]$. Hence, $\gamma^0\in\hat{\mathcal X}(\varphi)$. \label{rmk:gamma0inXhat}
\end{remark}
\appendix
\section{Return estimates}\label{section:escape}
In this section, we will study the energy return estimates and return probability estimates which are used to show the large deviation principle on the infinite time SLE$_\kappa(\rho)$ curves. This is done by combining some ideas from \cite{FL15} and \cite{PW21}. In particular, the following proposition and lemmas will be useful. 
\begin{propositionCite}[{{\cite[Proposition A.3]{PW21}}}]\label{prop:propA3} Let $\kappa\in(0,4]$. There exists constants, $c_\kappa>0$ such that $\lim_{\kappa\to 0+}\kappa\log c_\kappa=C\in(-\infty,\infty)$, and the following holds. Let $K\subset\overline\H$ be a hull such that $$K\cap(\H\setminus \D)=\{z\},$$ and let $\gamma^\kappa$ be an SLE$_\kappa$ from $z$ to $\infty$ in $\H\setminus K$, then, for any $r\in(0,1/3)$ we have 
$$\P[\gamma^\kappa\cap S_r\neq \varnothing]\leq c_\kappa r^{8/\kappa-1}.$$\end{propositionCite}
For domains $D$ with smooth boundary and two disjoint subsets $A_1,A_2\subset\partial D$ the Brownian excursion measure is defined by
$$\mathcal E_D(A_1,A_2)=\int_{A_1}\omega_x(A_2,D)|dx|.$$ 
It can be shown that $\mathcal E_D(A_1,A_2)$ is conformally invariant, and therefore the definition may be extended to domains $D$ with non-smooth boundary. We extend the definition in two ways (as in \cite{FL15}). Firstly, if $D$ is not connected, and $A_1=\cup_i \eta_{1,i}$ and $A_2=\cup_i \eta_{2,i}$ are disjoint unions of boundary arcs, then
$$\mathcal E_D(A_1,A_2)=\sum_i\sum_j\int_{\eta_{1,i}}\omega_{x}(\eta_{2,j},D_{i,j})|dx|,$$ 
where $D_{i,j}$ is the unique connected component of $D$ where both $\eta_{1,i}$ and $\eta_{2,j}$ are accessible (here we must treat $\partial D$ in terms of prime ends). Finally, if $A_1,A_2\subset \C$ are not contained in $\partial D$, but are contained in $\partial (D\setminus(A_1\cup A_2))$ then we set 
$$\mathcal E_D(A_1,A_2)=\mathcal E_{D\setminus (A_1\cup A_2)}(A_1,A_2).$$
\begin{lemmaCite}[{{\cite[Lemma A.2]{PW21}}}]\label{lemma:escapeeta} Let $\kappa\in(0,4].$ There exists constants $c_\kappa'\in(0,\infty),$ such that $\lim_{\kappa\to0+}\kappa\log c_\kappa'=C'\in(-\infty,\infty),$ and the following holds. Let $D$ be a simply connected domain and $x,y\in\partial D$ two distinct boundary points. Let $\gamma'$ be a chord from $x$ to $y$ in $D$, and let $\eta$ be a chord (with arbitrary endpoints in $D$) disjoint from $\gamma'$. Finally, let $\gamma^\kappa$ be a chordal SLE$_\kappa$ in $(D;x,y)$. Then, we have
$$\P[\gamma^\kappa\cap \eta\neq \varnothing]\leq c_\kappa\mathcal E_D(\eta,\gamma')^{8/\kappa-1}.$$
\end{lemmaCite}
\begin{lemmaCite}[{{\cite[Lemma 3.3]{FL15}}}]
\label{lemma:brownianexc} Let $D$ be a simply connected domain and $\mathcal S$ the set of crosscuts of $D$ that are subsets of the circle $\partial (r\D)$. Then 
$$\sum_{\eta\in \mathcal S}\mathcal E_D(\partial (R\D),\eta)\leq 2\mathcal E_D(\partial (R\D),\partial (r\D)).$$
\end{lemmaCite}

\subsection{Chordal case}\label{section:chordalescape}
Recall that $\mathcal X^C$ denotes the family of all simple curves in $\H$ starting at $0$ and ending at $\infty$. Let $D_R=R\D$ and $S_R=\partial D_R\cap \H$. For a simple curve $\gamma$ starting at $0$, let $$\tau_R=\inf\{t:|\gamma(t)|=R\}.$$ Note that this differs from the definition of $\tau_R$ in Section \ref{section:ldpinftime}. The goal of this section is to prove the following proposition.
\begin{proposition}\label{prop:chordalescape}
For all $r>0$ and every $M\in[0,\infty)$ there is an $R>r$ s.t. 
\begin{enumerate}[label=(\alph*)]	
	\item $\inf\{I^C_{\rho,v_0}(\gamma)|\gamma\in\mathcal X^C,
	\gamma_{[\tau_R,\infty)}\cap S_r\neq\varnothing\}\geq M$
	\item $\limsup_{\kappa\to 0+}\kappa\log \P^{\kappa,\rho}[\gamma^{\kappa,\rho}_{[\tau_R,\infty)}\cap S_r\neq \varnothing]\leq -M$
	\item[(b')] $\limsup_{\kappa\to 0+}\kappa\log \P^{\kappa,\kappa+\rho}[\gamma^{\kappa,\kappa+\rho}_{[\tau_R,\infty)}\cap S_r\neq \varnothing]\leq -M$
\end{enumerate}
\end{proposition}
We prove (a) separately from (b) and (b'). We will assume, without loss of generality, that $x_0>0$. 
\begin{proof}[Proof of Proposition \ref{prop:chordalescape} (a)]
Let $y_0=-\frac{2x_0}{\rho+2}$ and fix $r>0$. Fix $R>\min(x_0,r)$ and denote
$$E=\{\gamma\in\mathcal X^C:\gamma_{[\tau_R,\infty)}\cap S_r\neq \varnothing\}.$$
For any $\gamma\in E$, let $$\tau_{r,R}=\inf\{t>\tau_R:|\gamma(t)|=r\}\quad \text{and}\quad R'=\sup_{t\in[\tau_{R},\tau_{r,R}]}|\gamma(t)|\geq R.$$ Let $\theta\in(0,\pi)$ be such that $\gamma(\tau_{R'})=R'e^{i\theta}$. Using $$|x_{\tau_{R'}}-W_{\tau_{R'}}|=\pi\omega_\infty([W_{\tau_{R'}},x_{\tau_{R'}}],\H)=\pi\omega_\infty(\gamma_{\tau_{R'}}^+\cup[0,x_0],\H\setminus\gamma_{\tau_{R'}}),$$ and monotonicity of harmonic measure we find
\begin{equation}|x_{\tau_{R'}}-W_{\tau_{R'}}|\leq \pi\omega_\infty(a_\theta^+,\H\setminus a_\theta)=2\sin(\theta/2)(1+\sin(\theta/2))R',\label{eq:harmmeasurebound}\end{equation}
where $a_\theta=\{R'e^{i\phi}:\phi\in[0,\theta]\}$. Similarly,
$$|x_{\tau_{R'}}-y_{\tau_{R'}}|\geq |W_{\tau_{R'}}-g_{\tau_{R'}}(0-)|=\pi\omega_\infty(\gamma_{\tau_{R'}}^-,\H\setminus\gamma_{\tau_{R'}})\geq \omega_\infty(a_\theta^-\cup[0,R'],\H\setminus a_\theta)=\cos^2(\theta/2).$$ 
We now consider two cases. Let $\vare\in(0,\pi/2)$ and suppose that $\theta\leq\vare$. Then
$$\frac{x_{\tau_{R'}}-W_{\tau_{R'}}}{x_{\tau_{R'}}-y_{\tau_{R'}}}\leq \frac{2\sin(\vare/2)(1+\sin(\vare/2))}{\cos^2(\vare/2)}.$$ Since the right hand side approaches $0$ as $\vare\to 0+$ there is an $\vare_0\in(0,\pi/2)$ so that Lemma \ref{lemma:welding} gives $$\theta\leq \vare_0\implies I^C_{\rho,x_0}(\gamma)\geq M.$$
Fix such an $\vare_0$. It remains to show that, if $R$ is sufficiently large, any $\gamma\in E$ with $\theta>\vare_0$ has $I^C_{\rho,x_0}(\gamma)\geq M$. Let $\hat\gamma=g_{\tau_{R'}}(\gamma_{[\tau_{R'},\tau_{r,R}]})-W_{\tau_{R'}}$. If $R/r$ is sufficiently large (e.g., $R/r\geq \sqrt{2}$) a harmonic measure estimate shows that, for any $z\in S_r\setminus \gamma_{\tau_{R'}}$, $$\sin(\pi\omega(z,\gamma_{\tau_{R'}}^+\cup R^+,\H\setminus\gamma_{\tau_{R'}}))\leq 8\frac{r}{R}.$$ Hence,
$I^C(\hat\gamma)\geq -8\log8r/R.$ Moreover, $|x_{\tau_{R'}}-W_{\tau_{R'}}|\geq \sin^2(\vare_0/2)R$, and $|x_{\tau_{r,R}}-W_{\tau_{r,R}}|\leq 4R$, by a arguments similar to those above. So, by Proposition \ref{prop:chordalbounds}, we have 
$$I^C_{\rho,x_0}(\gamma)\geq I^C_{\rho,x_{\tau_{R'}}-W_{\tau_{R'}}}(\hat\gamma)\geq \min(\tfrac{\rho+2}{2},1)\bigg(\min(\tfrac{\rho+2}{2},1)8\log\frac{R}{8r}-|\rho|\log\frac{4}{\sin^2(\vare_0/2)}\bigg). 
$$
Since the second term does not depend on $R$ we see that we can choose $R$ sufficiently large so that the right hand side is larger than $M$. This finishes the proof.
\end{proof}
We now move toward the proof of parts (b) and (b'). The main idea is (loosely) to partition the event 
$E=\{\gamma\in\mathcal X^C:\gamma_{[\tau_R,\infty)}\cap S_r\neq \varnothing\}$ 
into two sub-events, one where we have uniform control on the Radon-Nikodym derivative of SLE$_\kappa(\rho)$ with respect to SLE$_\kappa$, and another where we do not. The probability of the first sub-event will then be controlled using Lemma \ref{prop:propA3}, and the probability of the other sub-event will be controlled using Lemma \ref{lemma:probWelding} and \ref{lemma:boundEn}.\par
For a simple curve $\gamma$ and a force point $x_0>0$ we let $\hat\tau_R:=\inf\{t:|x_t-O_t^-|\geq R\}$, where $O_t^-=g_t(0-)$. Note that $R\mapsto\tau_R$ continuous and strictly increasing on $[x_0,\infty)$ since
$$d(x_t-O_t^-)=\frac{2}{x_t-W_t}dt+\frac{2}{W_t-O_t^-}dt>0$$
by the Loewner equation.
\begin{lemma}\label{lemma:probWelding}Fix $x_0>0$ and $\rho>-2$. For any $R>|x_0|$ and any $\vare\in(0,1)$ we have that 
$$\P^{\kappa,\rho}[|x_{\hat\tau_R}-W_{\hat\tau_R}|<\vare R]\leq C_1(\kappa,\rho)\vare^{2\frac{2+\rho}{\kappa}},$$
where $C_1(\kappa,\rho)$ is a constant such that $$\lim_{\kappa\to 0+}\kappa\log C_1(\kappa,\rho)=\lim_{\kappa\to 0+}\kappa\log C_1(\kappa,\kappa+\rho)=C_2(\rho)\in(-\infty,\infty).$$
\end{lemma}
\begin{proof}
Let $x_0>0$ and $O_t^{-}=g_t(0-)$. The process $$Y_s = \frac{x_{t(s)}-W_{t(s)}}{x_{t(s)}-O_{t(s)}^-}\in[0,1],\quad s(t)=\frac{\kappa}{2}\log\frac{x_t-O^-_t}{x_0},$$ satisfies the SDE 
\begin{equation}dY_s=-\frac{2}{\kappa}Y_sds + \frac{2+\rho}{\kappa}(1-Y_s)ds+
\sqrt{Y_s(1-Y_s)}
d B_s,\label{eq:YsSDE}\end{equation}
with $Y_0=1$, where $B_s$ is a Brownian motion (this follows by a simple computation, but is also a direct consequence \cite[Lemma 3.3]{Z22}). For sufficiently small $\kappa$, the SDE (\ref{eq:YsSDE}) has an invariant distribution with density 
$$\Psi(y) = \frac{\Gamma(2\frac{4+\rho}{\kappa})}{\Gamma(2\frac{2+\rho}{\kappa})\Gamma(\frac{4}{\kappa})}y^{2\frac{2+\rho}{\kappa}-1}(1-y)^{\frac{4}{\kappa}-1}$$
(see, e.g., \cite[Proposition 2.20]{Z22}). Let $\hat Y_s$ be a process, satisfying (\ref{eq:YsSDE}) started at the invariant density.
Since $Y_0=1$ we have, for all $s>0$ and $\vare\in(0,1)$, that
$$\P^{\kappa,\rho}[Y_s<\vare]\leq \P[\hat Y_s<\vare]\leq \frac{\Gamma(2\frac{4+\rho}{\kappa})}{\Gamma(2\frac{2+\rho}{\kappa})\Gamma(\frac{4}{\kappa})}\frac{\kappa}{2(2+\rho)}\vare^{2\frac{2+\rho}{\kappa}}.$$
By Stirling's formula
$$\lim_{\kappa\to 0+}\kappa\log\frac{\Gamma(2\frac{4+\rho}{\kappa})}{\Gamma(2\frac{2+\rho}{\kappa})\Gamma(\frac{4}{\kappa})}=2(2+\rho)\log\frac{4+\rho}{2+\rho}+4\log\frac{4+\rho}{2}.$$ Replacing $\rho$ with $\kappa+\rho$ gives the same limit. Since $s(\tau_R)=\frac{\kappa}{2}\log\frac{x_t-O_t^-}{x_0}$ is a deterministic time, this finishes the proof.
\end{proof}
\begin{lemma}\label{lemma:boundEn}
Fix $x_0>0$, $\rho>-2$, and an integer $n_0\geq 3^4$. Let $\vare_n = (n+n_0)^{-1/2}.$ For every integer $n\geq 0$ we define the event $$E_n=\{\exists t\in[\tau_{2^{n}R},\tau_{2^{n+1}R}] \text{ s.t. } |x_t-W_t|\leq \vare_n 2^nR\}.$$
Then there exists a constant $C_3=C_3(\rho)$ such that,
\begin{align}\limsup_{\kappa\to 0+}\kappa\log\P^{\kappa,\rho}[\cup_{n=0}^\infty E_n] \leq C_3-\frac{\rho+2}{2}\log n_0,\label{eq:boundEn}\\
\limsup_{\kappa\to 0+}\kappa\log\P^{\kappa,\kappa+\rho}[\cup_{n=0}^\infty E_n]\leq C_3-\frac{\rho+2}{2}\log n_0.\label{eq:boundEnPlus}\end{align}
\end{lemma}
\begin{proof}
We begin by noting that for any $r>0$ we have $\hat \tau_{r/2}\leq \tau_r$ (this follows using (\ref{eq:harmmeasurebound})). Hence, $E_n\subset \hat E_n$, where 
$$\hat E_n=\{\exists t\in[\hat\tau_{2^{n-1}R},\tau_{2^{n+1}R}] \text{ s.t. } |x_t-W_t|\leq\vare_n 2^nR\}.$$
Let
$$\hat E_n^1=\{|x_{\hat\tau_{2^{n-1}R}}-W_{\hat\tau_{2^{n-1}R}}|\leq \sqrt{\vare_n} 2^nR\},\quad \hat E_n^2=\hat E_n\setminus \hat E_n^1,$$
so that
$$\limsup_{\kappa\to 0+}\kappa\log\P^{\kappa,\rho}[\cup_{n=0}^\infty E_n]\leq \max\bigg\{\limsup_{\kappa\to 0+}\kappa\log\sum_{n=0}^\infty\P^{\kappa,\rho}[\hat E_n^1],\limsup_{\kappa\to 0+}\kappa\log\sum_{n=0}^\infty\P^{\kappa,\rho}[\hat E_n^2]\bigg\}.$$
We estimate the two limits separately. Using Lemma \ref{lemma:probWelding}
we obtain
$$\limsup_{\kappa\to 0+}\kappa\log\sum_{n=0}^\infty\P^{\kappa,\rho}[\hat E_n^1]\leq \limsup_{\kappa\to 0+}\kappa\log C_1(\kappa,\rho)2^{2\frac{2+\rho}{\kappa}}\sum_{n=0}^\infty\vare_n^{\frac{2+\rho}{\kappa}}\leq C_4(\rho)-\frac{2+\rho}{2}\log n_0,$$ where $C_4(\rho)$ is a constant. 
We now study $\hat E^2_n$. For $y>0$ we let $Z_t^{y}$ be the solution of 
$$dZ^y_t = \frac{\rho+2}{Z^y_t}+\sqrt{\kappa}dB_t,\ Z^y_0=y,$$
i.e., a stochastic process satisfying the same SDE as $x_t-W_t$, started at $y$. By the strong Markov property of Itô diffusions, the bound $\tau_{2^{n+1}R}\leq (2^{n+1}R)^2/2$, and
$$\P[\exists t\in[0,(2^{n+1}R)^2/2]:Z^X_t<\vare_n 2^nR]\leq \P[\exists t\in[0,(2^{n+1}R)^2/2]:Z^y_t<\vare_n 2^nR]$$
for all $y>\vare_n 2^n R$ and random variables $X\in[y,\infty)$
 (since $Z^y_t$ and $Z^X_y$ can be coupled to coincide after their collision), we have
$$\P^{\kappa,\rho}[\hat E_n^2]\leq \P[\exists t\in[0,(2^{n+1}R)^2/2]:Z^{\sqrt{\vare_n} 2^n R}_t<\vare_n 2^nR].$$ By using Lemma \ref{lemma:martingale} (and the explicit form of the constant, see (\ref{eq:explicitconstant})) on the right-hand side we obtain
$$\limsup_{\kappa\to 0+}\kappa\log\sum_{n=0}^\infty\P^{\kappa,\rho}[\hat E_n^2]\leq \limsup_{\kappa\to 0+}\kappa\log C_5(\kappa,\rho)\sum_{n=n_0}^\infty n^{-\frac{\rho+2}{2\kappa}+\frac{3}{4}}=C_6(\rho)-\frac{\rho+2}{2}\log n_0,$$
whenever $n_0^{1/4}\geq 3$ (this comes from the condition $\vare<\vare_0$ of Lemma \ref{lemma:martingale}), and where $\limsup_{\kappa\to 0+} \kappa\log C_5=:C_6\in\R$. 
Thus,
$$\limsup_{\kappa\to 0+}\kappa\log\P^{\kappa,\rho}[\cup_{n=0}^\infty E_n]\leq \max\{ C_4(\rho),C_6(\rho)\}-\frac{\rho+2}{2}\log(n_0),$$
which shows (\ref{eq:boundEn}). Following the same steps we see that the same computation with $\rho$ replaced with $\kappa+\rho$ holds, which shows (\ref{eq:boundEnPlus}). 
\end{proof}
Define $M^{\kappa,\rho}_t=|g_t'(x_0)|^{(4-\kappa+\rho)\rho/(4\kappa)}|x_t-W_t|^{\rho/\kappa}$.
\begin{lemma}On the event $E_n^c$ (with $E_n$ as in Lemma \ref{lemma:boundEn}) we have $M^{\kappa,\rho}_{\tau_{2^{n+1}R}}/M^{\kappa,\rho}_{\tau_{2^{n}R}}\leq f_n(\kappa,\rho)$ where 
$$f_n(\kappa,\rho)=\begin{cases}
\big(\frac{8}{\vare_n}\big)^{\frac{\rho}{\kappa}},&\rho\in[0,\infty),\\
e^{-(n+n_0)\frac{\rho(4-\kappa+\rho)}{\kappa}}\big(\frac{4}{\vare_n}\big)^{-\frac{\rho}{\kappa}},&\rho\in(-2,0),
\end{cases}
$$
for all $\kappa<2$.\label{lemma:boundmartingale}
\end{lemma}
\begin{proof}
We start by studying the case $\rho\in[0,\infty)$. In this case the exponents of $M_t^{\kappa,\rho}$, are non-negative (at least when $\kappa<4$). Firstly, recall that $|g_t'(x_0)|$ is decreasing in $t$. 
Moreover, since $x_0<R,$ we have
$|x_{\tau_{2^{n+1}R}}-W_{\tau_{2^{n+1}R}}|\leq 4\cdot 2^{n+1}R.$
Finally, on the event $E_n^c$ we have that
$|x_{\tau_{2^{n}R}}-W_{\tau_{2^{n}R}}|\geq \vare_n 2^nR.$
Therefore,
$$\frac{M^{\kappa,\rho}_{\tau_{2^{n+1}R}}}{M^{\kappa,\rho}_{\tau_{2^{n}R}}}\leq \bigg(\frac{8}{\vare_n}\bigg)^{\frac{\rho}{\kappa}}.$$
If $\rho\in(-2,0)$, the exponents of $M_t^{\kappa,\rho}$ are negative (at least when $\kappa<2$). Since $$\log\frac{|g_{\tau_{2^{n+1}R}}'(x_0)|}{|g_{\tau_{2^{n}R}}'(x_0)|}=-2\int_{\tau_{2^{n}R}}^{\tau_{2^{n+1}R}}\frac{1}{(x_s-W_s)^2}ds,$$
we have on the event $E_n^c$ that
$$\log\frac{|g_{\tau_{2^{n+1}R}}'(x_0)|}{|g_{\tau_{2^{n}R}}'(x_0)|}\geq -2\frac{\tau_{2^{n+1}R}-\tau_{2^{n}R}}{(\vare_n 2^n R)^2}\geq -2\frac{(2^{n+1}R)^2/2}{(\vare_n 2^n R)^2}=-4(n+n_0).$$
Moreover, $|x_{\tau_{2^nR}}-W_{\tau_{2^nR}}|\leq 4\cdot 2^nR$, and 
$|x_{\tau_{2^{n+1}R}}-W_{\tau_{2^{n+1}R}}|\geq \vare_n 2^nR$ on $E_n^c$.
This yields,
$$\frac{M^{\kappa,\rho}_{\tau_{2^{n+1}R}}}{M^{\kappa,\rho}_{\tau_{2^{n}R}}}\leq e^{-(n+n_0)\frac{\rho(4-\kappa+\rho)}{\kappa}}\bigg(\frac{4}{\vare_n}\bigg)^{-\frac{\rho}{\kappa}}  $$
\end{proof}
\begin{proof}[Proof of Proposition \ref{prop:chordalescape} (b) and (b')]For any $R>x_0$, integer $n_0\geq 3^4$, and $E_n$ as in Lemma \ref{lemma:boundEn}, we have
$$\P^{\kappa,\rho}[\gamma_{[\tau_{R},\infty)}\cap S_r\neq \varnothing]\leq \P^{\kappa,\rho}[\gamma_{[\tau_{R},\infty)}\cap S_r\neq \varnothing, (\cup_{n=0}^\infty E_n)^c] + \P^{\kappa,\rho}[\cup_{n=0}^\infty E_n].$$
Hence,
\begin{equation}\label{eq:termzerochordal}\limsup_{\kappa\to 0+}\kappa \log\P^{\kappa,\rho}[\gamma_{[\tau_{R},\infty)}\cap S_r\neq \varnothing]\end{equation}
is bounded above by the maximum of
\begin{align}&\limsup_{\kappa\to 0+}\kappa\log \P^{\kappa,\rho}[\gamma_{[\tau_{R},\infty)}\cap S_r\neq \varnothing, (\cup_{n=0}^\infty E_n)^c],\label{eq:termonechordal}\\
&\limsup_{\kappa\to 0+}\kappa\log\P^{\kappa,\rho}[\cup_{n=0}^\infty E_n]\leq C_3-(\rho+2)\log n_0,\label{eq:termtwochordal}
\end{align}
where the bound on (\ref{eq:termtwochordal}) was shown in Lemma \ref{lemma:boundEn}. We now fix $n_0\geq 3^4$ such that the right-hand side of (\ref{eq:termtwochordal}) is bounded above by $-M$. It remains to show that, given the choice of $n_0$, there exists an $R$ such that (\ref{eq:termonechordal}) is bounded above by $-M$. For each non-negative integer $n$ we have, by the strong domain Markov property of SLE$_\kappa(\rho)$, and Lemma \ref{lemma:boundmartingale}, that
\begin{align*}&\P^{\kappa,\rho}[\gamma_{[\tau_{2^n R},\tau_{2^{n+1}R}]}\cap S_r\neq \varnothing, (\cup_{n=0}^\infty E_n)^c]\\
 \leq \int & \P^{\kappa,\rho}_{\gamma_{\tau_{2^nR}}}[\hat\gamma_{\hat\tau_{2^{n+1}R}}\cap S_r\neq \varnothing,\hat M^{\kappa,\rho}_{\hat \tau_{2^{n+1}R}}/\hat M^{\kappa,\rho}_0\leq f_n(\kappa,\rho)]d\P^{\kappa,\rho}[\gamma_{\tau_{2^n R}}],\end{align*}
where $\P^{\kappa,\rho}_{\gamma_{\tau_{2^nR}}}$ denotes the law of an SLE$_\kappa(\rho)$ in $\H\setminus\gamma_{\tau_{2^nR}}$ from $\gamma(\tau_{2^nR})$ to $\infty$ with force point $x_0$, and $\hat \cdot$ denotes the corresponding curves, stopping times, etc. We now use the absolute continuity of SLE$_\kappa(\rho)$ with respect to SLE$_\kappa$ with Radon-Nikodym derivative $$\frac{d\P^{\kappa,\rho}_{\gamma_{\tau_{2^nR}}}}{d\P^{\kappa,0}_{\gamma_{\tau_{2^{n}R}}}}(\hat\gamma_{\hat \tau_{2^{n+1}R}})=\frac{\hat M_{\hat \tau_{2^nR}}}{\hat M_0},$$ (where $\hat M_{\hat \tau_{2^{n+1}R}} = M_{\tau_{2^{n+1}R}}$ and $\hat M_0 = M_{\tau_{2^nR}}$), assume that $r/R<1/3$ and apply Lemma \ref{prop:propA3}. This gives
\begin{align*}&\int \P^{\kappa,\rho}_{\gamma_{\tau_{2^nR}}}[\hat\gamma_{\hat\tau_{2^{n+1}R}}\cap S_r\neq \varnothing,\hat M^{\kappa,\rho}_{\hat \tau_{2^{n+1}R}}/\hat M^{\kappa,\rho}_0\leq f_n(\kappa,\rho)]d\P^{\kappa,\rho}[\gamma_{\tau_{2^n R}}]\\
\leq &\int f_n(\kappa,\rho)\P^{\kappa,0}_{\gamma_{\tau_{2^nR}}}[\hat\gamma\cap S_r\neq \varnothing] d\P^{\kappa,\rho}[\gamma_{\tau_{2^n R}}]\\
\leq &f_n(\kappa,\rho)c_\kappa\bigg(\frac{r}{2^nR}\bigg)^{8/\kappa-1}
\end{align*}
and as a consequence,
\begin{align*}\P^{\kappa,\rho}[\gamma_{[\tau_R,\infty)}\cap S_r\neq \varnothing,(\cup_{n=0}^\infty E_n)^c]
\leq &\sum_{n=0}^{\infty} \P^{\kappa,\rho}[\gamma_{[\tau_{2^nR},\tau_{2^{n+1}R}]}\cap S_r\neq \varnothing, E_n^c]\\
\leq &c_\kappa\bigg(\frac{r}{R}\bigg)^{8/\kappa-1} \sum_{n=0}^\infty f_n(\kappa,\rho)2^{-n(8-\kappa)/\kappa}.
\end{align*}
In the case $\rho\in(-2,0)$, the right-hand side can be bounded above in the following way. Recall that in this case
$f_n(\kappa,\rho)=e^{-(n+n_0)\frac{\rho(4-\kappa+\rho)}{\kappa}}(4(n+n_0)^{1/2})^{-\rho/\kappa}.$ Since $-\rho(4-\kappa+\rho)\leq 4$ for all $\rho\in(-2,0),\kappa>0$ we have
$$e^{-\rho(4-\kappa+\rho)}/2^{8-\kappa}\leq e^4/2^7\leq 1/2,$$ for all $\kappa<1$. So for all $\kappa<1$, by $\ell_p$-norm monotonicity,
\begin{align*}\sum_{n=0}^\infty C_n(\kappa,\rho)2^{-n(8-\kappa)/\kappa}
\leq & 2^{-\rho/\kappa}e^{-n_0\rho(4-\kappa+\rho)/\kappa}\sum_{n=0}^\infty \bigg(\frac{(n+n_0)^{-\rho/2}}{2^{n}}\bigg)^{1/\kappa}\\
\leq & 2^{-\rho/\kappa}e^{-n_0\rho(4-\kappa+\rho)/\kappa}\bigg(\sum_{n=0}^\infty \frac{(n+n_0)^{-\rho/2}}{2^{n}}\bigg)^{1/\kappa}
\end{align*}
where the series on the right-hand side is convergent. It follows that 
\begin{equation}\limsup_{\kappa\to 0+}\kappa\log \P^{\kappa,\rho}[\gamma_{[\tau_R,\infty)}\cap S_r\neq \varnothing,(\cup_{n=0}^\infty E_n)^c]\leq C_8(n_0,\rho) + 8\log\frac{r}{R},\label{eq:escapefinal}\end{equation}
where $C_8(n_0,\rho)<\infty$ is a constant. By a similar argument we obtain (\ref{eq:escapefinal}) in the case $\rho\in[0,\infty)$. By choosing $R$ sufficiently large with respect to $\rho$ and $n_0$ the right-hand side of (\ref{eq:escapefinal}) is bounded above by $-M$. It follows that (\ref{eq:termzerochordal}) is bounded above by $-M$ as desired. The case where $\rho$ is replaced by $\kappa+\rho$ follows in the same way.
\end{proof}

\subsection{Radial case}\label{section:radialescape}
Recall that $\mathcal X^R$ denotes the family of simple curves in $\D$ starting at $1$ and ending at $0$. Fix $v_0\in(0,2\pi)$ and $\rho>-2$. Let $D_R=\{z\in\D:|z-z_0|<R\}$ and $S_R=\partial D_R$. For $\gamma\in\mathcal X^R$, let $\tau_R=\inf\{t:|\gamma(t)|=R\}$. The goal of this section is to prove the following proposition.
\begin{proposition}\label{prop:escaperadial}
For every $R\in(0,1)$ and every $M\in[0,\infty)$ there is an $r\in(0,R)$ s.t. 
\begin{enumerate}[label=(\alph*)]	
	\item $\inf\{I^R_{\rho,e^{iv_0}}(\gamma):\gamma\in\mathcal X^R,
	\gamma_{[\tau_r,\infty)}\cap S_R\neq\varnothing\}\geq M$
	\item $\limsup_{\kappa\to 0+}\kappa\log \P[\gamma^{\kappa,\rho}_{[\tau_r,\infty)}\cap S_R\neq \varnothing]\leq -M$
	\item[(b')] $\limsup_{\kappa\to 0+}\kappa\log \P[\gamma^{\kappa,\kappa+\rho}_{[\tau_r,\infty)}\cap S_R\neq \varnothing]\leq -M$
\end{enumerate}
\end{proposition}
We first show the following topological lemma which will allow us to, in a helpful way, partition the set $\{\gamma\in\mathcal X^R, \gamma_{[\tau_r,\infty)}\cap S_R\neq\varnothing\}$. 
\begin{lemma}\label{lemma:etatop}
Let $0<r<r'<R<1$ and consider a simple curve $\gamma:(0,T]\to\D$ starting at $1$ and with $\tau_r=T$. Let $\Omega$ denote the connected component of $D_R\setminus\gamma$ containing $0$. The set $S_R\cap \partial \Omega$ consists of (at most) countably many connected arcs $\eta$. Let $\Omega_0$ and $\Omega_\eta$ be the connected components of $\D\setminus(\gamma\cup S_{r'})$ containing $0$ and $\eta$ respectively. Then there is a partitioning of the set $\{\eta\}$ into two sets $A$ and $B$ so that the following holds
\begin{enumerate}[label=(\roman*)]
	\item If $\eta\in A$, then only one side of $\gamma$ is accessible within $\Omega_\eta$.
	\item If $\eta\in B$ and $\tilde\gamma:(0,\tilde T]\to\D$, $\tilde T>T$, is a simple curve for which $\tilde\gamma(t)=\gamma(t)$ for all $t\in(0,T]$, and $\tilde\gamma(\tilde T)\in \eta$ and $\tilde T=\inf\{t>T:|\tilde\gamma(t)|=R\}$, then only one side of $\tilde\gamma$ is accessible from $0$ in $\Omega_0\setminus\tilde\gamma$.
	\item There exists a curve $\gamma'\subset \D\setminus\gamma$ with endpoints at $\gamma(T)$ and $e^{iv_0}$ which does not intersect $\cup_{\eta\in A}\Omega_\eta$.
\end{enumerate}
\end{lemma}
\begin{figure}
\includegraphics[width=\textwidth]{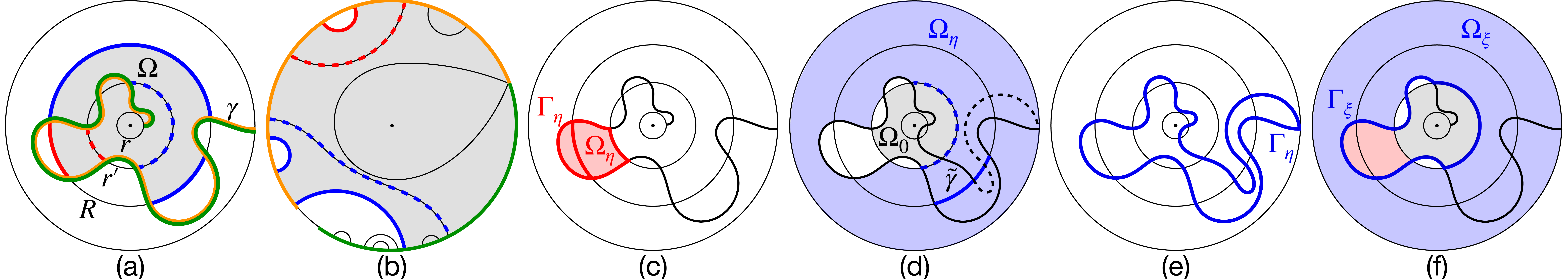}
\caption{Set-up in Lemma \ref{lemma:etatop} and its proof. In (a) we illustrate the left and right sides of a curve $\gamma$ in green and orange respectively, and the arcs $\eta$ in red and blue if they belong to $A$ and $B$ respectively. The dashed arcs correspond to arcs $\xi_\eta$. In (b) we show (schematically) how the set-up in (a) looks after mapping out $\gamma$. In (c)-(f) we illustrate various parts of the proof. Specifically, in (d) we illustrate an extension $\tilde\gamma$ of $\gamma$ (the solid curve) and a further extension, avoiding $\Omega_0$ (the dashed curve).\label{fig:toplemma}}
\end{figure}
\begin{proof}
The set $\partial\Omega_0\cap S_{r'}$ consists of (possibly) countably many connected arcs $\xi$. Among these arcs there is, for every $\eta$, exactly one arc $\xi=\xi_\eta$ separating $0$ from $\eta$ in $\D\setminus\gamma$. We let $\eta\in A$ if the endpoints of $\xi_\eta$ belong to the same side of $\gamma$, and $\eta\in B$ if $\xi_\eta$ if the endpoints of $\gamma$ belong to different sides of $\gamma$, see Figure \ref{fig:toplemma}(a) and (b).\par
For $\eta\in A$, $\xi_\eta$ and the subarc of $\gamma$ connecting the two endpoints of $\xi_\eta$ form a Jordan curve $\Gamma_\eta$ where the bounded component of $\C\setminus\Gamma_\eta$ contains (or equals) $\Omega_\eta$, and the remaining parts of $\gamma$ are contained in the unbounded component (see Figure \ref{fig:toplemma}(c)). Therefore, $\partial\Omega_\eta$ can only contain points from one side of $\gamma$, this shows (i).\par
If $\eta\in B$, then, by construction, $\Omega_\eta$ has all of $\partial \D$ as accessible points. Thus, any curve $\tilde\gamma$ as in (ii) could be continued from $\tilde\gamma(\tilde T)$ to $1$ without crossing $\Omega_0$, forming a Jordan curve $\Gamma_\eta$ (see Figure \ref{fig:toplemma}(d) and (e)). Since only one side of $\Gamma_\eta$ is accessible in $\C\setminus\Gamma_\eta$ from the origin, and since the continuation of $\tilde\gamma$ does not cross $\Omega_0$, only one side of $\tilde\gamma$ can be accessible from the origin in $\Omega_0\setminus\tilde\gamma$. This shows (ii).\par
Finally, there is exactly one arc $\xi\in\partial \Omega_0\cap S_{r'}$ separating $0$ from $1$ in $\D\setminus\gamma$ (namely $\xi=\xi_\eta$ for all $\eta\in B$). Let $\Omega_\xi$ be the outer component of $\D\setminus\Gamma_\xi$ where $\Gamma_\xi$ is the Jordan curve consisting of $\xi$ and the appropriate sub-arc of $\gamma$ (see Figure \ref{fig:toplemma}(f)). Observe that $\xi$ does not separate any of the arcs in $A$ from $0$ in $\D\setminus\gamma$, and hence $\cup_{\eta\in A}\Omega_{\eta}\subset \D\setminus \Omega_\xi$. On the other hand, $\Omega_0\cup\Omega_{\xi}$ is an open connected set with $\gamma(T)$ and $e^{iv_0}$ as accessible points. Therefore, there exists a curve $\gamma'$ in $\Omega_0\cup\Omega_\xi$ connecting $0$ to $1$, and any such a curve is disjoint from $\Omega_{\eta\in A}\Omega_\eta$. 
\end{proof}
For some $0<r<r'<R<1$, and simple curve $\gamma:(0,\tau_r]\to \D$ starting at $1$, let $A_{r,r',R,\gamma}$  and $B_{r,r',R,\gamma}$ be the sets of arcs $A$ and $B$  from Lemma \ref{lemma:etatop}. 
We partition the set of simple returning curves $$E_{r,R}=\{\gamma\in\mathcal X^R:\gamma_{[\tau_r,\infty)}\cap S_R\neq \varnothing\}=\tilde A_{r,r',R}\cup \tilde B_{r,r',R},$$
where 
$$\tilde A_{r,r',R}:=\{\gamma\in E_{r,R}:\gamma(\tau_{r,R})
\in\eta,\eta\in A_{r,r',R,\gamma_{\tau_r}}\},$$
$$\tilde B_{r,r',R}:=\{\gamma\in E_{r,R}:\gamma(\tau_{r,R})\in\eta,\eta\in B_{r,r',R,\gamma_{\tau_r}}\},$$
and $\tau_{r,R}=\inf\{t>\tau_r:|\gamma(t)|=R\}$.
\begin{lemma}\label{lemma:escapeangle}
If $\gamma\in\tilde B_{r,r',R}$ and $r/r'<3-2\sqrt{2}$, then $\sin((v_{\tau_{r,R}}-w_{\tau_{r,R}})/2))\leq 4\sqrt{r/r'}$.  
\end{lemma}
\begin{proof}
By applying Lemma \ref{lemma:etatop}(ii) to $\gamma_{\tau_r}$, we find that only one side of $\gamma_{\tau_{r,R}}$ will be accessible from $\Omega_0$. Without loss of generality, we may assume it to be the right side. Using monotonicity of harmonic measure we find  
$$\frac{(v_{\tau_{r,R}}-w_{\tau_{r,R}})}{2\pi}=\omega(0,\gamma_{\tau_{r,R}}^+\cup a,\D\setminus\gamma_{\tau_{r,R}})\geq\omega(0,\gamma_{\tau_{r,R}},\Omega_0\setminus\gamma_{\tau_{r,R}}),$$
where $a$ is the circular arc from $1$ to $e^{iv_0}$.
By Beurling's projection theorem (see, e.g., \cite[Theorem III.9.3]{GM05}), and an explicit computation, we have that
\begin{equation*}\label{eq:harmBeur}\omega(0,\gamma_{\tau_{r,R}},\Omega_0\setminus\gamma_{\tau_{r,R}})\geq\omega(0,[r/r',1],\D\setminus[r/r',1])=\frac{2}{\pi}\arcsin\frac{1-r/r'}{1+r/r'},
\end{equation*}
and if $r/r'<3-2\sqrt{2}$ then the right-hand side is larger than $\pi/2$. By monotonicity of $\sin$ on $[\pi/2,\pi]$, and the explicit computation
$$\sin((v_{\tau_{r,R}}-w_{\tau_{r,R}})/2))=\sin(\pi\omega(0,[r/r',1],\D\setminus[r/r',1]))\leq 4\sqrt{r/r'}.$$ 
\end{proof}

\begin{lemma}\label{lemma:escapeenergy}
If $\gamma\in\tilde A_{r,r',R}$ and $r'/R<3-2\sqrt{2}$, then $I^{\D,1,e^{iv_0}}(\gamma)\geq -4\log(16r'/R).$
\end{lemma}
\begin{proof}
Since $\gamma\in\tilde A_{r,r',R}$ we have that $\gamma_{\tau_{r,R}}\in \eta$ for some $\eta\in A_{r,r',R,\gamma_{\tau_r}}$, and by Lemma \ref{lemma:etatop}(i) only one side of $\gamma_{\tau_r}$ is accessible from $\Omega_\eta$. Without loss of generality, we may assume that it is the right side. We have, by monotonicity of harmonic measure,
$$\omega(\gamma(\tau_{r,R}),\gamma_{\tau_r}^+\cup a,\D\setminus\gamma_{\tau_r})\geq\omega(\gamma(\tau_{r,R}),\gamma_{\tau_r},\Omega_\eta)$$
where $a$
is the circular arc from $1$ to $e^{iv_0}$. 
One can see, similar to Lemma \ref{lemma:etatop}(iii), that there is a curve $\hat\eta$ from $0$ to $1$ in $\D\setminus\gamma_{\tau_r}$, which does not cross $\Omega_\eta$ for any $\eta\in A_{r,r',R,\gamma_{\tau_r}}$, and therefore (again by monotonicity)
$$\omega(\gamma(\tau_{r,R}),\gamma_{\tau_r},\Omega_\eta)\geq \omega(\gamma(\tau_{r,R}),\hat\gamma\cup[1,\infty],\hat \C\setminus(\overline{D_{r'}}\cup\hat\gamma\cup[1,\infty])).$$
Let $\psi: \hat\C\setminus\overline{D_{r'}}\to \D$ be a conformal map with $\psi(\gamma_{\tau_{r,R}})=0$. Then, $|\psi(\infty)|=r'/R$ and $\psi(\hat\gamma\cup[1,\infty])$ is a path connecting $\partial\D$ and $\psi(\infty)$. Hence, Beurling's projection theorem (see, e.g., \cite[Theorem III.9.3]{GM05}) gives
$$\omega(\gamma(\tau_{r,R}),\hat\gamma\cup[1,\infty],\hat \C\setminus(\overline{D_{r'}}\cup\hat\gamma\cup[1,\infty]))\geq\omega(0,[r'/R,1],\D\setminus[r'/R,1]).$$
The same argument as in the proof of Lemma \ref{lemma:escapeangle} now shows
\begin{equation}\label{eq:claim}\sin(\pi\omega(\gamma(\tau_{r,R}),\gamma_{\tau_r}^+\cup a(0,v_0),\D\setminus\gamma_{\tau_r}))\leq 4\sqrt{r'/R},\end{equation}
given that $r'/R<3-2\sqrt{2}$, which gives the desired bound.
\end{proof}
\begin{proof}[Proof of Proposition \ref{prop:escaperadial}(a)] Fix $R$, $\rho$, $v_0$ and $M$. Set $r'\in(0,R)$ so that $r'/R<3-2\sqrt{2}$ and $$\min(1,\tfrac{2+\rho}{4})\bigg(-4\min(1,\tfrac{2+\rho}{4})\log(16r'/R)+(6+\rho)\log\sin(v_0/2)\bigg)\geq M.$$ Then, set $r\in(0,r')$ so that $r/r'<3-2\sqrt{2}$ and 
$$\min(1,\tfrac{2+\rho}{4})\bigg(-(6+\rho)\log(4\sqrt{r/r'})+\log(6+\rho)\log\sin(v_0/2)\bigg)\geq M.$$
Then, any curve $\gamma\in\tilde A_{r,r',R}\cup \tilde B_{r,r',R}$ has $I^R_{\rho,e^{iv_0}}(\gamma)= I^{\D,1,e^{iv_0}}_{-6-\rho,0}(\gamma)\geq M$ by the upper bound from Proposition \ref{prop:radialbounds} and the estimates from Lemma \ref{lemma:escapeangle} and \ref{lemma:escapeenergy}.
\end{proof}
Let $0<\alpha<\beta<1$ and $0<r_0<r_0'<R<1$ (these constants will be fixed in the proof of Proposition \ref{prop:escaperadial}(b) and (b'), but for now we consider them as arbitrary). Set $r_n=\alpha^n r_0$ and $r_n'=\beta^n r_0'$. For ease of notation, we denote $\tau_{n}=\tau_{r_n}$. By Koebe-1/4 we have that
$\crad(0,\D\setminus\gamma_{\tau_n})\in [r_n,4r_n]$, which gives
$$\tau_n\in [-\log(4\alpha^nr_0),-\log(\alpha^n r_0)],\quad \tau_{{n+1}}-\tau_{n}\in [-\log(4\alpha),-\log(\alpha/4)].$$
\begin{lemma}\label{lemma:sinpoly}
Let $\vare_n=(n+n_0)^{-1/2}$ and $$n_0>\max(1/\sin^2(v_0/6),1/\sin^2((2\pi-v_0)/6),\log(r_0)/\log(\alpha)).$$ Then
$$\limsup_{\kappa\to 0+}\kappa\log\bigg(\sum_{n=0}^\infty\P^{\kappa,\rho}[\exists t\in[0,\tau_{n}]:\ \sin((v^{\kappa,\rho}_t-w^{\kappa,\rho}_t)/2)<\vare_n]\bigg)\leq C_1(\rho,v_0)-(\rho+2)\log n_0,$$
where $C_1(\rho,v_0)$ is a constant.
The same holds when $\P^{\kappa,\rho}$ is replaced by $\P^{\kappa,\kappa+\rho}$. 
\end{lemma}
\begin{proof}
Using $\vare_n<\vare_0<\min(\sin(v_0/6),\sin((2\pi-v_0)/6))$, (\ref{eq:explicitconstant}) from the proof of Lemma \ref{lemma:martingale}, and $\tau_{n}\leq -\log(\alpha^nr_0)$, we find 
\begin{align*}
\sum_{n=0}^\infty\P^{\kappa,\rho}[\exists t\in[0,\tau_{n}]:\ \sin((v^{\kappa,\rho}_t-w^{\kappa,\rho}_t)/2)<\vare_n]
&\leq -C_2(\kappa,\rho,v_0)\sum_{n=0}^\infty (n\log\alpha+\log r_0)\vare_n^{2\frac{\rho+2}{\kappa}-1}\\
&\leq -C_2(\kappa,\rho,v_0)\log\alpha\sum_{n=0}^\infty (n+n_0)^{3/2-(\rho+2)/\kappa},
\end{align*}
where $\limsup_{\kappa\to 0+}\kappa\log C_2(\kappa,\rho,v_0)=:C_1(\rho,v_0)\in\R$, and where the second inequality follows from $n_0>\log(r_0)/\log(\alpha)$. Hence,
$$\limsup_{\kappa\to 0+}\sum_{n=0}^\infty\P^{\kappa,\rho}[\exists t\in[0,\tau_{n}]:\ \sin((v^{\kappa,\rho}_t-w^{\kappa,\rho}_t)/2)<\vare_n]\leq C_1-(\rho+2)\log(n_0).$$
The same holds when $\P^{\kappa,\rho}$ is replaced by $\P^{\kappa,\kappa+\rho}$.
\end{proof}
\begin{lemma}\label{lemma:sinexp}If $4\sqrt{r_0/r_0'}<\min(\sin(v_0/6),\sin((2\pi-v_0)/6))$, then
$$\limsup_{\kappa\to 0+}\kappa\log\bigg(\sum_{n=0}^\infty\P^{\kappa,\rho}[\exists t\in[0,\tau_{r_{n+1}}]:\sin((v^{\kappa,\rho}_t-w^{\kappa,\rho}_t)/2)<4\sqrt{r_n/r_n'}]\bigg)\leq C_3+(\rho+2)\log\frac{r_0}{r_0'},$$
where $C_3=C_3(\rho,v_0)$ is a constant.
\end{lemma}
The proof is almost identical to that of Lemma \ref{lemma:sinpoly} and is therefore omitted.
\begin{proof}[Proof of Proposition \ref{prop:escaperadial}(b) and (b')]
Fix $R\in(0,1)$ and $M>0$ and consider sequences $(r_n)$, $(r_n')$ as in the set-up above. The constants $\alpha$ and $\beta$ will depend on $\rho$ but not on $M$ or $R$ (in fact, we will see that fixing $\alpha/\beta\in(0,1)$ and then choosing $\beta<1$ so that 
$\beta\leq 4^{-4}(\frac{\alpha}{\beta})^{5}$ if $\rho\in(-2,2)$ and $\beta<(\frac{\alpha}{\beta})^{1+\rho/6}$ if $\rho\in[2,\infty)$ will be sufficient for our purposes). Using the lemmas above, we will show that there is a choice of $r=r_0$ such that the statement holds. We start by imposing that $r_0<r'_0(3-2\sqrt{2})$ so that $0<r_n<r_n'(3-2\sqrt{2})<r_n'<R$ for all $n\geq 0$.
Let $\gamma$ be a radial SLE$_\kappa(\rho)$ in $\D$ from $0$ to $1$ with force point $e^{iv_0}$. 
We have
$$\P^{\kappa,\rho}[\gamma_{[\tau_{0},\infty)}\cap S_R\neq \varnothing]\leq\sum_{n\geq 0}\P^{\kappa,\rho}[\gamma_{[\tau_n,\tau_{n+1}]}\cap S_R\neq \varnothing],$$
and the $n$-th term on the right-hand side is bounded above by
\begin{align*}& \P^{\kappa,\rho}[\exists t\in[0,\tau_{{n+1}}]:\sin(\theta_t)<\tilde\vare_n]+\P^{\kappa,\rho}[\gamma_{[\tau_{n},\tau_{{n+1}}]}\cap S_R\neq \varnothing,\ \sin(\theta_t)\geq \tilde\vare_n \ \forall t\in[0,\tau_{{n+1}}]]\end{align*}
where $\tilde\vare_n=\max(4\sqrt{r_n/r_n'},\vare_n)$, $\vare_n=\vare(n_0)$, for some positive integer $n_0$, is as in Lemma \ref{lemma:sinpoly}, and $\theta_t:=(v_t-w_t)/2$. 
Therefore,
$\limsup_{\kappa\to 0+}\kappa\log\P^{\kappa,\rho}[\gamma_{[\tau_{0},\infty)}\cap S_R\neq \varnothing]$
is bounded above the maximum of 
\begin{align}&\limsup_{\kappa\to 0+}\kappa\log\sum_{n=0}^\infty\P^{\kappa,\rho}[\exists t\in[0,\tau_{{n+1}}]:\sin(\theta_t)<\tilde\vare_n],\label{eq:termone}\\
&\limsup_{\kappa\to 0+}\kappa\log\sum_{n=0}^\infty\P^{\kappa,\rho}[\gamma_{[\tau_{n},\tau_{{n+1}}]}\cap S_R\neq \varnothing \text{ and } \sin(\theta_t)\geq \tilde\vare_n \ \forall t\in[0,\tau_{{n+1}}]]\label{eq:termtwo}.\end{align}
By choosing $n_0$ sufficiently large, and the ratio $r_0/r_0'$ sufficiently small, Lemma \ref{lemma:sinpoly} and \ref{lemma:sinexp} show that (\ref{eq:termone}) is bounded above by $-M$. Note that the choices of $n_0$ and $r_0/r_0'$ depend only on $M$ and $\rho$, and in particular not on $r_0'/R$, so if we are able to bound (\ref{eq:termtwo}) above by $-M$, by choosing $r_0'/R$ appropriately, we are done.\par
Denote $B_n:=B_{r_n,r_n',R,\gamma_{\tau_n}}$. By Lemma \ref{lemma:escapeangle} and 
$$\{\gamma:\gamma_{[\tau_{n},\infty)}\cap S_R\neq \varnothing\}=\tilde A_{r_n,r_n',R}\cup\tilde B_{r_n,r_n',R}$$
we have 
\begin{align*}&\P^{\kappa,\rho}[\gamma_{[\tau_{n},\tau_{{n+1}}]}\cap S_R\neq \varnothing \text{ and } \sin(\theta_t)\geq \tilde\vare_n \ \forall t\in[0,\tau_{{n+1}}]]\\
\leq &\P^{\kappa,\rho}[\gamma_{[\tau_{n},\tau_{{n+1}}]}\cap(\cup_{\eta\in B_n}\eta)\neq \varnothing \text{ and } \sin(\theta_t)\geq \tilde\vare_n \ \forall t\in[0,\tau_{{n+1}}]]\\
=&\int_{\{\sin(\theta_t)\geq \tilde\vare_n\}} \P^{\kappa,\rho}_{\gamma_{\tau_n}}[\hat\gamma_{\hat \tau_{{n+1}}}\cap (\cup_{\eta\in B_n}\eta)\neq \varnothing \text{ and } \sin(\hat \theta_t)\geq \tilde\vare_n \ \forall t\in[0,\hat \tau_{ {n+1}}]]d\P^{\kappa,\rho}[\gamma_{\tau_{n}}],\end{align*}
where the strong domain Markov property of SLE$_\kappa(\rho)$ is used in the last step, and $ \P^{\kappa,\rho}_{\gamma_{\tau_n}}$ denotes the law of an SLE$_\kappa(\rho)$ from $\gamma(\tau_n)$ to $0$ in $\D\setminus\gamma_{\tau_n}$ with force point $e^{iv_0}$, and $\hat\cdot$ denotes the corresponding curve, stopping times, etc.
Recall that SLE$_\kappa(\rho)$ in $\D\setminus\gamma^{\kappa,\rho}_{\tau_{r_n}}$ is absolutely continuous with respect to a SLE$_\kappa(\kappa-6)$ or equivalently a (reparametrized) \emph{chordal} SLE$_\kappa$ in the same domain from $\gamma(\tau_{r_n})$ to $e^{iv_0}$ for stopping times strictly before the swallowing time of $0$ or $e^{iv_0}$. For the stopping time $\hat\tau_n'=\min(\hat \tau_{{n+1}},\inf\{t:\sin(\hat\theta_t)\leq \tilde\vare_n\})$, which is such a stopping time, we have 
$$\frac{d \P^{\kappa,\rho}_{\gamma_{\tau_n}}}{d \P^{\kappa,\kappa-6}_{\gamma_{\tau_n}}}(\hat \gamma_{\hat\tau_n'})=\bigg(\frac{\hat g_{\hat \tau_n'}'(0)}{\hat g_{0}'(0)}\bigg)^{\frac{(4+\rho)\rho-(\kappa-2)(\kappa-6)}{8\kappa}}\bigg(\frac{|\hat V_{\hat\tau_n'}-\hat W_{\hat\tau_n'}|}{|\hat V_{0}-\hat W_{0}|}\bigg)^{\frac{\rho-\kappa+6}{\kappa}}\bigg(\frac{|\hat g_{\hat \tau_n'}'(V_0)|}{|\hat g_{0}'(V_0)|}\bigg)^{\frac{(4-\kappa+\rho)\rho+2(\kappa-6)}{4\kappa}}.$$
We claim that, on the set $E_n=\{\hat\gamma_{\hat \tau_{n+1}}:\sin(\hat\theta_t)\geq \tilde\vare_n\ \forall t\in[0,\hat\tau_{n+1}]\}$ we have 
\begin{equation}\label{eq:boundRadonNik}\frac{d \P^{\kappa,\rho}_{\gamma_{\tau_n}}}{d \P^{\kappa,\kappa-6}_{\gamma_{\tau_n}}}(\hat \gamma_{\hat\tau_n'})\leq M_{n,\kappa}:=
f_\kappa\cdot \begin{cases}
(\frac{\alpha}{\beta})^{n\frac{\kappa-6-\rho}{2\kappa}}
(\frac{\alpha}{4})^{n\frac{(4-\kappa+\rho)\rho+2(\kappa-6)}{8\kappa}},
&\rho\in(-2,2],\\
(\frac{\alpha}{\beta})^{n\frac{\kappa-6-\rho}{2\kappa}},
&\rho\in[2,\infty),\\
\end{cases}\end{equation}
where $f_\kappa=f_\kappa(\rho,\alpha,n_0,r_0/r_0')$ (but does not depend on $n$ or $r_0'/R$), and for every fixed $\rho$, $\alpha$, $n_0$, $r_0/r_0'$ satisfying the constraints above, 
$$\limsup_{\kappa\to 0+}f_\kappa(\rho,\alpha,n_0,r_0/r_0')=\limsup_{\kappa\to 0+}f_\kappa(\kappa+\rho,\alpha,n_0,r_0/r_0')=\tilde C$$ where $\tilde C=\tilde C(\rho,\alpha,n_0,r_0/r_0')\in\R$.
To show this, first note that 
$$\frac{\hat g_{\hat \tau_n'}'(0)}{\hat g_{0}'(0)}=\frac{e^{-\tau_{{n+1}}}}{e^{-\tau_{n}}}\in[\alpha/4,4\alpha],$$
since $\hat\tau_n'=\hat\tau_{{n+1}}$ on $E_n$. So, the first factor of the Radon-Nikodym derivative can be ``swallowed'' by $f_\kappa$. Secondly, since $|\hat V_t-\hat W_t|=2\sin(\hat\theta_t)$, and $\rho-\kappa+6>0$ when $\kappa\leq 4$, 
$$\bigg(\frac{|\hat V_{\hat\tau_n'}-\hat W_{\hat\tau_n'}|}{|\hat V_{0}-\hat W_{0}|}\bigg)^{\frac{\rho-\kappa+6}{\kappa}}\leq \bigg(4\sqrt{\frac{r_n}{r_n'}}\bigg)^{\frac{\kappa-6-\rho}{\kappa}}\leq \bigg(4\sqrt{\frac{r_0}{r_0'}}\bigg)^{\frac{\kappa-6-\rho}{\kappa}}\bigg(\frac{\alpha}{\beta}\bigg)^{n\frac{\kappa-6-\rho}{2\kappa}},$$
where the first factor on the right-hand side is ``swallowed'' by $f_\kappa$. Finally, by the radial Loewner equation
$$\log\frac{|\hat g_{\hat\tau_n'}'(V_0)|}{|\hat g_{0}'(V_0)|}=-\int_0^{\hat\tau_n'}\frac{1}{2\sin^2(\hat\theta_t)}dt.$$
This shows that
$$0\leq -\log\frac{|\hat g_{\hat\tau_n'}'(V_0)|}{|\hat g_{0}'(V_0)|}\leq \frac{\hat\tau_n'}{2\tilde\vare_n^2}\leq \frac{(\tau_{{n+1}}-\tau_{n})}{2\vare^2_n}\leq -\log(\alpha/4)\frac{(n+n_0)}{2},$$
and therefore
$$\bigg(\frac{\alpha}{4}\bigg)^{\frac{n+n_0}{2}}\leq \frac{|\hat g_{\hat\tau_n'}'(V_0)|}{|\hat g_{0}'(V_0)|}\leq 1.$$ Since $(4-\kappa+\rho)\rho+2(\kappa-6)$ is non-negative on $[2,\infty)$ and non-positive on $(-2,2]$, the third factor of the Radon-Nikodym derivative can be bounded above by $1$ if $\rho\in[2,\infty)$ and by
$$\bigg(\frac{\alpha}{4}\bigg)^{n_0\frac{(4-\kappa+\rho)\rho+2(\kappa-6)}{8\kappa}}\bigg(\frac{\alpha}{4}\bigg)^{n\frac{(4-\kappa+\rho)\rho+2(\kappa-6)}{8\kappa}}
$$
if $\rho\in(-2,2]$ where again, the first factor is ``swallowed'' by $f_\kappa$. It follows that (\ref{eq:boundRadonNik}) holds. 
We see that 
\begin{align*}&\int_{\{\sin(\theta_t)\geq \tilde\vare_n\}} \P^{\kappa,\rho}_{\gamma_{\tau_n}}[\hat\gamma_{\hat \tau_{{n+1}}}\cap (\cup_{\eta\in B_n}\eta)\neq \varnothing \text{ and } \sin(\hat \theta_t)\geq \tilde\vare_n \ \forall t\in[0,\hat \tau_{ {n+1}}]]d\P^{\kappa,\rho}[\gamma_{\tau_{n}}]\\
&\leq \int_{\{\sin(\theta_t)\geq \tilde\vare_n\}}M_{n,\kappa} \P^{\kappa,\kappa-6}_{\gamma_{\tau_n}}[\hat\gamma_{\tau_{n+1}}\cap (\cup_{\eta\in B_n}\eta)\neq \varnothing]d\P^{\kappa,\rho}[\gamma_{\tau_{n}}]\\
&\leq \int_{\{\sin(\theta_t)\geq \tilde\vare_n\}}M_{n,\kappa}\tilde\P^{\kappa}_{\gamma_{\tau_n}}[\tilde\gamma\cap (\cup_{\eta\in B_n}\eta)\neq \varnothing]d\P^{\kappa,\rho}[\gamma_{\tau_{n}}],\end{align*}
where $\tilde \P^\kappa_{\gamma_{\tau_n}}$ denotes the law of a chordal SLE$_\kappa$, $\tilde\gamma$, in $\D\setminus\gamma_{\tau_n}$ from $\gamma(\tau_n)$ to $e^{iv_0}$. 
Using Lemma \ref{lemma:escapeeta} with $D=\D\setminus\gamma_{\tau_{n}}$, $x=\gamma(\tau_{n})$, $y=e^{iv_0}$, $\eta\in B_n$, and $\gamma'$ as in Lemma \ref{lemma:etatop}(iii), we obtain
\begin{align*}\tilde \P^{\kappa}_{\gamma_{\tau_n}}[\tilde\gamma\cap (\cup_{\eta\in B}\eta)\neq \varnothing]
&\leq \sum_{\eta\in B_n}\tilde \P^{\kappa}_{\gamma_{\tau_n}}[\tilde\gamma\cap \eta\neq \varnothing]
\leq c_\kappa\sum_{\eta\in B_n}\mathcal E_{\D\setminus\gamma_{\tau_n}}(\eta,\gamma')^{\frac{8}{\kappa}-1}\\
&\leq c_\kappa\sum_{\eta\in B_n}\mathcal E_{\D\setminus\gamma_{\tau_n}}(\eta,S_{r_n'})^{\frac{8}{\kappa}-1}\leq c_\kappa\bigg(\sum_{\eta\in B_n}\mathcal E_{\D\setminus\gamma_{\tau_n}}(\eta,S_{r_n'})\bigg)^{\frac{8}{\kappa}-1}\\
&\leq c_\kappa\bigg(2\mathcal E_{\D\setminus\gamma_{\tau_n}}(S_R,S_{r_n'})\bigg)^{\frac{8}{\kappa}-1}\leq c_\kappa\bigg(\tilde c\sqrt{\frac{r_{n}'}{R}}\bigg)^{\frac{8}{\kappa}-1}=c_\kappa\bigg(\tilde c\beta^{n/2}\sqrt{\frac{r_0'}{R}}\bigg)^{\frac{8}{\kappa}-1}
\end{align*}
In the third inequality, we have used that $S_{r'_n}$ separates $\gamma'$ from every $\eta\in B_n$ (Lemma \ref{lemma:etatop}(iii)). In the fourth inequality we use monotonicity of the $\ell^p$ norm and $8/\kappa-1\geq 1$. The fifth and sixth inequalities follow from Lemma \ref{lemma:brownianexc} and \cite[Equation (2.5)]{FL15}, whenever $r_0'/R\leq 1/2$, where $\tilde c$ is a constant. Now,
\begin{align*}
&\sum_{n=0}^\infty\P^{\kappa,\rho}[\gamma_{[\tau_{n},\tau_{{n+1}}]}\cap S_R\neq \varnothing,\ \sin(\theta_t)\geq \tilde\vare_n \ \forall t\in[0,\tau_{{n+1}}]]\\
\leq &\sum_{n=0}^\infty\P^{\kappa,\rho}[\gamma_{[\tau_{n},\tau_{{n+1}}]}\cap(\cup_{\eta\in B_n}\eta)\neq \varnothing,\ \sin(\theta_t)\geq \tilde\vare_n \ \forall t\in[0,\tau_{{n+1}}]]\\
\leq &\sum_{n=0}^\infty M_{n,\kappa}c_\kappa\bigg(\tilde c \beta^{n/2}\sqrt{\frac{r_0'}{R}}\bigg)^{8/\kappa-1}\\
\leq &f_\kappa c_\kappa\bigg(\tilde c \sqrt{\frac{r_0'}{R}}\bigg)^{8/\kappa-1}\sum_{n=0}^\infty \beta^{n\frac{8-\kappa}{2\kappa}}\cdot \begin{cases}
(\frac{\alpha}{\beta})^{n\frac{\kappa-6-\rho}{2\kappa}}
(\frac{\alpha}{4})^{n\frac{(4-\kappa+\rho)\rho+2(\kappa-6)}{8\kappa}},
&\rho\in(-2,2],\\
(\frac{\alpha}{\beta})^{n\frac{\kappa-6-\rho}{2\kappa}},
&\rho\in[2,\infty).\\
\end{cases}
\end{align*}
We now set $0<\alpha<\beta<1$ in an appropriate way so that the series on the right hand side converges. 
When $\rho\in[2,\infty)$ we first fix the ratio $\alpha/\beta\in(0,1)$ arbitrarily and set $\beta\leq (\frac{\alpha}{\beta})^{1+\frac{\rho}{6}}$. Then
$$\sum_{n=0}^\infty \beta^{n\frac{8-\kappa}{2\kappa}}\Big(\frac{\alpha}{\beta}\Big)^{n\frac{\kappa-6-\rho}{2\kappa}}\leq \sum_{n=0}^\infty \Big(\frac{\alpha}{\beta}\Big)^{n\frac{8-\kappa}{2\kappa}(1+\frac{\rho}{6})}\Big(\frac{\alpha}{\beta}\Big)^{n\frac{\kappa-6-\rho}{2\kappa}}\leq \sum_{n=0}^\infty \Big(\frac{\alpha}{\beta}\Big)^{\frac{n}{\kappa}}<\infty.$$
For $\rho\in(-2,2]$ we again fix $\alpha/\beta\in(0,1)$, but now set $\beta\leq 4^{-4}(\frac{\alpha}{\beta})^{5}.$ Then, for $\kappa<2$
$$\sum_{n=0}^\infty \beta^{n\frac{8-\kappa}{2\kappa}}\Big(\frac{\alpha}{\beta}\Big)^{n\frac{\kappa-6-\rho}{2\kappa}}\Big(\frac{\alpha}{4}\Big)^{n\frac{(4-\kappa+\rho)\rho+2(\kappa-6)}{8\kappa}}
\leq \sum_{n=0}^\infty \Bigg(
\beta^{8}\Big(\frac{\alpha}{\beta}\Big)^{-36}\Big(\frac{1}{4}\Big)^{-32}\Bigg)^{\frac{n}{8\kappa}}\leq \sum_{n=0}^\infty \Big(\frac{\alpha}{\beta}\Big)^{\frac{n}{2\kappa}}<\infty.$$
This shows that (\ref{eq:termtwo}) is bounded above by $ \hat C(\rho) +4\log \frac{r_0'}{R},$
for a constant $\hat C(\rho)$, and by choosing the ratio $r_0'/R$ sufficiently small, the upper bound can be made smaller than $-M$. This finishes the proof of (b). The proof of (b') is almost identical.  
\end{proof}

\bibliographystyle{acm}
\bibliography{ref}

\begin{thebibliography}{10}

\bibitem{AHP24}
{\sc Abuzaid, O., Olsiewski~Healey, V., and Peltola, E.}
\newblock {Large deviations of Dyson Brownian motion on the circle and
  multiradial SLE$_{0+}$}, 2024.

\bibitem{AP24}
{\sc Abuzaid, O., and Peltola, E.}
\newblock {Large Deviations of Radial SLE$_{0+}$}.
\newblock {\em In preparation\/}.

\bibitem{ABKM20}
{\sc Alberts, T., Byun, S.-S., Kang, N.-G., and Makarov, N.}
\newblock {Pole dynamics and an integral of motion for multiple SLE$(0)$}.
\newblock {\em arXiv preprint arXiv:2011.05714\/} (2020).

\bibitem{AKR20}
{\sc Aldana, C.~L., Kirsten, K., and Rowlett, J.}
\newblock {Polyakov formulas for conical singularities in two dimensions}.
\newblock {\em arXiv preprint arXiv:2010.02776\/} (2020).

\bibitem{A83}
{\sc Alvarez, O.}
\newblock {Theory of strings with boundaries: fluctuations, topology and
  quantum geometry}.
\newblock {\em Nuclear Physics B 216}, 1 (1983), 125--184.

\bibitem{B19}
{\sc Bishop, C.}
\newblock {Weil-Petersson curves, $\beta$-numbers, and minimal surfaces}.
\newblock {\em Preprint\/} (2019).

\bibitem{BBVW23}
{\sc Bridgeman, M., Bromberg, K., Vargas-Pallete, F., and Wang, Y.}
\newblock {Universal Liouville action as a renormalized volume and its gradient
  flow}.
\newblock {\em arXiv preprint arXiv:2311.18767\/} (2023).

\bibitem{B23}
{\sc Brolin, A.}
\newblock {M{\"o}bius and Loewner energy on curves with corners}, 2023.

\bibitem{DZ10}
{\sc Dembo, A., and Zeitouni, O.}
\newblock {\em {Large Deviations Techniques and Applications}}, 2nd ed.
  2010.~ed.
\newblock Stochastic Modelling and Applied Probability, 38. Springer Berlin
  Heidelberg, Berlin, Heidelberg, 2010.

\bibitem{D05}
{\sc Dubédat, J.}
\newblock {SLE$(\kappa,\rho)$ martingales and duality}.
\newblock {\em The Annals of Probability 33}, 1 (Jan. 2005).

\bibitem{D07}
{\sc Dubédat, J.}
\newblock {Commutation relations for Schramm‐Loewner evolutions}.
\newblock {\em Communications on Pure and Applied Mathematics 60}, 12 (May
  2007), 1792–1847.

\bibitem{FL15}
{\sc Field, L., and Lawler, G.}
\newblock {Escape probability and transience for SLE}.
\newblock {\em Electronic Journal of Probability 20}, none (2015), 1 -- 14.

\bibitem{FS17}
{\sc Friz, P.~K., and Shekhar, A.}
\newblock {On the existence of SLE trace: finite energy drivers and
  non-constant $\kappa$}.
\newblock {\em Probability Theory and Related Fields 169}, 1 (2017), 353--376.

\bibitem{GM05}
{\sc Garnett, J.~B., and Marshall, D.~E.}
\newblock {\em {Harmonic measure}}.
\newblock Cambridge University Press, 2005.

\bibitem{G23}
{\sc Guskov, V.}
\newblock {A large deviation principle for the Schramm--Loewner evolution in
  the uniform topology}.
\newblock {\em Annales Fennici Mathematici 48}, 1 (2023), 389--410.

\bibitem{JV23}
{\sc Johansson, K., and Viklund, F.}
\newblock {Coulomb gas and the Grunsky operator on a Jordan domain with
  corners}.
\newblock {\em arXiv preprint arXiv:2309.00308\/} (2023).

\bibitem{KNK04}
{\sc Kager, W., Nienhuis, B., and Kadanoff, L.~P.}
\newblock {Exact solutions for Loewner evolutions}.
\newblock {\em Journal of statistical physics 115\/} (2004), 805--822.

\bibitem{K17}
{\sc Kemppainen, A.}
\newblock {\em {Schramm–Loewner Evolution}}, 1st ed. 2017.~ed.
\newblock SpringerBriefs in Mathematical Physics, 24. Springer International
  Publishing, Cham, 2017.

\bibitem{K24}
{\sc Krusell, E.}
\newblock {Polyakov-Alvarez formula for curvilinear polygonal domains with
  slits}.
\newblock {\em Preprint\/} (2024).
\newblock \url{https://people.kth.se/~ekrusell/PolyakovAlvarezSlits.pdf}.

\bibitem{LSW03}
{\sc Lawler, G., Schramm, O., and Werner, W.}
\newblock {Conformal restriction: the chordal case}.
\newblock {\em Journal of the American Mathematical Society 16}, 4 (2003),
  917--955.

\bibitem{L05}
{\sc Lind, J.}
\newblock {A sharp condition for the Loewner equation to generate slits}.
\newblock {\em Annales Fennici Mathematici 30}, 1 (Feb. 2005), 143–158.

\bibitem{LMR10}
{\sc Lind, J., Marshall, D.~E., and Rohde, S.}
\newblock {Collisions and spirals of Loewner traces}.
\newblock {\em Duke Mathematical Journal 154}, 3 (2010), 527 -- 573.

\bibitem{LR16}
{\sc Lind, J.~R., and Rohde, S.}
\newblock {Loewner curvature}.
\newblock {\em Mathematische annalen 364\/} (2016), 1517--1534.

\bibitem{M23}
{\sc Mesikepp, T.}
\newblock {A deterministic approach to Loewner-energy minimizers}.
\newblock {\em Mathematische Zeitschrift 305}, 4 (2023), 59.

\bibitem{MW23}
{\sc Michelat, A., and Wang, Y.}
\newblock {The Loewner Energy via the Renormalised Energy of Moving Frames}.
\newblock {\em Archive for Rational Mechanics and Analysis 248}, 2 (2024), 15.

\bibitem{MS16}
{\sc Miller, J., and Sheffield, S.}
\newblock {Imaginary geometry I: interacting SLEs}.
\newblock {\em Probability Theory and Related Fields 164\/} (2016), 553--705.

\bibitem{MS16c}
{\sc Miller, J., and Sheffield, S.}
\newblock {Imaginary geometry II: Reversibility of
  SLE$_{\kappa}(\rho_{1};\rho_{2})$ for $\kappa\in(0,4)$}.
\newblock {\em The Annals of Probability 44}, 3 (2016), 1647 -- 1722.

\bibitem{MS16d}
{\sc Miller, J., and Sheffield, S.}
\newblock {Imaginary geometry III: reversibility of SLE$_\kappa$ for
  $\kappa\in(4, 8)$}.
\newblock {\em Annals of Mathematics\/} (2016), 455--486.

\bibitem{MS16b}
{\sc Miller, J., and Sheffield, S.}
\newblock {Quantum Loewner evolution}.
\newblock {\em Duke Mathematical Journal 165}, 17 (Nov. 2016).

\bibitem{MS17}
{\sc Miller, J., and Sheffield, S.}
\newblock {Imaginary geometry IV: interior rays, whole-plane reversibility, and
  space-filling trees}.
\newblock {\em Probability Theory and Related Fields 169\/} (2017), 729--869.

\bibitem{NRS19}
{\sc Nursultanov, M., Rowlett, J., and Sher, D.~A.}
\newblock {The heat kernel on curvilinear polygonal domains in surfaces}.
\newblock {\em arXiv preprint arXiv:1905.00259\/} (2019).

\bibitem{OPS88}
{\sc Osgood, B., Phillips, R., and Sarnak, P.}
\newblock {Extremals of determinants of Laplacians}.
\newblock {\em Journal of Functional Analysis 80}, 1 (1988), 148--211.

\bibitem{PW21}
{\sc Peltola, E., and Wang, Y.}
\newblock {Large deviations of multichordal {SLE}$_{0+}$, real rational
  functions, and zeta-regularized determinants of Laplacians}.
\newblock {\em Journal of the European Mathematical Society 26}, 2 (2023),
  469--535.

\bibitem{P81}
{\sc Polyakov, A.~M.}
\newblock {Quantum geometry of bosonic strings}.
\newblock {\em Physics Letters B 103}, 3 (1981), 207--210.

\bibitem{RW21}
{\sc Rohde, S., and Wang, Y.}
\newblock {The Loewner energy of loops and regularity of driving functions}.
\newblock {\em International mathematics research notices 2021}, 10 (2021),
  7433--7469.

\bibitem{S00}
{\sc Schramm, O.}
\newblock {Scaling limits of loop-erased random walks and uniform spanning
  trees}.
\newblock {\em Israel journal of mathematics 118}, 1 (2000), 221--288.

\bibitem{S11a}
{\sc Schramm, O.}
\newblock {Conformally invariant scaling limits: an overview and a collection
  of problems}.
\newblock In {\em Proceedings of the International Congress of Mathematicians,
  August 22-30, 2006, Madrid\/} (2006), vol.~1, European Mathematical Society,
  pp.~513--543.

\bibitem{SW05}
{\sc Schramm, O., and Wilson, D.~B.}
\newblock {SLE coordinate changes.}
\newblock {\em The New York Journal of Mathematics [electronic only] 11\/}
  (2005), 659--669.

\bibitem{S22}
{\sc Sheffield, S.}
\newblock {What is a random surface?}
\newblock In {\em Proceedings of the International Congress of
  Mathematicians\/} (Virtual, 2022), vol.~2, European Mathematical Society,
  pp.~1202 -- 1258.

\bibitem{S11b}
{\sc Smirnov, S.}
\newblock {Towards conformal invariance of 2D lattice models}.
\newblock In {\em Proceedings of the International Congress of Mathematicians,
  August 22-30, 2006, Madrid\/} (2006), vol.~2, European Mathematical Society,
  pp.~1421--1451.

\bibitem{SW24}
{\sc Sung, J., and Wang, Y.}
\newblock {Quasiconformal deformation of the chordal Loewner driving function
  and first variation of the Loewner energy}.
\newblock {\em Mathematische Annalen\/} (2024), 1--24.

\bibitem{VW20}
{\sc Viklund, F., and Wang, Y.}
\newblock {Interplay between Loewner and Dirichlet energies via conformal
  welding and flow-lines}.
\newblock {\em Geometric and Functional Analysis 30\/} (2020), 289--321.

\bibitem{VW24}
{\sc Viklund, F., and Wang, Y.}
\newblock {The Loewner–Kufarev energy and foliations by Weil–Petersson
  quasicircles}.
\newblock {\em Proceedings of the London Mathematical Society 128}, 2 (Feb.
  2024).

\bibitem{W19b}
{\sc Wang, Y.}
\newblock {Equivalent descriptions of the Loewner energy}.
\newblock {\em Inventiones mathematicae 218}, 2 (2019), 573--621.

\bibitem{W19a}
{\sc Wang, Y.}
\newblock {The energy of a deterministic Loewner chain: Reversibility and
  interpretation via {SLE}$_{0+}$}.
\newblock {\em Journal of the European Mathematical Society 21}, 7 (2019),
  1915--1941.

\bibitem{W22}
{\sc Wang, Y.}
\newblock {Large deviations of Schramm-Loewner evolutions: A survey}.
\newblock {\em Probability Surveys 19\/} (2022), 351--403.

\bibitem{W04}
{\sc Werner, W.}
\newblock {Girsanov’s transformation for SLE$(\kappa,\rho)$ processes,
  intersection exponents and hiding exponents}.
\newblock In {\em Annales de la Facult{\'e} des sciences de Toulouse:
  Math{\'e}matiques\/} (2004), vol.~13, pp.~121--147.

\bibitem{W65}
{\sc Wigley, N.}
\newblock {Development of the mapping function at a corner}.
\newblock {\em Pacific Journal of Mathematics 15}, 4 (1965), 1435--1461.

\bibitem{W14}
{\sc Wong, C.}
\newblock {Smoothness of Loewner slits}.
\newblock {\em Transactions of the American Mathematical Society 366}, 3
  (2014), 1475--1496.

\bibitem{Z22}
{\sc Zhan, D.}
\newblock {{SLE}$_\kappa(\rho)$ bubble measures}.
\newblock {\em arXiv preprint arXiv:2206.04481\/} (2022).

\end{thebibliography}
\end{document}